\documentclass{amsart}
\usepackage{amsmath, amssymb, amscd, latexsym, tikz, tikz-cd}
\usetikzlibrary{matrix, decorations, arrows}
\theoremstyle{plain}

\newtheorem{thm}{Theorem}[section]
\newtheorem{backthm}{Background Theorem}[section]
\newtheorem{lem}{Lemma}[section]
\newtheorem{corollary}{Corollary}[section]

\theoremstyle{definition}
\newtheorem{defn}{Definition}[section]

\theoremstyle{remark}

\newcommand{\ddx}{\frac{\partial}{\partial x}}
\newcommand{\ddy}{\frac{\partial}{\partial y}}
\newcommand{\ddz}{\frac{\partial}{\partial z}}
\newcommand{\dda}{\frac{\partial}{\partial x_1}}
\newcommand{\ccc}{\textbf{\textsf{C}}}

\numberwithin{equation}{section}
\subjclass[2010]{Primary: 53D99; Secondary: 53Z05}
\title{Equivalence of Nondifferentiable Metrics}
\author{Alexander Golubev}
\address{New York Institute of Technology, College of Engineering and 
Computing Sciences, New York, New York}
\email{agolubev@nyit.edu}

\begin{document}

\begin{abstract}We study nondifferentiable metrics occuring
in general relativity via the method of equivalence of Cartan
adapted to the Courant algebroids. We derive  new local differential 
invariants naturally associated with the loci of nondifferentiability 
and rank deficiency of the metric. As an application, we utilize the 
newfangled invariants to resolve the problem of causality in the interior 
of the black holes that contain closed timelike geodesics. Also, a no-go
type theorem limits the evolution scenarios for gravitational collapse. 
\end{abstract}
\date{}

\maketitle
\tableofcontents

\section{Introduction}
\paragraph*{}
Our objective is to formulate the notion of invariants of nondifferentiable
(pseudo-)\\
metrics, and to adapt \'{E}lie Cartan's method of equivalence~\cite{Cartan}
(originally developed to elicit differential invariants of smooth structures) 
so as to assimilate various metric pathologies\footnote{We deliberately avoid 
the term `singularity' in this context. Its definition in general relativity 
is narrow and does not encompass all the phenomena treated herein.} into 
differential geometry. The need for such assimilation arises primarily in 
general relativity with black hole solutions of the Einstein vacuum equations 
(EVE). There, genuine information is encoded via nondifferentiable tensor 
coefficients, although several blow-up solutions of other nonlinear PDE 
merit further investigation.\\
\indent
A trademark of metric nondifferentiability is geodesic incompleteness. Its
presence in general relativity unambigously indicates that strong
gravitational fields are beyond the scope of classical theory. Even more
troublesome, the Kerr spacetime allows closed timelike geodesics - a fact
that flies in the face of causality.\\
\indent
So far, all attempts to exclude or domesticate the `wrong' metric 
configurations in a systematic way have failed. Most notably, the weak and 
strong cosmic censorship hypotheses due to Roger Penrose~\cite{Pen2} turned 
out to be of limited applicability (see~\cite{DL} for a detailed discussion 
as well as the recent results). Furthermore, the hyperbolic PDE stability 
of solutions of EVE on the function spaces of initial data on spacelike 
hypersurfaces has not been established in full generality.\\
\indent
In this paper we adopt a different viewpoint. Rather than trying to downplay
the undesirable pathologies by selecting generic and/or stable solutions of
EVE while retaining the desirable ones (the black hole event horizon, to name
one), we acknowledge their legitimacy and incorporate them into our framework.\\
\indent
Such a program necessitates building an apparatus to treat metric configurations
with or without loci of nondifferentiability uniformly. In particular, the method 
of equivalence must be made to work in this formalism.\\
\indent
Our approach brings about a new mathematical object - a preframe, that
generalizes the concept of coframe onto a Courant algebroid (see Section 2.1.1),
realized as the bundle $T^*M \oplus TM$ equipped with the Courant bracket.
In what follows, $(M, \omega)$ is a symplectic manifold. With this additional
structure, we can represent the Courant algebroid as a Lie bialgebroid, such
that the two complementary Dirac structures (maximally isotropic subbundles of 
$T^*M \oplus TM$ with respect to the positive Courant bilinear form and 
closed under the Courant bracket) are related via the bundle map
assembled from the symplectic 2-form and its attendant Poisson tensor.
Subsequently, we show that for this Courant algebroid there exists a unique
(up to a symplectomorphism) connection, which we designate the symplectic
connection, that lifts preframes injectively to the cotangent bundle manifold:
$$\ccc: \;T^*M \oplus TM \hookrightarrow TT^*M.$$
Fortuitously, the symplectic connection acts as a morphism of Lie
bialgebroids, the image bialgebroid being an isotropic integrable subbundle of 
$TT^*M \oplus T^*T^*M$ (see Section 2.1.3).\\
\indent
Via the canonical symplectomorphism of~\cite{AM}
$$\boldsymbol{\alpha}: TT^*M \longrightarrow T^*TM,$$
we obtain a full coframe on $T^*TM$, and thus pave the way for the 
application of Cartan's equivalence method. However, this is a principally
new variety of equivalence in that the resulting mapping does not descend
from the level of the tangent bundle manifold down to $(M, \omega)$ in 
general. \\
\indent
It does in the case of smooth metrics, wherein our preframes turn out to be 
sections of a Dirac subbundle, and locally equivalent\footnote{This version 
of equivalence was invented by Hitchin~\cite{Hitch}. He constructed the 
moduli space of generalized Calabi-Yau manifolds by quotienting out the 
action of exact 2-forms (B-fields) as well as the action of conventional 
biholomorphic maps.} to the standard coframes viewed as sections of the 
trivial Dirac subbundle $T^*M$. Explicitly, the Hitchin map is the 
projection $${(\text{Id}}^* \oplus 0): T^*M\oplus TM \longrightarrow T^*M.$$
\noindent
No loss of data occurs since the preframe is integrable.\\
\indent
Obversely, preframes not closed under the Courant bracket are not sections
of some Dirac subbundle and are considered singular. Their lifts to $T^*TM$
are best characterized in terms of the Lie algebra of isometries within
Diff($T^*M$). They are certain conjugacy classes of Diff($T^*M$)/Diff($M$).
No smooth metric configuration would yield a preframe of this type.\\ 
\indent
Having introduced preframes at a conceptual level, a more detailed account 
of how equivalence was established is now in order. Thus
for a nondifferentiable metric tensor $[g_{ij}]$ on $M$, $\dim M = 2k$, 
we introduce $2k^2 + k$ local auxiliary scalar fields satisfying the 
Laplace-Beltrami equation on the metric background. They encapsulate the 
pathologies associated with the original metric, and allow for it to be 
reverse-engineered. To be usable at the coframe level, they (locally being 
at most twice differentiable) would have to be smoothed. An application of 
the Nash-Gromov deep smoothing operators~\cite{N, Grom} outputs multiple 
smooth functions. Effectively, this approach trades off nondifferentiability 
for nonuniqueness.\\
\indent
Our way of coming to grips with the overflowing scalar fields involves an
infinite-dimensional completely integrable Hamiltonian system that allows
to organize them into parametric families united by Hamiltonian flows. 
Specifically, we take up the isospectral class `manifold' of the Hill operator
with periodic potentials, following a seminal paper of McKean and Trubowitz
\cite{McT}. The auxiliary scalar fields are represented by products of an
eigenfunction of some appropriate Hill operator with a potential periodic 
with respect to a local parameter transversal to the loci of nondifferentiability 
and/or rank deficiency, which we dub the blow-up parameter, and a nonvanishing 
function. The nondifferentiability property becomes quantifiable by comparison 
with the spectrum of the harmonic oscillator (the Hill operator with a constant 
potential). Smooth metrics produce constant Hill operator eigenfunctions, 
whereas those with nondifferentiability loci require oscillating ones.\\
\indent
The oscillating eigenfunctions are incorporated into preframes, rendering 
those preframes singular. Their attendant nonholonomic frames on $TT^*M$ 
are gotten by lifting the preframe and its complement via the symplectic 
connection. On the cotangent bundle manifold, we use the adjoined variables
(impulse variables in the canonical phase space) to parameterize nonunique 
smoothed auxilliary scalar fields in a prolongation. Hence the equivalence 
would encompass the totality of smoothed solutions of the Laplace-Beltrami 
equation.\\
\indent
The maximal structure group for our equivalence problem has
to preserve the preframes (for rotations among the base variables 
potentially interfere with the individual auxilliary scalar fields), as
well as to respect the (action of the) symplectic connection $\ccc$
and the canonical symplectomorphism $\boldsymbol{\alpha}$.
Therefore only a maximal torus group $\bigoplus_{i=1}^{2k} SO_i(2, 
\mathbb{R})$ survives. Cartan's group reduction/frame normalization 
technique can only be implemented with the auxilliary scalar fields 
that have locally nonvanishing first and second derivatives, 
known as free maps - those belonging to the set of solutions of the
freedom partial differential relation (\cite{Grom}, Chapter 1). 
To ensure we deal with free auxilliary scalar fields, we apply the 
homotopy principle to deform our scalar fields into free ones. The 
resulting $4k$-coframes exhibit some specific intrinsic torsion, 
which serves as a repository of data pertaining to metric loci of 
nondifferentiability.\\
\indent
The upshot of this revamping is, new invariants emerge. They encode the
minutia of metric blow-ups via holomorphic structures of the Hill surfaces
(hyperelliptic Riemann surfaces of infinite genus)~\cite{McT2}. Using recent 
advances in hyperelliptic function theory~\cite{FKT}, particularly the 
Torelli theorem for certain Riemann surfaces of infinite genus, these 
invariants can be conveniently presented as the Riemann period matrices 
parameterized by the base manifold variables. Their link to nonintegrable
preframes involves an inversion formula for theta functions found by Its 
and Matveev in the finite-genus case~\cite{IM}, and generalized by McKean 
and Trubovitz to cover the Hill surfaces of infinite genus~\cite{McT2}.\\
\indent
After this protracted discussion of technical matters, we at last take up
the applications of our differential invariants to black hole solutions of
EVE, collected in Chapter 5. First, we compute the global singular
preframe of the Kerr spacetime, and delineate one open subset (we call this
subset representative) of it that contains parts of all the geometrically 
significant loci of nondifferentiabily, as well as those of rank deficiency. 
They are (subsets of) the event horizon, the Cauchy horizon, and the locus 
of curvature blow-up. That is a synopsis of Section 5.1.\\
\indent
Next, we offer a resolution of the nonuniqueness of $C^0$-extensions of 
solutions of EVE past the Cauchy horizon. The strong cosmic censorship
hypothesis (in a precise formulation) has been disproved by Dafernos and
Luk~\cite{DL}. We find a way to select a unique extension based on the 
values of the germaine Riemann period matrices. Thus the (internal) locus
of curvature blow-up turns out to be determined by the behavior of metric
coefficients at the (external) locus of rank deficiency/nondifferentiability
controlled by the future development imposed by EVE. However, the apparent
breakdown of causal structure still persists for dynamics on that background. 
To save the situation, we draw distinction between ordinary differentiable
configurations and their attendant identically zero Riemann period matrices,
and pathological ones. We speculate that preservation of the Riemann period
matrix is inextricably linked with mass distribution, to the extent it would
eliminate at least some trajectories should a mass cross the event horizon.
But the door still remains half-open for an intricate mass trajectory
to sneak past the horizon without disturbing the invariants. This hesitancy
is informed by the effects of gravitational waves on black hole solutions.
They disturb the event horizons temporarily. After the wavefront passes,
the black holes return to their original state unscathed. One would expect
a similar scenario to play out if the mass moves in a special way.\\
\indent
Section 5.4 deals with one particular implication of the invariants being
there to label black holes - the quantum effect of black hole evaporation
due to emission of thermal radiation~\cite{Hawk, HaHa}. Following the 
prescriptions of quantum field theory, we introduce absorption and emission
operators acting on the eigenvalues of the Hill operator whose periods
(parameterized by the manifold variables) bridge the representative set
from the locus of nondifferentiability to the regular metric. Evaporation 
would effectively shrink the locus of nondifferentiability, and modify  
the discriminant so that the odd- and even-indexed eigenvalues would move
closer to each other. Conversely, absorption operators would elongate the 
intervals of instability of the Hill operator. Hence absoption/emission
operators have to be Hermitian, and, being defined on a  Hilbert space
of quantum states, have compact kernels with no further restrictions. 
The resulting configurations at the Riemann surface level would 
no longer encode viable geometric structures, let alone black hole 
solutions of EVE, but the chronological products of (combinations of) 
those operators may furnish a representation of the information imprinted 
on the event horizon and subsequently lost due to evaporation. We 
contrive a necessary condition for the information to be retrievable 
in terms of formal properties of absorption/emission operators. That is, 
their brackets under composition must remain noncommutative at all values 
of some affine time parameter. Thus our quantum operator algebra must be 
defined over the field of quaternions. Otherwise, the existence of 
invariants would not forestall the loss of information predicated on 
emission of thermal radiation.\\
\indent
In Section 5.5, we state and prove a No-Go theorem that limits the 
gravitational collapse of axisymmetric rotating distributed masses
down to the Kerr spacetime. Once again, the time development set by
EVE is reexpressed in terms of time-dependent discriminant evaluated
at a fixed eigenvalue of the Hill operator. Its second partial derivative
with respect to global time coordinate undergoes some anharmonic
oscillations prior to coming to a stationary state. These oscillations
are governed by the universal oscillator ODE with a time-dependent
damping factor. Our equation establishes a link between the Kerr
spacetime and the initial mass distribution informing damping.\\
\indent
Further exploratory steps are taken in Section 5.6. To account 
for multiple black holes, moving
in space-time relative to each other, possibly coalescing, we need
a more robust formalism. One blow-up parameter would not be enough,
even locally. A logical way to overcome this limitation is to
introduce multiple globally independent parameters. Hence, we argue 
that the spectral theory of two-dimensional Schr\H{o}dinger operator, 
$$[\frac{{\partial}^2}{\partial x_1^2} + \frac{{\partial}^2}
{\partial x_2^2} + Q(x_1,x_2)],$$ with doubly periodic potential,
investigated in~\cite{K}, turns out to be a substitute for 
the Hill operator. Complexified Fermi curves (\cite{FKT}, Section 16)
take the place of the Hill surfaces. Torelli theorem remains, but the
Its-Matveev formula no longer applies. One implication of a 
multiparametric theory is that the time coordinate would exist only 
locally in asymptotically flat space-time.\\

\vspace{0.2 in}
\textbf{Acknowledgements.} We would like to express our gratitude to Vladimir 
Matveev and Igor Krichever for generously sharing their insights into spectral 
theory and integrable systems, and Andrei Todorov for stimulating discussions.

\section{Preframes}

\subsection{Preliminaries}
\paragraph*{}
The point of departure here is an assembly of topics from
differential geometry in the category of smooth manifolds and
diffeomorphisms. 
\subsubsection{Courant algebroids}
The notion of Courant algebroid, conceived as a natural outgrowth of symplectic
or Poisson structures, and modeled after Lie algebroids, was first introduced by 
Theodore Courant~\cite{Cour}. Below we provide a somewhat systematic overview, 
tailored to our needs, following Liu, Weinstein, and Xu~\cite{LWX}, and the 
original material referenced therein.

\begin{defn}
A \textit{Lie algebroid} on a manifold $M$ is a vector bundle $A \longrightarrow M$,
equipped with a vector bundle map, $\hat{a}: A \longrightarrow TM$, over $M$,
called the \textit{anchor} of $A$, and a bracket $[\cdot, \cdot]: \; \Gamma (A)
\times \Gamma (A) \longrightarrow \Gamma (A)$, which is bilinear and
antisymmetric, satisfies the Jacobi identity, and is such that
\begin{equation}
[X, uY] = u[X, Y] + (\hat{a}(X)u)Y,\;\;\forall X, Y \in \Gamma (A),\;u\in C^{\infty}(M),
\end{equation}
\begin{equation}
\hat{a}([X, Y]) = [\hat{a}(X), \hat{a}(Y)],\;\;\forall X, Y \in \Gamma (A).
\end{equation}
\end{defn}
\noindent
The Lie algebroid $A$ is \textit{transitive} if $\hat{a}$ is fiberwise surjective,
\textit{regular} is $\hat{a}$ is of locally constant rank, and 
\textit{totally intransitive} if $\hat{a} =0$. The manifold $M$ is the
\textit{base} of $A$.
\begin{defn}
A \textit{base-preserving morphism of Lie algebroids}, referred to as simply
morphism, is a vector bundle map $\varphi : (A_1, {\text{pr}}_1, M) \rightarrow
(A_2, {\text{pr}}_2, M)$ such that ${\hat{a}}_2 \cdot \varphi ={\hat{a}}_1$,
and $$\varphi [X, Y] = [\varphi(X), \varphi(Y)].\;\;\;\forall X,Y \in \Gamma(A_1).$$
\end{defn} 
\indent
The anchor of a Lie algebroid encodes its geometric properties. If the algebroid
is transitive, then right inverses to the anchor are connections. If the algebroid
is regular, then the image of the anchor defines a foliation of the base manifold,
the \textit{characteristic foliation}, and over each leaf of that foliation, the
Lie algebroid is transitive.
\begin{defn}
A \textit{Lie bialgebroid} on a manifold $M$ is a dual pair of Lie algebroids 
$(A, A^*)$ such that the coboundary operator $d_*$ on $\Gamma ({\wedge}^* A)$ 
satisfies
$$d_*[X,Y] =[d_*X,Y]+[X,d_*Y],\;\;\;\forall X, Y \in \Gamma(A).$$
\end{defn}
\noindent
On classical Lie bialgebras we have the compatibility condition between
$(\mathfrak{g}, [\cdot, \cdot])$ and  $({\mathfrak{g}}^*, [\cdot, \cdot]^*)$ 
formulated in terms of a cocycle. This definition is just an extension. \\
\indent
For a Poisson manifold $(M, \pi)$, there is the \textit{standard Lie 
bialgebroid} $(TM, T^*M)$ with the coboundary operator $d_{\pi}= [\pi, \cdot]$. 
\begin{defn}\label{D: Courant algebroid}
A \textit{Courant algebroid} is a vector bundle $E\; \longrightarrow M$
equipped with a nondegenerate symmetric bilinear form $(\cdot, \cdot)$ on
the bundle, a skew-symmetric bracket $[\cdot, \cdot]$ on $\Gamma(E)$, and
a bundle map $ \rho :E\; \longrightarrow TM$ such that the following
properties are satisfied:\\

\begin{enumerate}
\item $[e_1, [e_2, e_3]] + [e_3, [e_1, e_2]] + [e_2, [e_3, e_1]]
= \mathcal{D}T(e_1, e_2, e_3)\;\;\forall e_1, e_2, e_3 \in \Gamma(E);$\\
\item $\rho[e_1, e_2] = [\rho e_1, \rho e_2]\;\;\forall e_1, e_2 \in \Gamma(E);$\\
\item $ [e_1,fe_2] =f[e_1, e_2] + (\rho (e_1))e_2 - (e_1, e_2) \mathcal{D}f\;\;
\forall e_1, e_2 \in \Gamma(E),\;\; \forall f \in C^{\infty}(M);$\\
\item ${\rho}_{\circ} \mathcal{D} =0,$ i.e. $(\mathcal{D}f, \mathcal{D}g) =0\;\;
\forall f, g \in C^{\infty}(M);$\\
\item $\rho(e)(h_1, h_2) =([e, h_1] + \mathcal{D}(e, h_1), h_2) +
(h_1, [e, h_2] + \mathcal{D}(e, h_2)),\\ 
\forall e, h_1, h_2 \in \Gamma(E)$ \\
where $T(e_1, e_2, e_3)\in C^{\infty}(M)$ is defined by\\
$$ T(e_1, e_2, e_3) \overset{\textnormal{def}}{=} \frac{1}{3}(
(e_1, [e_2, e_3]) + (e_3, [e_1, e_2]) + (e_2, [e_3, e_1])),$$\\
and $\mathcal{D}: C^{\infty}(M) \longrightarrow \Gamma(E)$ is the map defined by\\
$$(\mathcal{D}f, e) \overset{\textnormal{def}}{=}\frac{1}{2}\rho (e) f.$$
\end{enumerate}
\end{defn}
\begin{defn}
Let $E$ be a Courant algebroid. A subbundle $L \subset E$ is called \textit{isotropic} 
if it is isotropic under the symmetric bilinear form $(\cdot, \cdot)$.
It is called \textit{integrable} if $\Gamma (L)$ is closed under the bracket
$[\cdot, \cdot]$. A \textit{Dirac structure}, or \textit{Dirac subbundle} is a 
subbundle $L$ which is maximally isotropic and integrable.
\end{defn}
\indent
Any Dirac structure is trivially a Lie algebroid with anchor $\rho |_L$. 
Conversely, suppose $A$ and $A^*$ are both Lie algebroids over the
base manifold $M$, with anchors $\hat{a}$ and ${\hat{a}}_*$ respectively. 
Let $E$ denote
their vector bundle direct sum $E = A \oplus A^*$ (defined as the pullback
of $A {\times}_M A^* $). On $E$ there exist two natural nondegenerate 
bilinear forms, one symmetric, and the other antisymmetric, which are 
defined as follows:\\
\begin{equation}\label{E: bilinear forms}
(X_1 + {\xi}_1, X_2 + {\xi}_2)_{\pm} \overset{\textnormal{def}}{=}
\frac{1}{2}(\langle {\xi}_1, X_2 \rangle \pm \langle {\xi}_2, X_1 \rangle).
\end{equation} 
\indent
The additional structure begets some refined additional structure:
\begin{defn}
A Dirac structure $L$ such that the antisymmetric bilinear form defined 
in~\eqref{E: bilinear forms} satisfies $(X_1 + {\xi}_1, X_2 + {\xi}_2)_- =0$,
$\forall (X_{1,2} + {\xi}_{1,2}) \in \Gamma (L)$ is called a \textit{null Dirac structure}.
\end{defn}
\indent
On $\Gamma (E)$, we introduce a bracket operation by\\
\begin{alignat}{2}\label{E: generalized bracket}
[e_1, e_2] =&([X_1, X_2] + {\mathcal{L}}_{{\xi}_1} X_2 -
{\mathcal{L}}_{{\xi}_2} X_1 - d_*(e_1, e_2)_-)\\ +
&([{\xi}_1, {\xi}_2] + {\mathcal{L}}_{X_1} {\xi}_2 -
{\mathcal{L}}_{X_2} {\xi}_1 + d(e_1, e_2)_-),\notag
\end{alignat}
\noindent
where $e_{1,2} = X_{1,2} + {\xi}_{1,2}$, $d:C^{\infty}(M) \longrightarrow \Gamma (A^*)$,
$d_*: C^{\infty}(M) \longrightarrow \Gamma (A)$ are the differentials.\\
\indent
Now we let $\rho : E \longrightarrow M$ be the bundle map defined
by $\rho = \hat{a} + {\hat{a}}_*$. That is\\
\begin{equation}
\rho (X + \xi) \overset{\textnormal{def}}{=} \hat{a}(x) + {\hat{a}}_*(\xi),\;\;
\forall X \in \Gamma (A),\; \forall \xi \in \Gamma (A^*).
\end{equation}
\noindent
Coincidentally, we have $\mathcal{D}$ of Definition~\ref{D: Courant algebroid}
expressible in terms of the exterior differential operators introduced 
in~\eqref{E: generalized bracket}: $\mathcal{D} = d_* + d$.\\
\indent
Finally, we specialize our choice of the component algebroids. Namely,
$A = TM$, $A^* = T^*M$, the latter equipped with the trivial bracket.
Then~\eqref{E: generalized bracket} takes the form:\\
\begin{equation}\label{E: my bracket}
[e_1, e_2] = [X_1, X_2] + {\mathcal{L}}_{X_1} {\xi}_2 -
{\mathcal{L}}_{X_2} {\xi}_1 + d(e_1, e_2)_-,
\end{equation}
\noindent
This is the Courant bracket that we all know and love.\\
\indent
There are multiple Dirac subbundles associated with the Courant
algebroid $TM \oplus T^{*}M$. In particular, we may have a pair
of transversal ones. Then $ L_1 \oplus L_2 = TM \oplus T^{*}M$.
Such a pair constitutes a Lie bialgebroid (\cite{LWX}, Theorem 2.6):
\begin{backthm}
In a Courant algebroid $(E, \rho, [\cdot, \cdot], (\cdot, \cdot))$
with a pair of transversal Dirac structures, $L_1,\;L_2$, $L_1 \oplus L_2 = E$,
there exists an attendant Lie bialgebroid $(L_1, L_2)$, where $L_2 =(L_1)^*$
under the dual pairing given by $2(\cdot, \cdot)$.
\end{backthm}
\indent
Typically, Dirac structures are obtained by interpolating the tangent and 
cotangent bundles via graphs: $X \mapsto X + i_X\beta$, $\beta \in 
\Gamma ({\Lambda}^2T^*M),\; d\beta = 0$. In (\cite{LWX}, Section 6)
the bundle maps associated with closed 2-forms are designated `strong
Hamiltonian operators'. However, as Liu et al. point out, there are Dirac
structures that are not such graphs. Those are generated by other bundle
maps, satisfying a nonlinear differential equation. \\ 
\indent
From this point on, our base manifold is symplectic, $(M, \omega)$, 
$\dim M = 2k$. Thus there is a bundle isomorphism induced by the
symplectic form, and its inverse induced by the (nondegenerate)
Poisson structure:\\
\begin{equation}
TM \xrightarrow{{\omega}^{\#}} T^*M,\;\;
T^*M \xrightarrow{{\pi}^{\#}} TM,
\end{equation}
implicitly defined via
\begin{equation}
\langle {\omega}^{\#}(X), Y \rangle = \omega (X,Y),\;\;\;
\langle \varphi, {\pi}^{\#}(\beta) \rangle = \pi (\varphi, \beta).
\end{equation}
On $TM \oplus T^{*}M$ they combine into a discrete automorphism
of the Courant algebroid. Our choice is ${\omega}^{\#} - {\pi}^{\#}$.
The formal definition involves projections and the fact that 
the Whitney sum of two vector bundles over the same base manifold is
defined as the pullback bundle of the diagonal embedding
$M \longrightarrow M \times M$ establishing the isomorphism
$TM \oplus T^{*}M \cong TM {\times}_M T^*M$:
\begin{equation*}\begin{CD}
TM @<{(Id \oplus 0)}<< TM \oplus T^*M @>{(0 \oplus Id^*)}>> T^*M.
\end{CD}\end{equation*}
\begin{equation}
({\omega}^{\#} - {\pi}^{\#})(\xi +X)\overset{\textnormal{def}}{=}
({\omega}^{\#}\circ (Id \oplus 0) -{\pi}^{\#}\circ (0 \oplus Id^*))(\xi +X)
= {\omega}^{\#}(X) - {\pi}^{\#}(\xi)
\end{equation}
for smooth sections $X\in \Gamma(TM),\;\;\xi \in \Gamma(T^*M).$
This bundle automorphism satisfies
$$({\omega}^{\#} - {\pi}^{\#})\circ({\omega}^{\#} - {\pi}^{\#}) =-(Id^*\oplus Id).$$
\begin{defn}\label{D: preframe}
A \textit{preframe} on a manifold $(M,\omega)$, $\dim M =2k$, is a 
2k-tuple $\{X_i + {\xi}_i\}, \;\;X_i + {\xi}_i \in \Gamma (TM \oplus T^*M)$
such that 
\begin{equation}
\text{span}\{\{X_i + {\xi}_i\},({\omega}^{\#} - {\pi}^{\#})\{X_i + {\xi}_i\}\}
= TM \oplus T^*M, \;\;\text{and} 
\end{equation}
\begin{equation}
\{X_i + {\xi}_i\} \cap ({\omega}^{\#} - {\pi}^{\#})\{X_i + {\xi}_i\} =0.
\end{equation}
\end{defn}
\noindent
Notably, there is no differentiability provision insofar as it does not
interfere with the rank requirement.

\subsubsection{Hitchin spinors}
The Courant bracket~\eqref{E: my bracket} is antisymmetric but does not
in general satisfy the Jacobi identity. One new feature however is that
unlike the Lie bracket on sections of $TM$, this bracket has non-trivial
automorphisms defined by forms. Let $\beta \in \Gamma({\Lambda}^2 T^*M)$ 
be closed, and define the vector bundle automorphism of $TM \oplus T^*M$ 
known in literature as the B-field by
\begin{equation}
B_{\beta}(X + \xi) =X + \xi + i_X\beta.
\end{equation}
Then one may easily check that
\begin{equation}\label{E: Courant auto}
B_{\beta}([(X_1 + {\xi}_1, X_2 + {\xi}_2]) =
[B_{\beta}((X_1 + {\xi}_1),B_{\beta}(X_2 + {\xi}_2)].
\end{equation}
This interacts with a natural metric structure on $TM \oplus T^*M$.\\
\indent
Let $V$ be an n-dimensional real vector space and consider 2n-dimensional
space $V \oplus V^*$. The direct sum admits a natural nondegenerate inner
product of signature $(n,n)$ defined by
\begin{equation}\label{E: Hitchin1}
(X + \xi, X + \xi) = - \langle X, \xi \rangle,\;\;X\in V,\;\;\xi \in V^*.
\end{equation}
\noindent
The natural action of $GL(n, \mathbb{R})$ preserves this inner product.
The Lie algebra of the orthogonal group of all transformations preserving
~\eqref{E: Hitchin1} allows the following decomposition:
\begin{equation}
\mathfrak{so}(V \oplus V^*) = \text{End}V \oplus {\Lambda}^2V \oplus
{\Lambda}^2V^*.
\end{equation}
\noindent
In particular, $\beta \in {\Lambda}^2V^*$ acts via
$$X + \xi \mapsto i_X \beta$$ and thus exponentiates to an orthogonal
action on $V \oplus V^*$ given by
$$X + \xi \mapsto X + \xi + i_X \beta.$$ This is the algebraic action of
a closed 2-form which preserves the bracket~\eqref{E: my bracket}.\\
\indent
Consider the exterior algebra ${\Lambda}^*V^*$ and the action of 
$X + \xi \in V \oplus V^*$ on it defined by
\begin{equation}\label{E: Courant action}
(X + \xi) \cdot \varphi \overset{\textnormal{def}}{=} i_X\varphi +
\xi \wedge \varphi.
\end{equation}
\noindent
We have $$(X + \xi)^2 \cdot \varphi =i_X(\xi \wedge \varphi) +
\xi \wedge i_X \varphi =(i_X\xi)\varphi =-(X + \xi,X + \xi)\varphi.$$
\noindent
With this operation, the exterior algebra becomes a module over the
Clifford algebra of $V \oplus V^*$. Hence we obtain the spin representation
of the group $Spin(V \oplus V^*)$ acting on the space of differential forms
with values in the canonical line bundle $\sqrt{{\Lambda}^nV}$. Below we 
use the same notation for spinors and forms, as there is a one-to-one 
correspondence between them, and the context makes it clear. Splitting
into even and odd forms, we separate the two irreducible half-spin
representations:
$$ {\mathfrak{S}}^+ = {\Lambda}^{\text{even}}V^* \otimes \sqrt{{\Lambda}^nV},$$
$$ {\mathfrak{S}}^- = {\Lambda}^{\text{odd}}V^* \otimes \sqrt{{\Lambda}^nV}.$$
\noindent
Exponentiating $\beta \in {\Lambda}^2V^* \subset \mathfrak{so}(V \oplus V^*)$,
we end up with an element $\exp \beta \in Spin(V \oplus V^*)$ acting on spinors:
$$\exp \beta (\phi) = (1+ \beta + \frac{1}{2}\beta \wedge \beta + \cdots)
\wedge \phi.$$ 
\noindent
When $\dim V =2k$, there is an invariant bilinear form $\langle \phi, \varphi \rangle$
on ${\mathfrak{S}}^{\pm}$, symmetric for $k$ even, and antisymmetric for $k$ odd.
Using the exterior product, we expand the spinors in graded components: 
$$\langle \phi, \varphi \rangle =\sum_l (-1)^l {\phi}_{2l} \wedge {\varphi}_{2k -2l} 
\in  {\Lambda}^{2k}V^* \otimes (\sqrt{{\Lambda}^{2k}V^*})^2$$
for $\phi,\;\varphi \in {\mathfrak{S}}^+$, and
$$\langle \phi, \varphi \rangle =\sum_l (-1)^l {\phi}_{2l+1} \wedge 
{\varphi}_{2k -2l-1} \in  {\Lambda}^{2k}V^* \otimes (\sqrt{{\Lambda}^{2k}V^*})^2$$
for $\phi,\;\varphi \in {\mathfrak{S}}^-$.\\
\indent
The action~\eqref{E: Courant action} extends to spinors. Thus for $\varphi \in 
{\mathfrak{S}}^{\pm}$, we delineate its annihilator, the linear subspace
$$L_{\varphi} = \{ X+\xi \in V\oplus V^*|(X+\xi)\cdot \varphi =0\}.$$
\noindent
Since $(X+\xi)\in L_{\varphi}$ satisfies
$$0= (X+\xi)\cdot ((X+\xi)\cdot \varphi) =-(X+\xi, X+\xi)\varphi$$
\noindent
we see that $X+\xi$ is null, and so $L_{\varphi}$ is isotropic. Now we are
in a position to define a different notion of integrability in Courant 
algebroids.
\begin{defn}
A spinor $\varphi$, whose attendant annihilator subspace $L_{\varphi}$ is
maximally isotropic, $\dim L_{\varphi} =\dim V$ is called a \textit{pure spinor}.
\end{defn}  
\noindent
Any two pure spinors are related by some element of $Spin(V\oplus V^*)$. Purity
is a non-linear condition.\\
\indent
In this article, having confined ourselves to symplectic
manifolds $(M, \omega)$, we primarily work with the pure spinors
$$1 + \omega  +\frac{1}{2}\omega \wedge \omega + \cdots + \frac{1}{k!}{\omega}^k,$$
$$1 - \omega  -\frac{1}{2}\omega \wedge \omega - \cdots - \frac{1}{k!}{\omega}^k.$$
\noindent
Their maximal isotropic subspaces are the graph Dirac structures
$\{X +i_X\omega| X \in \Gamma (TM)\}$, and $({\omega}^{\#} - {\pi}^{\#})
\{X +i_X\omega| X \in \Gamma (TM)\}$. The relationship between the two, mediated
by our bundle map is not coincidental. Generally, we have
\begin{lem}
The bundle map $({\omega}^{\#} - {\pi}^{\#})$ preserves integrability
of Dirac structures. For every maximally isotropic integrable subbundle
$L \subset TM \oplus T^*M$, we have 
$$[({\omega}^{\#} - {\pi}^{\#})L, ({\omega}^{\#} - {\pi}^{\#})L]
\subset ({\omega}^{\#} - {\pi}^{\#})L.$$
\end{lem}
\begin{proof}
The automorphism~\eqref{E: Courant auto} preserves integrability (or lack
thereof). On the other hand, the action of $({\omega}^{\#} - {\pi}^{\#})$ admits
a representation on ${\mathfrak{S}}^{\pm}$ via the symplectic star operator, first
introduced by J.-L. Brylinski~\cite{Bryl}. Recall that ${\pi}^{\#}$ extends to
an operator taking differential forms into multivector fields by multiplicativity:
$${\pi}^{\#}({\beta}_1 \wedge \cdots \wedge {\beta}_l) = {\pi}^{\#}({\beta}_1)
\wedge \cdots \wedge {\pi}^{\#}({\beta}_l),$$ where ${\beta}_i$ are 1-forms. 
Furthermore, the contraction $i_{{\pi}^{\#}({\beta}_1 \wedge \cdots \wedge 
{\beta}_l)} {\omega}^k$ is a differential form of degree $2k-l$. Then we set 
$\ast \beta = i_{{\pi}^{\#}\beta} {\omega}^k$. This is a younger sibling of the 
Hodge star with the symplectic volume form in place of the Riemannian volume 
form. On 0-forms, we equate $\ast 1 = {\omega}^k$ to keep the normalizations 
consistent. Hence in the spinorial realm we have that same operator
\begin{equation*}\begin{CD}
(\phi, L_{\phi}) @>{\ast}>> (\ast \phi, L_{\ast \phi})
\end{CD}\end{equation*}
\noindent
acting according to the following rule: $\ast (\beta \otimes 
\sqrt{{\Lambda}^{2k}T^*M}) = (\ast \beta) \otimes \sqrt{{\Lambda}^{2k}T^*M}$.
Here $L_{\phi}$ stands for the annihilator subbundle of the spinor $\phi$.
Complementarity of the transformed spinors is manifested via 
$\langle \phi, \ast \phi \rangle \neq 0$.
Now we observe that the symplectic star operator intertwines with $B_{\beta}$ 
of~\eqref{E: Courant auto}: $\ast \circ B_{\beta} = B_{\ast \beta} \circ \ast$. 
Ergo so does $({\omega}^{\#} - {\pi}^{\#})$.
\end{proof}
\noindent
While there may exist an alternative proof of this result, not involving
spinors, our method highlights the fact that, on symplectic manifolds, Dirac
structures' integrability is akin to the standard Lie bracket integrability
of Lagrangian submanifolds.
\subsubsection{The symplectic connection}
Courant algebroids on symplectic manifolds enjoy some nice
properties. Here we establish the concept of a connection that
relates null Dirac structures on $(M, \;\omega)$ to Lagrangian
subbundles on $(T^*M, \;\Omega)$. The background information on
Lie algebroids, tangent bundles of bundle manifolds, and connections
comes from the definitive monograph of Kirill Mackenzee~\cite{Mack}.\\
\indent
Consider the manifold $T^*M$, where $(M, \omega)$ is a symplectic
manifold. Applying the tangent functor to the vector bundle operations
in $T^*M$ yields a vector bundle $(TT^*M,\; d{\text{pr}}^*,\;TM)$ called
\textit{the tangent prolongation of} $T^*M$. In conjunction with the
standard vector bundle structure $(TT^*M,\;{\text{pr}}_{T^*M},\;T^*M)$,
this forms \textit{the tangent double vector bundle of} $T^*M$. That
bundle is our basic object to work with. In order to describe it, 
we need to set the stage with the various projections. Their domains 
and ranges are indicated on the diagrams below, undecipherable to dvi 
viewers, but helpful to humans. We denote elements
of $TT^*M$ by $\mathcal{A}, \mathcal{B}, \mathcal{C}, \cdots$, and we 
write $(\mathcal{A}, \beta, X, x)$ to list all the variables in a local
bundle chart.
\begin{equation}
\begin{tikzcd}[column sep=small] 
&  
  TT^*M                                         
  \arrow[rightarrow]{dr}{{\text{pr}}_{T^*M}}       
  \arrow[swap]{dl}{d {\text{pr}}^*}                
\\
TM                                                             
  \arrow[swap]{dr}{\text{pr}}  \arrow{rr}{{\omega}^{\#}}       
&&
T^*M                                                            
  \arrow{dl}{{\text{pr}}^*} \\                                  
&
M  
\end{tikzcd}
\qquad\qquad
\begin{tikzcd}[column sep=small] 
&  
  (\mathcal{A}, \beta, X, x)
  \arrow[rightarrow]{dr}{{\text{pr}}_{T^*M}}
  \arrow[swap]{dl}{d {\text{pr}}^*} 
\\
(X,x)
  \arrow[swap]{dr}{\text{pr}}  \arrow{rr}{i_X\omega =\sum f(x)\beta}
&&
(\beta, x)
  \arrow{dl}{{\text{pr}}^*}  \\
&
(x)
\end{tikzcd}
\end{equation} 
\noindent
The projections are intrinsically consistent, specifically,
$${\text{pr}}^*({\text{pr}}_{T^*M}(\mathcal{A}))=x=
{\text{pr}}(d {\text{pr}}^*(\mathcal{A})). $$ 
\indent
There is a canonical 1-form $\vartheta \in \Gamma (T^*T^*M)$, called
the Liouville form, given by
\begin{equation}
\vartheta (\mathcal{A}) \overset{\textnormal{def}}{=}\langle 
{\text{pr}}_{T^*M}(\mathcal{A}),d {\text{pr}}^*(\mathcal{A})\rangle,
\;\;\; \mathcal{A} \in \Gamma (TT^*M). 
\end{equation}
\noindent
Hence the cotangent bundle comes naturally equipped with the canonical
2-form $\Omega = -d \vartheta$, and we select our symplectic bundle manifold
to be $(T^*M, \Omega)$.\\
\indent
In $TM$ and $T^*M$, we use standard notation to denote fiber addition and
multiplication by scalars. The zero of $T^*M$ over $x \in M$ is $0^{T^*M}_x$,
and the zero of $TM$ over $x \in M$ is $0^{TM}_x$. With respect to the 
standard vector bundle structure $(TT^*M,\;{\text{pr}}_{T^*M},\;T^*M)$, we 
continue to use + for addition, - for subtraction, and juxtaposition for
scalar multiplication. The symbol $T_{\beta}T^*M$ will always denote the fiber
$({\text{pr}}_{T^*M})^{-1}(\beta)$ for $\beta \in \Gamma (T^*M)$ with respect
to this bundle. The zero element in $T_{\beta}T^*M$ is denoted 
${\tilde{0}}_{\beta}$. Below we refer to this bundle structure as the 
standard tangent bundle structure.\\
\indent
In the prolonged tangent bundle structure, $(TT^*M,\; d{\text{pr}}^*,\;TM)$,
we use $\tilde{+}$ for addition, $\tilde{-}$ for subtraction, and $\cdot$
for scalar multiplication. The fiber over $X \in TM$ will always be denoted
$(d{\text{pr}}^*)^{-1}(X)$, and the zero element of this fiber is $(d0)_X$.
If we consider elements $\mathcal{A} \in \Gamma(TT^*M)$ as derivatives of 
paths in $T^*M$ and write $$\mathcal{A} =\frac{d}{dt}{\beta}_t \Bigr|_0,$$
where ${\beta}_t$ denotes a path in $T^*M$ for all $0 \leqslant t \leqslant
\epsilon$, then ${\text{pr}}_{T^*M}(\mathcal{A}) = {\beta}_0$, and
$d {\text{pr}}^*(\mathcal{A})= \frac{d}{dt}{\text{pr}}^*{\beta}_t|_0$.
If $\mathcal{A}, \mathcal{B} \in \Gamma (TT^*M)$ are located on the same
fiber, that is if $d {\text{pr}}^*(\mathcal{A})=d {\text{pr}}^*(\mathcal{B})$,
we can arrange that $\mathcal{A} =\frac{d}{dt}{\beta}_t|_0$, $\mathcal{B} =
\frac{d}{dt}{\eta}_t|_0$, where ${\text{pr}}^*({\beta}_t) ={\text{pr}}^*({\eta}_t)$
for all $t \leqslant \epsilon$, and, as a consequence,
$$\mathcal{A} \tilde{+}  \mathcal{B} = \frac{d}{dt}({\beta}_t + {\eta}_t)\Bigr|_0,       
\;\;\;c \cdot \mathcal{A} = \frac{d}{dt}c{\beta}_t\Bigr|_0.$$
\indent
For each $x \in M$, the tangent space $T_{0_x^{T^*M}}T_x^*M$ identifies canonically
with $T_x^*M$; we denote the element of $T_{0_x^{T^*M}}T_x^*M$ corresponding to 
$\beta \in T_x^*M$ by $\bar{\beta}$ and call it \textit{the core element}
corresponding to $\beta$. For $\beta, \eta \in T_x^*M,\;\;\;c \in \mathbb{R}$,
$$\bar{\beta} + \bar{\eta} = \overline{\beta + \eta} = \bar{\beta} \tilde{+} 
\bar{\eta}, \;\;\; c\bar{\beta} = \overline{c\beta} = c \cdot \bar{\beta}.$$
\indent
Regarding $d {\text{pr}}^*$ as a morphism over ${\text{pr}}^*$, the 
induced map $ d {\text{pr}}^{*!}: TT^*M \longrightarrow {\text{pr}}^{*!}TM$
is a surjective submersion. Here the target space ${\text{pr}}^{*!} TM$ is
the pullback bundle of $TM$. The kernel over $\beta \in T_x^*M$ consists of
vertical tangent vectors in $T_{\beta}T^*M$. We identify $T_{\beta}T^*M$
with $\{\beta\} \times T_x^*M$. Thus the kernel of $ d {\text{pr}}^{*!}$ is
identified with ${\text{pr}}^{*!}T^*M$. The injection  ${\text{pr}}^{*!}T^*M 
\longrightarrow TT^*M$ takes $(\beta, \eta) \in T_x^*M \times T_x^*M$ to
the vector in $T_x^*M$, denoted by $\eta @ \beta$ originating at the 
point $\beta(x)$ and parallel to $\eta (x)$. In terms of the prolongation
structure, its expression becomes $({\tilde{0}}_{\beta} \tilde{+} \bar{\eta})(x)$.
Thus, over $T^*M$, we have a short exact sequence of vector bundles
\begin{equation}\label{E: core sequence 1}
{\text{pr}}^{*!}T^*M \longrightarrow TT^*M \longrightarrow {\text{pr}}^{*!} TM,
\end{equation}
\indent
Analogously, there is a short exact sequence over $TM$,
\begin{equation}\label{E: core sequence for TM}
{\text{pr}}^{!}T^*M \longrightarrow TT^*M \longrightarrow {\text{pr}}^{!}TM,
\end{equation}
\noindent
where the inclusion ${\text{pr}}^{!}T^*M \hookrightarrow TT^*M$ is effected
via $(X,\eta) \mapsto (d0)_X + \bar{\eta}$.\\
\indent
We refer to~\eqref{E: core sequence 1} as the core sequence for 
${\text{pr}}_{T^*M}$, and to~\eqref{E: core sequence for TM} as the core
sequence for $d{\text{pr}}^{*}$.
\begin{defn}
A \textit{linear vector field} on $T^*M$ is a pair $(\mathcal{A}, X)$,
$\mathcal{A} \in \Gamma(TT^*M),\;\; X \in \Gamma(TM)$ such that 
\[
\begin{tikzcd}
T^*M \arrow{r}{\mathcal{A}}\arrow{d}[swap]{{\text{pr}}^{*}} & TT^*M
\arrow{d}{d{\text{pr}}^{*}} \\
M \arrow{r}{X}         &  TM
\end{tikzcd}
\]
\noindent
is a morphism of vector bundles.
\end{defn}
\noindent
To unravel this definition, ${d{\text{pr}}^{*}}(\mathcal{A}) = X$, and
$$\mathcal{A}(\beta + \eta) =\mathcal{A}(\beta) \tilde{+}\mathcal{A}(\eta),
\;\; \mathcal{A}(t\beta) = t \cdot \mathcal{A}(\beta),\;\;\; \forall
\beta, \eta \in \Gamma (T^*M), \;\;\;t \in \mathbb{R}.$$
\noindent
The sum of two linear vector fields is a linear vector field, and a 
scalar multiple of a linear vector field is a linear vector field. Given
a connection in $T^*M$, the  horizontal lift of any vector field on $M$
is a linear vector field. According to (\cite{Mack}, Section 3.4,
Corollary 3.4.3, and Theorem 3.4.5), the module of linear vector
fields is a Lie algebroid, and is isomorphic to the module of 
derivations on $T^*M$, as well as to the module of derivations on $TM$.\\
\indent
With the bijective correspondence between linear vector fields on $T^*M$
and $TM$ provided via the derivation modules, we now introduce a 
canonical pairing between $TT^*M$ and $TTM$, thought of as prolongation 
bundles over $TM$. Given $\widehat{X} \in \Gamma(TTM),\;\;\mathcal{A} \in 
\Gamma(TT^*M)$ such that ${d{\text{pr}}}(\widehat{X}) = 
{d{\text{pr}}^{*}}(\mathcal{A})$, and ${d{\text{pr}}}: TTM \longrightarrow
TM$ is the projection. Then we can find smooth paths 
$X_t \in \Gamma(TM) \times \mathbb{R}$, and ${\beta}_t \in \Gamma(T^*M)
\times \mathbb{R}$ with the property $$\widehat{X} = \frac{d}{dt}X_t\Bigr|_0, 
\;\;\;\mathcal{A} = \frac{d}{dt}{\beta}_t\Bigr|_0,$$
\noindent
and ${\text{pr}}(X_t) = {\text{pr}}^*({\beta}_t)$ for all $0 \leqslant t 
\leqslant \epsilon$.
\begin{defn}
The following binary operation $TTM \times TT^*M \longrightarrow C^{\infty}(M)$
is called \textit{the tangent pairing}:
$$\langle\langle \widehat{X}, \mathcal{A} \rangle \rangle  
\overset{\textnormal{def}}{=}\frac{d}{dt}\langle X_t, {\beta}_t\rangle
\Bigr|_0.$$
\end{defn}
\noindent
In (\cite{Mack}, Section 3.4) it is shown to be nondenegerate. Hence every 
smooth section of $T^*T^*M$ possesses a tangent pairing functional
representative for some suitable vector field, and $\langle\langle \mathcal{A},
\cdot \rangle \rangle \in \Gamma(T^*T^*M), \forall \mathcal{A} \in \Gamma(TT^*M)$.\\
\indent
Now we adapt infinitesimal connection theory to our bundle manifold
$(T^*M,\; \Omega)$. 
\begin{defn}\label{D: Koszul connection}
A \textit{Koszul connection} in $(T^*M,\;{\text{pr}}^*,\;M)$
is a map $$\nabla: \Gamma(TM) \times \Gamma(T^*M) \longrightarrow \Gamma(T^*M),
\;\;\;(X,\beta)\mapsto {\nabla}_X(\beta),$$
\noindent
which is bilinear and satisfies the two identities
$$ {\nabla}_{f}X(\beta)=f{\nabla}_X(\beta),\;\;\;{\nabla}_X(f\beta)=
f{\nabla}_X(\beta) + X(f){\nabla}_X(\beta),$$
holding for all $X \in \Gamma(TM),\;\beta \in \Gamma(T^*M),\;f\in C^{\infty}(M).$
\end{defn}
\noindent
It is clear that ${\nabla}_X$, for each $X \in \Gamma(TM)$, is a derivative 
endomorphism of $\Gamma(T^*M)$. Put differently, for every $X \in \Gamma(TM)$,
there exists a unique element ${\ccc}_X \in \Gamma(TT^*M)$. We denote the linear 
vector field corresponding to ${\nabla}_X$ by $({\ccc}_X, X)$. From these vector 
fields we construct a map $$\ccc: T^*M {\times}_M TM \longrightarrow TT^*M,\;\;\;
(\beta, X) \mapsto {\ccc}_X(\beta).$$
\noindent
From Definition~\ref{D: Koszul connection} it follows that $\ccc$ is a
right inverse to 
$$({\text{pr}}_{T^*M}, d{\text{pr}}): TT^*M \longrightarrow T^*M {\times}_M TM.$$
\noindent
Furthermore, ${\nabla}_X(\beta)$ is linear  in $\beta$, and ${\nabla}_X$ is
linear in $X$, so $\ccc$ is an arrow-reversing morphism for the two short
exact sequences~\eqref{E: core sequence 1},~\eqref{E: core sequence for TM} 
simultaneously. This result is formalized in (\cite{Mack}, Section 5.2,
Proposition 5.2.4) as stating that there is a bijective correspondence between
Koszul connections and maps $\ccc$. Therefore, from now on, we will use the
term `connection' to designate the map as well as its attendant Koszul connection.\\
\indent
Koszul connections are plentiful. Hence we specify our choice - one that is fully
compatible with the symplectic structure $(M, \omega)$, and Courant algebroid 
structure $(TM \oplus T^*M,\;[\cdot,\cdot],\;(\cdot,\cdot))$. 
\begin{defn}
A Koszul connection in $(T^*M, \;\Omega)$ that satisfies the conditions of
isotropy and nondegeneracy, formulated as
\begin{equation}
\Omega (\ccc(X + i_X \omega), \ccc (Y + i_Y \omega)) = 0;
\end{equation}
\begin{multline}\label{E: symplectic connection2}
\Omega (\ccc(X + i_X \omega), \ccc(({\omega}^{\#}-{\pi}^{\#})(Y + i_Y 
\omega)))\\
= {\langle} i_X \omega, \;{\pi}^{\#}(i_Y \omega){\rangle} -
\langle {\omega}^{\#}(Y),\; X \rangle,\;\;\;\forall X, Y \in \Gamma(TM),
\end{multline}
is called \textit{the symplectic connection}.
\end{defn}
\noindent
The symplectic connection is unique up to a symplectomorphism.\\
\indent
On $TT^*M \oplus T^*T^*M$ there is a natural Courant algebroid. Therefore, 
the symplectic connection maps pairs of transversal Dirac structures on 
$TM \oplus T^*M$ into pairs of isotropic integrable subbundles of dimension 
$2k$ dual with respect to the pairing $\langle \langle \cdot, \cdot \rangle
\rangle$. 
\begin{thm}\label{T: C-morphism}
The symplectic connection, viewed as a map 
$\ccc : T^*M \oplus TM \longrightarrow TT^*M \oplus T^*T^*M$, 
is a morphism of Lie bialgebroids.
\end{thm}
\begin{proof}
On $TT^*M \oplus T^*T^*M$, the bundle isomorphism ${\Omega}^{\#}$ takes vector
fields into differential forms. Thus the way to utilize our connection has to
be $$\ccc : (A, A^*) \longrightarrow (\ccc(A), {\Omega}^{\#}(\ccc(A^*))).$$ 
The second algebroid would always remain totally intransitive, just like $T^*M$.
We begin by tackling the case of  $(\ccc (TM),{\Omega}^{\#}(\ccc (T^*M)))$. 
To prove that it is indeed a Lie bialgebroid, we invoke (\cite{Mack}, Chapter 4, 
Proposition 4.3.3), according to which any smooth map $F: M \rightarrow M'$ 
induces a Lie algebroid morphism $dF: TM \rightarrow TM'$. In this case, $F$ 
is the natural embedding into the cotangent bundle: $F: M \hookrightarrow T^*M$. 
It is easy to see that there is an orthogonal matrix $[O]$ such that 
$dF[O] = \ccc|_{TM}$ in view of the nondegeneracy property of $\ccc$ 
encapsulated in~\eqref{E: symplectic connection2}. We infer $\ccc(TM)$ is 
integrable. Furthermore, it is the subbundle of horizontal vector fields of 
$TT^*M$. Now, the differential forms
defined by $\langle \langle \ccc(TM), \cdot \rangle \rangle$ are in the image
of the bundle diffeomorphism induced by the canonical symplectic form:
$$\langle \langle \ccc(TM), \cdot \rangle\rangle = {\Omega}^{\#}(\ccc(T^*M)).$$ 
\indent
As for arbitrary bialgebroids, we use the fact that the Hitchin spinor group
action is transitive. There exist one-parameter subgroups connecting the
identity of $Spin(2k,2k)$ with any fixed element. In particular, given a 
family of Dirac structures generated by a family 
of closed 2-forms $t\beta \in {\Lambda}^2T^*M, \;\; d\beta =0$, written as 
$\{X + i_Xt\beta|X \in TM\}$, $t\in [0,1]$, 
consider its image $\ccc(\{X +i_Xt\beta\})$. Now fix a nearby subbundle, 
$t< \epsilon $. We can take the word 'nearby` to mean that the projection  
$${\text{pr}}^{*!}: \ccc(\{X +i_Xt\beta\}) \longrightarrow \ccc(TM)$$ 
is a bundle diffeomorphism due to the fact that $\ccc(TM) \subsetneq 
{\text{pr}}^{*!}TM$, that is, all the lifts of vector fields are
horizontal in $TT^*M$. Then $\ccc(\{X +i_Xt\beta\})$ is an integrable 
subbundle of $TT^*M$. It is easy to check that the requirements of 
Definition 2.1. are satisfied with $\hat{a}(t) = \text{Id}|_{\ccc(\{X +
i_Xt\beta\}}$, the last mapping being the restriction of the identity of 
$TT^*M$. Furthermore, by virtue of ${\text{pr}}^{*!}$ being a surjective 
submersion for all $t \leqslant \epsilon  $, the pullback of the 
characteristic foliation of the Lie algebroid $\ccc(TM)$ is the 
characteristic foliation of the Lie algebroid $\ccc(\{X +i_Xt\beta\})$. 
The change of coordinates required to take the horizontal characteristic 
foliation into the characteristic
foliation of the image of the graph Dirac subbundle is 
\begin{equation}\label{E: graph foliation}
x_i \mapsto x_i + t\sum_j z_{ij}(x_1,...,x_{2k}){\dot{x}}_j,\;\;z_{ij}=-z_{ji}
\;\;i,j\in\{1,...,2k\},
\end{equation}
\noindent
where the dotted symbols are independent fiber variables, 
known as `momentum variables' on $T^*M$. Thus the Liouville form can be 
written locally as $\vartheta = \sum {\dot{x}}_jdx_j$.\\
\indent
To cover the interval $0 \leqslant t \leqslant 1$, we
use compositions of bundle diffeomorphisms projecting $\ccc(\{X +i_X(t + \epsilon)
\beta\})$ onto $\ccc(\{X +i_Xt\beta\})$. However, not all $\ccc(L)$ can
be projected surjectively onto $\ccc(TM)$. A refinement of the above argument is
in order. Thus let $\ccc(L)$ be the image of a Dirac structure such that 
fiberwise $$\dim {\text{pr}}^{*!}(\ccc(L)) = 2k-l,\;\;\; l>0,$$
$l=l(x_1,\cdots, x_{2k})$ is an integer-valued function with jump discontinuities.
Then there are $l \leqslant 2k$ linearly independent elements of $\ccc(L)$ 
satisfying $$\ker {\text{pr}}^{*!} \cap \ccc(L) = \text{span} 
\{ \ccc({\xi}_1),\cdots,\ccc({\xi}_l)\}.$$
The brackets involving those elements are easy to identify: 
$$[\ccc(X), \ccc({\xi}_i)] = \ccc(f(x_1,\cdots ,x_{2k}))\ccc({\xi}_i) 
\subset \ccc(L),$$ and $[\ccc({\xi}_i), \ccc({\xi}_j]= 0$, so that the subbundle 
$\ccc(L)$ is integrable.\\
\indent
Finally, we obtain $(\ccc(L), {\Omega}^{\#}(\ccc(({\omega}^{\#} 
-{\pi}^{\#})L)))$. In this setting, the coboundary operator 
is given by $d_*= d_{\Pi} = [\Pi, \cdot]$.
\end{proof}
Smooth functions of the base variables are sitting pretty inside $C^{\infty}(T^*M)$:
\begin{corollary}
$\ccc((C^{\infty}(M), \pi )) \subsetneq (C^{\infty}(T^*M), \Pi )$
is a subalgebra of the algebra of derivations on $TT^*M$.
\end{corollary}
\begin{lem}
The symplectic connection maps null Dirac structures into the Lagrangian subbundles 
of $T^*M$.
\end{lem}
\begin{proof}
Let $L_0$ be a fixed null Dirac structure on $TM \oplus T^*M$. Then we have 
$$L_0 \oplus ({\omega}^{\#}-{\pi}^{\#})L_0 = TM \oplus T^*M.$$
\noindent
The complementary subbundle is a null Dirac structure as well. Their images under
the symplectic connection do not intersect: $$\ccc (L_0) \cap \ccc 
(({\omega}^{\#}-{\pi}^{\#})L_0) = 0.$$ Hence we can take the basis of 
$\ccc (L_0) \oplus \ccc (({\omega}^{\#}-{\pi}^{\#})L_0)$ as a frame of $TT^*M$.\\
\indent
The action of 2-forms producing null Dirac structures refines~\eqref{E: graph foliation} 
as $$x_i \mapsto x_i + f_i(x_1,...,x_{2k}){\dot{x}}_i,\;\;\; i \in \{1,...,2k\}$$
with arbitrary smooth $f_i$'s. Now by applying the Darboux theorem , we can
recast this change of coordinates to get the simple graph Lagrangian subbundle chart:
$$x_i \mapsto y_i + {\dot{x}}_i,\;\;\; i \in \{1,...,2k\}.$$
In particular,
$$\ccc(\frac{\partial}{\partial x_i} + dx_{i+k}) = \widehat{\frac{\partial}
{\partial x_i}} +{\frac{\partial}{\partial {\dot{x}}_{i+k}}},\;\;\;i\leqslant k,$$
$$\ccc(\frac{\partial}{\partial x_{i+k}} - dx_{i}) = \widehat{\frac{\partial}
{\partial x_{i+k}}} -{\frac{\partial}{\partial {\dot{x}}_i}},\;\;\;i\leqslant k.$$
\end{proof}
\begin{lem}
The symplectic connection is flat: ${\nabla}^{\ccc}_{[X, Y]}(\beta)=
{\nabla}^{\ccc}_X{\nabla}^{\ccc}_Y(\beta)- {\nabla}^{\ccc}_Y{\nabla}^{\ccc}_X(\beta).$
\end{lem}
\begin{proof}
Flatness is a direct consequence of Theorem~\ref{T: C-morphism}.
\end{proof}

\subsubsection{Cartan's method}
A priori, two smooth manifolds equipped with specific tensors (i. e. pseudometrics,
almost complex structures, symplectic 2-forms) are not diffeomorphic unless
some precise conditions are satisfied. To formulate those conditions, we use the
notion of differential invariants. To solve the problem of finding all first-order
differential invariants, \'{E}lie Cartan~\cite{Cartan} proposed a very 
general technique, known today as the equivalence method of Cartan. It allows one
to solve the equivalence problem:\\
\indent \textit{
Let ${\Psi}_V =({\Psi}_V^1,...,{\Psi}_V^n)^T$ be a coframe (an n-tuple of 
smooth nonvanishing 1-forms of maximal rank) on an open set $V \subset 
{\mathbb{R}}^n$, and let ${\psi}_U =({\psi}_U^1,...,{\psi}_U^n)^T$ be a coframe
on $U \subset {\mathbb{R}}^n$, and let $G \subset GL(n, \mathbb{R})$ be a 
prescribed linear subgroup, then find necessary and sufficient conditions that
there exist a diffeomorphism $\AA : U \longrightarrow V$ such that for each
$u \in U$ $${\AA}^*{\Psi}_V \Bigr|_{\AA (u)} = [a]_{VU}(u){\psi}_U|_u,$$
\noindent
where $[a]_{VU}(u) \in G$.} \\
\noindent
The subgroups for the aforementioned tensors are $SO(p, n-p, \mathbb{R})$,
$GL(n, \mathbb{C})$, regarded as a closed real analytic subgroup of 
$GL(2n, \mathbb{R})$, and $Sp(n)$ respectively.\\
\indent
There are excellent sources, the lecture notes by Robert Bryant~\cite{B1} and 
the book by Robert Gardner~\cite{Gard} to name two. Even though the question we 
are trying to address is local, it proves advantageous to use the 
global coordinate-free language of bundles and connections. Thus on a manifold 
$M$ we form the coframe bundle. A coframe in this context can best be thought of 
as a 1-jet of a coordinate system at the point. The fiber $F^*_x$ is just 
the set of all coframes. The disjoint union of all such fibers
is denoted by $F^*$ and the bundle structure is determined by the projection
$\Vec{{\text{pr}}^*}$ built from component projections ${\text{pr}}^*: T^*M
\longrightarrow M$ introduced in the previous section.
$$\Vec{{\text{pr}}^*} : \;F^* \longrightarrow M $$ via $\Vec{{\text{pr}}^*}
(F^*_x)=x$. The group $GL(n, \mathbb{R})$ acts on $F^*$ by the rule 
\begin{equation}
\psi \cdot [a] = [a]^{-1}\psi \qquad \psi \in \Gamma(F^*),\; [a] \in 
GL(n, \mathbb{R}).
\tag{Action}
\end{equation} 
\begin{defn}
A $G$-{\textit{structure}} on $M$ is a subset $F^*_{(0)}
(G,\;M)$ of the bundle of 1-jets ${J^1}M$, such that 
$G \subset GL(n, \mathbb{R})$ acts simply transitive on 
$$F^*_{(0)}(G,\;x)\;=\;F^*_{(0)}(G,\;M)\cap(\Vec{{\text{pr}}^*})^{-1}(x).$$
\end{defn} 
By way of unraveling the last definition, we say that a $G$-structure is just 
a subbundle of the coframe bundle (which is emphasized by the use of the 
parenthetical subscript) with an especially nice action of 
$G \subset GL(n, \mathbb{R})$.\\ 
\indent
Cartan's equivalence problem can now be recast in terms of $G$-structures. 
Namely, two $G$-structures $_{1}F^{*}_{(0)}$ and $_{2}F^{*}_{(0)}$ on 
$M_1$ and $M_2$ respectively are said to be equivalent if there is a 
diffeomorphism $\AA : M_1 \rightarrow M_2$, so that 
it induces a bundle isomorphism ${\AA}^* :{_{1}F^{*}_{(0)}}\rightarrow {_{2}F^{*}_{(0)}}$.\\
\indent
In a standard trivialization $U \times G$, the Maurer-Cartan structure equations take the 
form:
\begin{equation}\label{E: Cartan1}
d \psi = [dO] \wedge {\psi}_U + [O]d{\psi}_U =[dO][O]^{-1} \wedge [O]{\psi}_U +[O]d{\psi}_U.
\end{equation}
\noindent
Here $\psi =[O]{\psi}_U$, and $[dO][O]^{-1}$ is, of course, the Maurer-Cartan matrix of
right invariant forms on $G$.\\
\indent
Recalling that ${\psi}_U$ are basic, that is, both coefficients and differentials can
be expressed in terms of coordinates on $U$, we can rewrite the structure equations in
the group-fiber representation:
\begin{equation}\label{E: Cartan2}
d{\psi}^i = \sum a^i_{jl} o^l \wedge {\psi}^j + \frac{1}{2}\sum {\gamma}^i_{jm}(x, [O])
{\psi}^j \wedge {\psi}^m.
\end{equation}
\noindent
The coordinates work out so that $([dO][O]^{-1})^i_j = \sum a^i_{jl} o^l$, where
$o^l$ is the basis of Maurer-Cartan forms, and $a^i_{jl}$ are constants. On the
contrary, the coefficients of quadratic terms depend on the base coordinates, as well
as on the group coordinates. The second sum on the right is called \textit{torsion}.\\
\indent
$[dO][O]^{-1}$ is not preserved by every 
equivalence between $G$-structures. In order to measure its deviation from 
equivariance, we have to analyze~\eqref{E: Cartan1}. \\ ${\gamma}^i_{jm}(x, [O])
\in \Gamma(TM \otimes {\Lambda}^2 T^*M)$ changes in step with $[dO][O]^{-1}$.
Gauging the connection form  $[dO][O]^{-1} \mapsto [dO'][O']^{-1}$ leads to the 
replacement $\gamma \mapsto {\gamma}'$ 
such that ${\gamma}'-\gamma =\delta ([M] \otimes \psi)$, and $\delta$ is the natural 
linear map $$\mathfrak{g} \otimes V^* \overset{\delta}{\longrightarrow} 
V\otimes {\Lambda}^2 V^*$$  defined as the composition 
\begin{equation*}\begin{CD}
\mathfrak{g} \otimes V^* @>{\text{Inclusion}}\otimes{\text{Id}}>>
(V \otimes V^*)\otimes V^* @>{\text{Skewsymmetrization}}>> V\otimes 
{\Lambda}^2 V^*.
\end{CD}\end{equation*}
\noindent
The cokernel of $\delta$ is extremely important. It has a special name -
the \textit{intrinsic torsion} of $\mathfrak g$ and is customarily denoted by
$H^{0,\;2}(\mathfrak g)$ as a reminder that the cokernel happens to be a
Spencer cohomology space. However, $\delta$ (whence its cokernel) depends 
on $\mathfrak g$ as well as the embedding of $\mathfrak g$ into 
$\mathfrak{gl}(n, \mathbb R)$. Also,
$\ker \delta = \mathfrak{s} \subset \mathfrak{g} \subset \mathfrak{gl}(n, \mathbb{R})$ 
consists of all symmetric matrices in $\mathfrak{g}$, and we have the following
exact sequence:
\begin{equation}\label{E: torsion exact sequence}
0 \longrightarrow \mathfrak{s} \longrightarrow \mathfrak{g} \otimes V^* 
\overset{\delta}{\longrightarrow} V \otimes {\Lambda}^2 V^* \longrightarrow H^{0,\;2}
(\mathfrak g) \longrightarrow 0.
\end{equation}
\indent
Now we need to build the group action into the analysis. We view $G \subset 
GL(n, \mathbb{R})$ as the identity inclusion. Now we generalize this viewpoint
to fit representation theory and let
$$ \rho: G \longrightarrow \text{Aut}(V), \;\; V \cong {\mathbb{R}}^n $$
\noindent
be the representation defining $G \subset GL(n, \mathbb{R})$. This induces
$$ {\rho}_*: T_eG \longrightarrow \text{Hom}(V,V)$$
\noindent
and results in
$$ {\rho}_*: \mathfrak{g} \longrightarrow \text{Hom}(V,V).$$
\noindent
Associated to $\rho$ is the dual representation
$$ {\rho}^{\dagger}: G \longrightarrow \text{Aut}(V^*)$$
\noindent
given explicitly by ${\rho}^{\dagger}([O]) =(\rho([O])^T)^{-1}.$
The natural $G$-action on $V\otimes {\Lambda}^2V^*$ is constructed
out of the individual representations:
$$\rho \otimes {\Lambda}^2 {\rho}^{\dagger}: 
G \longrightarrow \text{Aut}(V\otimes {\Lambda}^2V^*).$$
\indent
There is also a natural $G$-action on $\mathfrak{g}\otimes V^*$, given by
$$ \text{Ad}\otimes{\rho}^{\dagger}: G \longrightarrow 
\text{Aut}(\mathfrak{g}\otimes V^*).$$
\noindent
All these representations tie up beautifully with the map $\delta$:
$$\rho \otimes {\Lambda}^2 {\rho}^{\dagger}([O])\circ \delta(\cdot) =
\delta({\rho}_* \circ \text{Ad} \otimes {\rho}^{\dagger}([O])(\cdot)),$$
\noindent
which shows that the map $\delta$ is actually a mapping of $G$-modules,
and~\eqref{E: torsion exact sequence} is an exact sequence of $G$-modules.\\
\indent
Since the quotient map onto $H^{0,\;2}(\mathfrak g)$ is a $G$-module morphism
as well, we see that the composition
$${\tau}_U(x, [O]):\;U\times G \longrightarrow H^{0,\;2}(\mathfrak g)$$
acts by matrix multiplication
$${\tau}_U(x,[O'][O]) = \rho \otimes {\Lambda}^2 {\rho}^{\dagger}([O'])
{\tau}_U(x, [O])$$
\noindent
and deduce that the image of a fiber of $U \times G$ over $U$ is an orbit
of action of $G$ on $H^{0,\;2}(\mathfrak g)$.
\begin{defn}
An equivalence problem is of \textit{first-order constant type} if the 
image of ${\tau}_U(x, [O])$ is a single orbit on $H^{0,\;2}(\mathfrak g)$.
\end{defn}
\noindent
The method of equivalence guarantees that such problems admit a solution,
whereas multiple orbits typically take some ingenious ad hoc steps, and
may not be solvable at all. We confine our inquiry to special cases of 
first-order constant type. Now we choose a fixed vector, usually a
particular normal form in the image of the structure map, say
$${\tau}_0 = {\tau}_U(x_0, [O]_0) \in {\tau}_U(U \times G),$$
then there would be an isotropy group of that vector:
$$G_{{\tau}_0} = \{[O] \in G|\rho \otimes {\Lambda}^2 {\rho}^{\dagger}([O])
={\tau}_0\}.$$
Since ${\tau}_U(U \times G)$ is an orbit, and there is only one orbit,
it follows that ${\tau}_U(U \times G)$ is the orbit of ${\tau}_0$. By
transitivity, it can be identified with the homogeneous space $G/G_{{\tau}_0}$.
The structure map $${\tau}_U: U \times G \longrightarrow G/G_{{\tau}_0}$$
has constant rank, hence ${\tau}_U^{-1}({\tau}_0)$ is a manifold. The point
$ {\tau}_U(x, e)$ is on the orbit, hence for each $x$ there is an 
$[O](x) \in G$ such that 
$$\rho \otimes {\Lambda}^2 {\rho}^{\dagger}([O](x)){\tau}_U(x, e) = {\tau}_0.$$
Therefore 
$${\tau}_U^{-1}({\tau}_0)= \{x, [O](x)G_{{\tau}_0}|x \in U\}$$
is a manifold which submerses onto $U$. A ${\tau}_0$-\textit{modified coframe}
is a section of ${\tau}_U^{-1}({\tau}_0)$.\\
\indent
Should there exist an equivalence ${\AA}^*$ with ${\AA}^*(x_0, [O]_0) =
(y_0, [O']_0)$, it will result in
$${\tau}_0 = {\tau}_U(x_0, [O]_0)={\tau}_V \circ {\AA}^*(x_0, [O]_0) =
{\tau}_V(y_0, [O']_0).$$
With all this preparation, we can state
\begin{backthm}[Reduction of the structure group,~\cite{Gard}, Lecture 4]
A mapping
$$ \AA: U \longrightarrow V$$ induces a $G$-equivalence if and only if
${\AA}^*$ induces a $G_{{\tau}_0}$-equivalence between ${\tau}_0$-modified coframes.
\end{backthm}
\indent
The question becomes how to choose a convenient ${\tau}_0$. With an eye towards
equivalences with particularly amenable structure groups, typically direct products
of orthogonal ones, we further narrow down the class of problems treated. Namely,
we only deal with coframes that allow $G$-equivariant splitting of the vector space 
$V \otimes {\Lambda}^2 V^*$.\\
\indent
To find a normal form for the orbit ${\tau}_U(U \times G)$ in 
$H^{0,\;2}(\mathfrak g)$, the easiest is to work at the Lie algebra level. The group
action can be uncovered by computing the relations
$$ d^2 {\psi}^i =0 \mod \psi,$$
\noindent
and using linear combinations to solve for the torsion terms of~\eqref{E: Cartan2}.
The action splits into rotations, rescalings, and translations.
\begin{align}\label{E: torsion group action}
d{\gamma}^i_{jm}= \;&{\gamma}^i_{jm}\sum a^m_{ml}o^l -
{\gamma}^i_{jm}\sum a^j_{jl}o^l \\\notag
+\;&\sum a^u_{il}o^l{\gamma}^i_{um}-\sum a^v_{il}o^l{\gamma}^i_{jv}\\ \notag
+\;& \sum a^v_{(\lozenge \overset{d}{\rightarrow}jm)l} 
o^l{\gamma}^{(\lozenge \overset{d}{\rightarrow}jm)}_{jm} \;\;\mod \psi,\notag
\end{align}
\noindent
where $d{\psi}^{\lozenge} = {\gamma}^{(\lozenge \overset{d}{\rightarrow}jm)}_{jm}
{\psi}^j \wedge {\psi}^m + \cdots$, and $\lozenge \neq i$.\\
\indent
In his theory of R\'{e}p\`{e}re Mobile, Cartan called such a normalization
\textit{the first-order normalization}. Despite the somewhat misleading 
designation, this procedure can be repeated with progressively smaller groups,
yielding higher-order differential invariants.\\
\indent 
We keep track of the iterations by indexing the resulting $G$-structures:
$$F^*_{(0)}(G,\;U)\rightarrow F^*_{(1)}(G_1,\;U) \rightarrow F^*_{(2)}(G_2,\;U)
\cdots \rightarrow F^*_{(m)}(G_m,\;U).$$
\indent
As it turns out, preframes can be lifted to an equivalence problem of
a first-order constant type. Across the board, we would be able to 
normalize and reduce the structure 
group down, hence to select a unique coframe (on $T^*TU$):
$$ F^*_{(0)}(G,\;T_UU) \rightarrow F^*_{(1)}(G_1,\;T_UU) \cdots
\rightarrow F^*_{(m)}(\{e\},\;T_UU).$$ We use the idiosyncratic symbol
$T_UU$ to single out an open subset of $TU$ with compact closure.\\
\indent
After a number of iterations, the first-order normalization will
produce some $F^*_{(m)}(G_m,\;U)$ with the trivial action of $G_m$ on
$H^{0,\;2}(\mathfrak g)$. If $G_m=\{e\}$, then 
${\psi} =({\psi}^1,...,{\psi}^n)^T$ no longer involve group
parameters and define an invariant coframe on $U$. If $G$ has been
reduced to a group with $\mathfrak{s}=0$, then  
$({\psi}^1,...,{\psi}^n, o^1,...,o^l)$ defines an invariant coframe
on $U\times G$.\\
\indent
In either case, we have an equivalence problem, one on $U$, the other
one on $U\times G$ with an invariant coframe, and hence equivalence
problems with $G=\{e\}$. Such problems form a special class, and
the underlying coframes encode the invariants of \textit{e-structures}
(spaces with a specified coframe to put it simply).
This has a complete solution (\cite{Gard}, Lecture 6).\\
\indent
Assume we are given two coframings $(U,\psi)$ and $(V,\Psi)$, and want
to find necessary and sufficient conditions that
$${\AA}^* \Psi =\psi.$$
The post-reduction structure equations are
$$d{\psi}^i  =\sum {\gamma}^i_{jm}{\psi}^j\wedge{\psi}^m \;\;\text{and}\;\;
d{\Psi}^i  =\sum {\Gamma}^i_{jm}{\Psi}^j\wedge{\Psi}^m.$$
Here ${\Gamma}^i_{jm}$'s are general torsion coefficients, not to be 
confused with the Christoffel symbols, although those two may coincide.\\
\indent
If there were an equivalence, then $\mathfrak{s}=0$ implies
$${\gamma}^i_{jm}= {\Gamma}^i_{jm}\circ \AA.$$
Given a coframe ${\psi} =({\psi}^1,...,{\psi}^n)^T,$ there are natural
`covariant derivatives' for a function $f: U \longrightarrow \mathbb{R},$
defined by $$df= \sum f_{|i}{\psi}^i.$$
Similarly, given ${\Psi} =({\Psi}^1,...,{\Psi}^n)^T,$ and a function
$h:V \longrightarrow \mathbb{R},$ we define $$dh= \sum h_{;i}{\Psi}^i.$$
Therefore,
\begin{alignat}{2}
\sum {\gamma}^i_{jm|l} &= d {\gamma}^i_{jm} =d({\Gamma}^i_{jm}\circ \AA)=
{\AA}^*(d{\Gamma}^i_{jm})\notag\\ 
&= {\AA}^*(\sum {\Gamma}^i_{jm;l}{\Psi}^l)=\sum {\Gamma}^i_{jm;l}\circ \AA
{\AA}^*{\Psi}^l \notag\\
&= \sum {\Gamma}^i_{jm;l}\circ \AA {\psi}^l,\notag
\end{alignat}
\noindent
and as a result, $${\gamma}^i_{jm|l} = {\Gamma}^i_{jm;l}\circ \AA.$$
This argument is inductive for higher covariant derivatives.\\
\indent
Now define
$${\digamma}_{s}(\psi)\overset{\textnormal{def}}{=}
\{{\gamma}^i_{jm},{\gamma}^i_{jm|l_1},\cdots ,{\gamma}^i_{jm|l_1, \cdots,|l_{s-1}};   
1 \leqslant i, j, m, l_1,\cdots, l_{s-1} \leqslant n\},$$
which we view as a lexicographically ordered set. This is a set of invariants
of the $s$-jet of the $e$-structure $\psi$.\\
\indent
Two natural invariants of this set may be singled out. Let
$$ r_s = \text{rank} \{d{\digamma}_{s}(\psi)\},$$
where we count the dimension of the closed linear span of the differentials 
that occur in the ordered set $ {\digamma}_{s}(\psi)$.\\
\indent
Thus $r_s$ is an integer-valued function on $U$. The \textit{order} of the
$e$-structure at $x \in U$ is the smallest $j$ such that
$$ r_j(x) =r_{j+1}(x).$$
If $j$ is the order of the $e$-structure at $x$, then the \textit{rank}  of
the $e$-structure at $x \in U$ is $r_j(x)$. Note that $0 \leqslant j \leqslant n$,
the lower bound occuring when the structure tensor has constant coefficients
(i. e. on a Lie group), and the upper bound occuring if and only if one invariant
function is adjoined at each jet level.\\
\indent
An $e$-structure is called \textit{regular of rank} $r_j$ \textit{at} $x$ if 
the rank is constant in an open neighborhood of $x$. In this case there exist
functions $\{f_1,\cdots, f_{r_j}\}$ defined in an open neighborhood of $x$ such 
that $$f_1,\cdots, f_{r_j} \in {\digamma}_{j}(\psi), \;\;\;\;
df_1 \wedge \cdots \wedge df_{r_j} \neq 0,$$ while any function 
$u \in {\digamma}_{j}(\psi)$ satisfies
$$ du \wedge df_1 \wedge \cdots \wedge df_{r_j} = 0.$$
These $f_i$'s can be extended to a local coordinate system.\\
\indent
Now we can cite the fundamental theorem of Cartan, disposing of the 
restricted equivalence problem.
\begin{backthm}[Equivalence of $e$-structures,~\cite{Gard}, Lecture 6]
\label{T: e-structure equivalence}
Let $\psi$ and $\Psi$ be regular $e$-structures of the same rank $r_j$ and order $j$. Let
$$ h_U: U\rightarrow {\mathbb{R}}^m,\;\text{and}\;\;h_V:V \rightarrow {\mathbb{R}}^m$$
be extensions of an independent set of elements of ${\digamma}_j(\psi)$ and
${\digamma}_j(\Psi)$ to coordinate systems constructed from identical 
lexicographic  choices of indices. Define
$$ \sigma = h^{-1}_V \circ h_U: U \longrightarrow V,$$
then necessary and sufficient conditions that there exist
$$ \AA: U \longrightarrow V\;\;\text{with}\;\; {\AA}^* \Psi =\psi$$
are that
$$ {\digamma}_{j+1}(\Psi)\circ \sigma = {\digamma}_{j+1}(\psi)$$
as lexicographically ordered sets.
\end{backthm}

\subsubsection{Preternatural rotations}
\paragraph*{}
In the two sections immediately preceding this one, we described
(in broad terms) a version of integrability of subbundles of $TM \oplus T^*M$,
elucidated via Dirac structures, and intrinsic torsion of coframes. Now we 
introduce a principally new concept- designed nonintegrability of preframes. 
In various applications of Cartan's equivalence method one comes across a
partially reduced problem, and selects the values of intrinsic torsion
coefficients based on external or aesthetic considerations. We just want 
a linearized version of nonintegrability afforded by the Courant bracket
encoding certain metric structures on the base manifold lifted up to $TT^*M$. 
The discrete bundle automorphism $({\omega}^{\#} - {\pi}^{\#})$ has its
counterpart upstairs. It reshuffles symplectic-conjugate variables. We 
identify $2k$ smooth sections of $TT^*M$ satisfying the following
conditions:
\begin{equation}
{\mathcal{L}}_{{\mathcal{K}}_l} \Omega = 0, \;\;\; l \leqslant 2k,
\end{equation}
\begin{equation}
[{\mathcal{K}}_l, {\mathcal{K}}_m] = 0, \;\;\;l, m \leqslant 2k,
\end{equation}
\noindent
such that if we have a preframe with elements $X_i +{\xi}_i,\;\;
X_j + {\xi}_j$, their bracket being
$$[X_i + {\xi}_i, X_j + {\xi}_j] = \sum_{l=1}^{2k}h_l 
(X_l +{\xi}_l) + \sum_{l=1}^{2k}f_l 
({\omega}^{\#} - {\pi}^{\#})(X_l +{\xi}_l), \;\;\;\sum f^2_l >0,$$
\noindent
lifting this bracket onto  $TT^*M$ yields
$$\ccc([(X_i + {\xi}_i), (X_j + {\xi}_j)]) = \sum_{l=1}^{2k}
\ccc(h_l) \ccc(X_l +{\xi}_l) + \sum_{l=1}^{2k}
\ccc(f_l)[\mathcal{K}, \ccc(X_l +{\xi}_l)], $$
\noindent
where $\mathcal{K} = \sum {\mathcal{K}}_m$. It is easy to see, that,
owing to the symplectic connection, locally we
can write ${\mathcal{K}}_l = {\dot{x}}_l \widehat{\frac{\partial}
{\partial x_l}} - x_l {\frac{\partial}{\partial {\dot{x}}_l}}.$
Those corresponding precisely to $({\omega}^{\#} - {\pi}^{\#})$ are
discrete, constant rotations of this kind with a single parameter
set at $ \frac{\pi}{2}$.\\
\indent
We call the rotations ${\mathcal{K}}_m$ `preternatural' as they have
no precursors among the continuous symmetries associated with the Courant 
algebroid. The closest one gets to a precursor is, the closed 2-forms 
acting on the sections as Nigel Hitchin prescribed~\cite{Hitch}. 
Once lifted, they  become the `natural' pseudorotations
${\mathcal{O}}_l = {\dot{x}}_l \widehat{\frac{\partial}
{\partial x_l}} + x_l {\frac{\partial}{\partial {\dot{x}}_l}}.$\\
\indent
In terms of torsion, the coframes gotten from the 
nonintegrable preframes acquire additional components. 

\subsection{Regular preframes}
\paragraph*{}
As the title suggests, here we enjoy the luxuries of a symplectic manifold.
The equivalence method can be applied, hence smooth local coframes abound. 
The preframes are naturally derived from those coframes. We find it 
convenient to use the standard Poisson structure to effect that derivation.
Thus given a fixed coframe $\{({\psi}_U)_i\}$ we state
\begin{defn}
A \textit{regular preframe} on an open set $U \subset (M, \omega)$, 
endowed with a smooth coframe 
$\{({\psi}_U)_i\}$ is a 2k-tuple $\{({\psi}_U)_i + i_{{\pi}_U}({\psi}_U)_i\}$.
\end{defn} 
\noindent
We can easily check that a regular preframe is a preframe according to 
Definition~\ref{D: preframe}. Indeed we have
\begin{equation}
\text{span}(\{({\psi}_U)_i + i_{{\pi}_U}({\psi}_U)_i\},
({\omega}^{\#}_U - {\pi}^{\#}_U)\{({\psi}_U)_i + i_{{\pi}_U}({\psi}_U)_i\})
= TU \oplus T^*U, \;\;\text{and} 
\end{equation}
\begin{equation}
\{({\psi}_U)_i + i_{{\pi}_U}({\psi}_U)_i\} \cap 
({\omega}^{\#}_U - {\pi}^{\#}_U)\{({\psi}_U)_i + i_{{\pi}_U}({\psi}_U)_i\}=0.
\end{equation}
\noindent
Right away, limitations show up: the structure group of the equivalence
problem being solved ought to preserve the symplectic structure. That
shouldn't come as a surprise, as our regular preframes are just $2k$-tuples
of linearly independent sections of the symplectic graph Dirac subbundle.
Conversely, any preframe closed under the Courant bracket is regular.
Also, $\{({\psi}_U)_i\}$ and $\{({\psi}_U)_i + i_{{\pi}_U}({\psi}_U)_i\}$ 
are equivalent under the action of the Hitchin spinor group $Spin(TU \oplus
T^*U)$.\\
\indent
Among regular preframes, we single out a proper subset, distinguished by
its simplicity:
\begin{defn}
A \textit{Hamiltonian  preframe} on an open set $U \subset (M, \omega)$, 
is a regular preframe obtained from a smooth coframe of exact 1-forms, 
$\{(df_i + i_{{\pi}_U}df_i\}$.
\end{defn}
\noindent
Due to the fact that ${\pi}^{\#}_U = ({\omega}^{\#}_U)^{-1}$ as bundle
isomorphisms, the vector fields $i_{{\pi}_U}df_i $ are Hamiltonian and
linearly independent in $U$. One predictable property of global 
Hamiltonian preframes is expressed via 
\begin{thm}
A diffeomorphism $F: (M,\omega) \longrightarrow (M,\omega)$ is symplectic
if and only if it the induced bundle map $dF \oplus dF^*: TM \oplus T^*M
\longrightarrow TM \oplus T^*M$ preserves the set of Hamiltonian preframes.
\end{thm}
\begin{proof}  
The condition that the set of global Hamiltonian preframes be nonempty 
is topologically extremely restrictive. In particular, $(M,\omega)$ would 
have to be simply connected and parallelizable. Hence any symplectic 
action is Hamiltonian, and $\Longrightarrow$ follows. As for 
$\Longleftarrow$, the action that preserves Hamiltonian 
preframes is Hamiltonian since it splits into block-diagonal components
in a fixed trivialization of $TM \oplus T^*M$, and is symplectic a fortiori.
\end{proof}
\indent
The regular metric case entails some positive functions of the diagonalized
metric with the line element $ds^2 = \sum g^{ii} dx_i^2$. Abstractly,
\begin{equation}\label{E: regular preframe}
\begin{cases}
 f(g^{11})({\cos}\frac{\pi}{4}dx_1 + {\sin}\frac{\pi}{4}
 \frac{\partial}{\partial x_{1+k}});\\
...................................................\\
f(g^{kk})({\cos}\frac{\pi}{4}dx_k + {\sin}\frac{\pi}{4}
\frac{\partial}{\partial x_{2k}});\\
f(g^{k+1 k+1})({\cos}\frac{\pi}{4}dx_{k+1} - {\sin}\frac{\pi}{4}
\frac{\partial}{\partial x_{1}});\\
...................................................\\
f(g^{2k2k})({\cos}\frac{\pi}{4}dx_{2k} - {\sin}\frac{\pi}{4}
\frac{\partial}{\partial x_{k}}).\\
\end{cases}
\end{equation} 
\noindent
And the complementary preframe is gotten by applying $({\omega}^{\#}_U - 
{\pi}^{\#}_U)$ to the above preframe. Here $f$ is a general designation
for some as yet unknown standard normal form. This is just a choice of
$2k$ linearly independent sections of the symplectic graph Dirac 
subbundle described earlier. Integrability follows once we see that 
the same coefficient appears for the forms and fields.

\subsection{Singular preframes}
\paragraph{}
We plan to represent the metric tensor coefficients with some
suitable expressions that capture the essence of their pathology.
One has to bear in mind that rank deficiencies are to be represented
without compromizing linear independence. For that reason we need the
trigonometric functions.

\begin{equation}\label{E: singular preframe}
\begin{cases}
\sum_j e^{w_{1j}}({\cos}({\kappa}_{1j}(x_1,...))dx_j + 
{\sin}({\kappa}_{1j}(x_1,...))\frac{\partial}{\partial x_{k+j}});\\
.........................................................................\\
\sum_j e^{w_{kj}}({\cos}({\kappa}_{kj}(x_1,...))dx_j + 
{\sin}({\kappa}_k(x_1,...))\frac{\partial}{\partial x_{k+j}});\\
\sum_j e^{w_{k+1j}}({\cos}({\kappa}_{k+1j}(x_1,...))dx_{k +j} - 
{\sin}({\kappa}_{k+1j}(x_1,...))\frac{\partial}{\partial x_j});\\
.........................................................................\\
\sum_j e^{w_{2kj}}({\cos}({\kappa}_{2kj}(x_1,...))dx_{j} - 
{\sin}({\kappa}_{2kj}(x_1,...))\frac{\partial}{\partial x_j}).\\
\end{cases}
\end{equation} 
\noindent
The linear independence entails $1 \leqslant j \leqslant k$. The 
complementary preframe is gotten the same way as with the regular case.\\

\section{Regularization}

\indent
Our approach to nondifferentiable metrics is based on a
decomposition of sterilized versions of the coefficients
into products of standard nonvanishing functions and
trigonometric functions. After smoothing, and overcoming
technical difficulties thereof, the latter will be used to 
parameterize preternatural rotations.\\

\subsection{Encoding the metric}
\paragraph{}
From this section on, we denote by $M^N \subset M$, respectively
$M^O \subset M$ the global locus of nondifferentiability, respectively
the global locus of rank deficiency. Those are assumed to be finite
collections of smooth submanifolds.\\
\indent
The (pseudo-)metric is a symmetric tensor, locally written as
$$ g_U = \sum_{ij} g^{ij} dx_i \otimes dx_j,$$
\noindent
where $U \subset M$ is open and sufficiently small to accommodate
our construction. For a fixed $U$, we have $U^N = U \cap M^N$,
$U^O = U \cap M^O$, the total loci of nondifferentiability and
rank deficiency.
At least some $g^{ij}$ are nondifferentiable. Formally, $\exists
i_0,\;j_0$ such that $g^{i_0j_0} \notin C^1(M).$ We denote their
loci of nondifferentiability by $U^N_{ij} \subset U^N.$ We further
restrict the class of such coefficients treated below. 
Our requirements for the metric are:
\begin{enumerate}
\item
The attendant Christoffel symbols and Riemann tensor coefficients
in the local coordinates on $U$ are either bounded away from zero or 
identically zero for some special triples and quadruples of indices:
\begin{equation}\label{E: Christoffel}
\Big|{\Gamma}^m_{ij}\Big|\;\Big|_{U} > C_U,\;\;
{\Gamma}^{m_0}_{{i_0}{j_0}}\Big|_{U} \equiv 0,\;\;
\Big|R^m_{ijs}\Big|\;\Big|_{U} > C_U,\;\;
R^{m_0}_{{i_0}{j_0}{s_0}}\Big|_{U} \equiv 0.
\end{equation} 
\noindent
where $C_U$ is a uniform constant depending only on the size of $U$.
\item
There ought to exist a single parameter,
serving as a coordinate on the base manifold, that designates all
loci of nondifferentiability of the metric. Specifically, we set
$x_1$ as the blow-up parameter. Hence, each $U^N_{ij}$ consists
of a finite number of smooth hypersurfaces transversal
to the $x_1$-direction in $U$. A locus of higher codimention would need 
some additional coordinates to be specified apart from a value of $x_1$.
\item
Furthermore, the metric coefficients are allowed to vanish, their loci
of rank deficiency, being a subset of $M$ where $\det g =0$, are smooth 
hypersurfaces\footnote{Vanishing loci depending on other coordinates
are acceptable, especially those associated with noncartesian coordinate
systems. They just have to have sufficiently high codimension.} collectively 
comprising its locus of rank deficiency $U^O_{ij}$, 
every one transversal to the $x_1$-direction as well. 
\item
Our fourth requirement is, the metric coefficients be explicitly given by
$$g^{ij}(x_1,...) = \frac{(g^{ij})^+}{(g^{ij})^-},\;\;\;
(g^{ij})^+,\;(g^{ij})^-  \in C^0(M).$$
However, away from their respective locus of nondifferentiability/rank
deficiency we need
$(g^{ij})^+ \in C^2(M\backslash U^O_{ij}), (g^{ij})^-\in 
C^2(M\backslash U^N_{ij}).$ 
\item Finally, we also require\\
$$\textnormal{rank}[|g^{ij}|]\bigr|_{M\backslash(U^N \cup U^O)} = \dim M,$$
\noindent
where $U^N = \bigcup_{i,j}U^N_{ij},\;\;U^O =\bigcup_{i,j}U^O_{ij}$.
The locus of rank deficiency of $g$ must consist of at most a finite
number of hypersurfaces.
\end{enumerate}
\indent
Our standard normal form on the tubular neighborhood 
$$((U^N \cup U^O) \times (-\epsilon, \epsilon)) \cap U$$  
potentially involving both loci with $\bar{U}$ compact is 
\begin{equation}\label{E: standard form}
(e^{w_{ij}}\cos({\kappa}_{ij}))^2\bigr|_{(U^N_{ij} \cup U^O_{ij})
\times (-\epsilon, \epsilon)}=\left|\frac{\sin ((g^{ij})^-)}
{(g^{ij})^-}\frac{\sin ((g^{ij})^+)}{(g^{ij})^+ }\right|.
\end{equation}
\noindent
The right-hand side of~\eqref{E: standard form} is a product of 
the well-known sinc functions. We adopt the conventional 
(in electrical engineering and digital signal processing) definition:
\begin{equation}\label{E: sinc 1}
\textnormal{sinc}(0) \overset{\textnormal{def}}{=} 
\lim_{x_1 \to x_1^{\textnormal{blow}}}\frac{1}{(g^{ij})^-}\sin ((g^{ij})^-)=1,
\end{equation}
\begin{equation}\label{E: sinc 2}
\textnormal{sinc}(0) \overset{\textnormal{def}}{=} 
\lim_{x_1 \to x_1^{\textnormal{nought}}}\frac{1}{(g^{ij})^+}\sin ((g^{ij})^+)=1.
\end{equation}
\noindent
With that definition in place, we delineate the relationship 
between the weights $w_{ij}(x_1,...,x_{2k})$'s, and the 
trigonometric functions, implicit in \eqref{E: standard form}. 
Wherever the metric coefficient blows up or vanishes, we set 
\begin{equation}
\cos({\kappa}_{ij}))\bigr|_{U^N_{ij}} = \pm 1,\;\;\;
\cos({\kappa}_{ij}))\bigr|_{U^O_{ij}} = \mp 1,\;\;\text{and}
\end{equation}
\begin{equation}
e^{2w_{ij}}\bigr|_{U^N_{ij}} = |(g^{ij})^+)^{-1}\sin ((g^{ij})^+|,\;\;\;
e^{2w_{ij}}\bigr|_{U^O_{ij}} = |(g^{ij})^-)^{-1}\sin ((g^{ij})^-|.
\end{equation}
\noindent
Outside of its loci of nondifferentiability and rank deficiency, but over
some other loci, $U^N_{uv} \cap (U^N_{ij} \cup U^O_{ij}) = \emptyset$,
$U^O_{uv} \cap (U^N_{ij} \cup U^O_{ij}) = \emptyset$, we maintain
\begin{equation}
\cos({\kappa}_{ij}))\bigr|_{U^N_{uv}} = \pm 1,\;\;\;
\cos({\kappa}_{ij}))\bigr|_{U^O_{uv}} = \mp 1.
\end{equation}
\noindent
The weights there break the pattern:
\begin{equation}
e^{2w_{ij}}\bigr|_{U^N_{uv}} = \left|\frac{\sin ((g^{ij})^-}{(g^{ij})^-}
\frac{\sin ((g^{ij})^+}{(g^{ij})^+}\right|,
\end{equation}
\begin{equation}
e^{2w_{ij}}\bigr|_{U^O_{uv}} = \left|\frac{\sin ((g^{ij})^-}{(g^{ij})^-}
\frac{\sin ((g^{ij})^+}{(g^{ij})^+}\right|.
\end{equation}
\indent
On $\partial U \backslash$loci of nondifferentiabily and rank deficiency, we require
\begin{equation}\label{E: boundary condition1}
\cos ({\kappa}_{ij}(x_1,...,x_{2k})) \bigr|_{\partial U \backslash (\partial U \cap 
(U^N \cup U^O) \times (-\epsilon, \epsilon))} =\frac{\sqrt{2}}{2};
\end{equation}
\begin{equation}\label{E: boundary condition2}
\dda {\kappa}_{ij}(x_1,...,x_{2k})) \bigr|_{\partial U \backslash (\partial U \cap 
(U^N \cup U^O) \times (-\epsilon, \epsilon))} = 0.
\end{equation}
\noindent
Informally speaking, away from $U^N \cup U^O$, all extant dependencies
are transferred onto the weights.\\
\indent
The tubular neighborhood of $U^N \cup U^O$ takes more work. 
The intricate blow-up behavior has to be captured consistently. 
To this end we use analytic properties of $g$. Specifically,
we introduce an auxiliary scalar field to represent each $g^{ij}$;
depending on the number of metric coefficients present on $U$, the
total may reach $2k^2 +k$.\\
\indent
To produce that scalar field, we utilize one second-order differential
equation closely allied with the metric. The Laplace-Beltrami operator
commutes with all Killing vector fields, and thus carries all metric
data. It is obtained by taking the (covariant) Hessian 
$$\textnormal{Hess}f \in \Gamma (T^*M \otimes T^*M),\;\;\;
\textnormal{Hess}f \overset{\textnormal{def}}{=} {\nabla}^2 f 
\equiv \nabla df,$$ 
\noindent
and contracting it with respect to the (pseudo-)metric. Thus let 
$\{X_i\}$ be a basis of $TM$ (not necessarily induced by a 
coordinate system). Then the components of
$\textnormal{Hess}f$ are $(\textnormal{Hess}f)_{ij} = 
\textnormal{Hess}f(X_i, X_j) = {\nabla}_{X_i} {\nabla}_{X_j}f - 
{\nabla}_{{\nabla}_{X_i}X_j}f$. In terms of the metric, 
the Laplace-Beltrami equation is
$$ {\Delta}_g f = \sum_{ij} g^{ij} (\textnormal{Hess}f)_{ij} =0.$$
\noindent
For the Minkowski pseudometric, the Laplace-Beltrami operator coincides
with the familiar d'Alembertian: $ {\Delta}_g f = {\square}_g f$. In local
coordinates, this becomes
$${\Delta}_g f = \sum_{ij} g^{ij}(\frac{{\partial}^2f}{\partial x_i 
\partial x_j}- {\Gamma}^m_{ij} \frac{\partial f}{\partial x_m})\;=0.$$
\noindent
Since $g^{ij},\;{\Gamma}^m_{ij}$ are nondifferentiable on $U^N$, we seek 
a solution on $U$ (subscripted $e^{w_{uv}(x_1,...,x_{2k})}_U
\sin({\kappa}_{uv}(x_1,...,x_{2k}))_U$), modified by the exponent 
$l_g(U)$, understood to be the second smallest positive integer fully 
damping the blowup of every summand of the Laplace-Beltrami operator: 
$$\lim_{x_1 \to x_1^{\textnormal{blow}}}\left|g^{ij}\frac{{\partial}^2}
{\partial x_i \partial x_j} (e^{w_{uv} (x_1,...,x_{2k})}_U
\sin({\kappa}_{uv}(x_1,\cdots, x_{2k}))_U)^{l_g -1}\right| \leqslant C,$$
$$\lim_{x_1 \to x_1^{\textnormal{blow}}}\left|g^{ij}{\Gamma}^m_{ij} 
\frac{\partial}{\partial x_m}(e^{w_{uv} (x_1,...,x_{2k})}_U
\sin({\kappa}_{uv}(x_1,\cdots, x_{2k}))_U)^{l_g -1}\right| \leqslant C,$$
\noindent
$\forall i,j,m,u,v$, and every value of $x_1^{\textnormal{blow}}$ 
uniformly in $U$. Therefore we take the value working in both cases,
specifically $l_g =10$, and the equation becomes
\begin{equation}\label{E: Laplace} 
{{\Delta}_g}(e^{w_{uv}(x_1,...,x_{2k})}_U\sin({\kappa}_{uv}
(x_1,...,x_{2k}))_U)^{10}=0,
\end{equation}
\noindent
subject to the boundary conditions~\eqref{E: boundary condition1},
~\eqref{E: boundary condition2}. The solutions of~\eqref{E: Laplace} 
are obtained by a limiting process involving series expansions of
offending $g^{ij},\;{\Gamma}^m_{ij}$, each pair diverging at its locus 
of nondifferentiability, as well as a power series for the exponent:
\begin{equation}
g^{ij} = \sum_{s=1}^{\infty} (a_{ij})_s,\;\;\;
{\Gamma}^m_{ij} =\sum_{s=1}^{\infty} (b_{ij}^m)_s,\;\;
10 = \sum_{s=0}^{\infty} l^s,\;\;\; l^s = (\frac{9}{10})^s.
\end{equation}
\noindent
Then~\eqref{E: Laplace} is approximated via partial sums:
\begin{multline}
\lim_{N \to \infty} \sum_{ij} \left(\sum_{s=1}^{N}(a_{ij})_s\right) 
\left(\frac{{\partial}^2}{\partial x_i \partial x_j}
- \sum_{s=1}^{N} (b_{ij}^m)_s \frac{\partial}{\partial x_m}\right)\\
\left(e^{w_{uv}(x_1,...,x_{2k})}_U\sin({\kappa}_{uv}
(x_1,...,x_{2k}))_U \right)^{1 + \sum_{s=1}^{N} l^s}\;=0.
\end{multline}
\indent
The limiting solutions are not unique, and generically are elements 
of $C^2(U)$. One bit of information captured this way is whether the 
trigonometric factor $\cos({\kappa}_{uv}(x_1,...,x_{2k}))_U$ of the 
solution in scrutiny is a Morse-Bott function at one (or more) of 
hypersurface components of $U^N_{uv}$ or $U^O_{uv}$.\\
\indent
Further treatment is obviously in order. Two objectives are as yet 
to be achieved: smoothing, and selecting appropriate scalar fields - 
both crucial for building an e-structure for application of 
Background Theorem~\ref{T: e-structure equivalence}. To this end
we introduce one selection criterion, a far-reaching implication
of $U^N_{uv} \cap U^O_{uv} = \emptyset$: a constraint on the partial
derivatives of the solution factors. Precisely, by our definitions,
$\frac{\partial {\kappa}_{uv}}{\partial x_{l_0}} \equiv 0$ implies
$\frac{\partial w_{uv}}{\partial x_{l_0}} \equiv 0$ and vice versa,
whenever there are manifold variables $x_{l_0}$ of which $g^{uv}$ is
independent. Now, as a consequence of~\eqref{E: Christoffel}, for
all nonvanishing Christoffel symbols, we have
\begin{equation}\label{E: solution rank}
\Big| \frac{\partial {\kappa}_{uv}}{\partial x_l}\Big| +
\Big|\frac{\partial w_{uv}}{\partial x_l}\Big| > C'_U,\;\;\; 
\forall l \neq l_0. 
\end{equation}
\noindent
Contrary to one's expectation, this does not exclude the solutions
with irregular $\textnormal{rank} \;\nabla 
(e^{w_{uv}(x_1,...,x_{2k})}_U\sin({\kappa}_{uv}(x_1,...,x_{2k}))_U).$

\subsection{Smoothing}
\paragraph{}
Presently, we construct a correspondence between our (nonunique)
solutions of the boundary value problem~\eqref{E: boundary condition1},
~\eqref{E: boundary condition2},~\eqref{E: Laplace} and elements of 
$C^{\infty}(U)$. The basic technique below is due to 
Nash~\cite{N}, modified by Gromov (\cite{Grom}, Section 2.3).\\
\indent
For a function $f: {\mathbb{R}}^{2k} \longrightarrow \mathbb{R}$, we
introduce its $C^0$-norm by $||f||_0 = {\sup}_{x \in {\mathbb{R}}^{2k}}
|f(x)|$, and for $0 <\alpha <1$ we set its H\"{o}lder $C^{\alpha}$-norm
to be
\begin{equation}
||f||_{\alpha} \overset{\textnormal{def}}{=} \max(||f||_0,\; 
\sup_{x, w}(|w|^{-\alpha}|f(x + w) -f(x)|)),
\end{equation}
\noindent
where $x$ is an arbitrary point, and $w$ runs over all nonzero vectors 
in the unit ball centered at the origin of ${\mathbb{R}}^{2k}$.
For an arbitrary $\alpha = j + \theta, \;j=0,1,\cdots,\; \theta \in 
[0,1)$, we put $||f||_{\alpha} \overset{\textnormal{def}}{=}
||J^j_f||_{\theta}$, $J^j_f$ being the jet of $f$ over $U$, and if 
$f \notin C^j(U)$, we set 
$||f||_{\alpha} = \infty,\;\; \forall \alpha \geqslant j$.\\
\indent
We fix a sequence of linear operators $S_i: C^0(U) \rightarrow 
C^{\infty}(U),\;\; i=0,1,\cdots$. 
\begin{defn}\label{D: locality}
The sequences that satisfy 
\begin{enumerate}
\item[]
Locality. Every $S_i$ does not enlarge supports of elements of the
domain space by more than ${\epsilon}_i = (2\epsilon +2)^{-1}$, that 
is, the value $(S_if)(x_1, \cdots, x_{2k})$ depends on that $f$ 
within the ball of radius ${\epsilon}_i$ for all $f$.\\
\item[]
Convergence. If $f \in C^{\alpha}({\mathbb{R}}^{2k}), \;\; \alpha =
0,1,\cdots$, then $S_if \rightarrow f$ as $i \rightarrow \infty$
in the usual (not fine) topology. Moreover, $C^{\alpha}$-convergence
$f_i \rightarrow f$ implies $C^{\alpha}$-convergence $S_if_i
\rightarrow f$.
\end{enumerate}
\noindent
are called \textit{local smoothing operators}.
\end{defn}
\begin{defn}\label{D: smoothing estimates}
A sequence of smoothing operators $S_0, S_1, \cdots$ has \textit{Nash
depth} $\bar{d_i}$, if for every compact $\bar{U} \subset 
{\mathbb{R}}^{2k}$, there are some constants $C_{\alpha},\;\alpha 
\in [0, \infty)$, uniformly bounded on every finite interval
$[0, \alpha] \subset [0, \infty)$, such that all functions satisfy
the following inequalities with the norms $||\cdot||_{\alpha} =
||\cdot||_{\alpha}(\bar{U})$ for all $\alpha \in [0, \infty)$, and
all $i \in {\mathbb{Z}}^+$ (smoothing estimates):
\begin{equation}
||S_{i-1}(f)||_{\alpha} \leq\; C_{\alpha}i^{2\beta}||f||_{\alpha -\beta}
\;\;\textnormal{for} \;0 \leq \beta \leq \alpha.
\end{equation}
\begin{equation}
||(S_i -S_{i-1})(f)||_{\alpha} \leq\; C_{\alpha}(i^{-2\bar{d}-1}
 + i^{-2\beta -1})||f||_{\alpha +\beta}
\;\;\textnormal{for} \;-\alpha \leq \beta < \infty.
\end{equation}
\begin{equation}
||S_{i-1}(f) - f||_{\alpha} \leq\; C_{\alpha}(i^{-2\bar{d}}
 + i^{-2\beta})||f||_{\alpha +\beta}
\;\;\textnormal{for} \; \beta \geq 0.
\end{equation}
\end{defn}
\indent
To construct `deep smoothing', we start with a $C^{\infty}$-function
$S: {\mathbb{R}}^{2k} \rightarrow \mathbb{R}$ supported in the unit 
ball centered at the origin, $B_0(1) \in {\mathbb{R}}^{2k}$ and for 
an arbitrary $f \in C^0({\mathbb{R}}^{2k})$ we consider the convolution
\begin{multline}
(S \ast f)(x_1,\cdots,x_{2k}) = \\ \int_{{\mathbb{R}}^{2k}}
S(y_1, \cdots,y_{2k})f(x_1 +y_1, \cdots, x_{2k} + y_{2k})dy_1 \cdots dy_{2k}.
\end{multline}
\noindent
Then we modify our original function, now set to be 
$$S_{\zeta}(x_1,\cdots,x_{2k})= {\zeta}^{2k}S(\zeta x_1, 
\cdots,\zeta x_{2k}),\;\;\; \forall \zeta \geqslant 1,$$ 
and observe that $S_{\zeta}$ is supported within a smaller ball, 
$B_0({\zeta}^{-1})$ to be precise, and $$||S_{\zeta} \ast f||_{\alpha} 
\leqslant \int_{{\mathbb{R}}^{2k}}|S_{\zeta}(y)||f||_{\alpha}dy.$$
\indent
At this point we normalize $S_{\zeta}$ via
$$\int_{{\mathbb{R}}^{2k}} S_{\zeta}(y) dy =1.$$
Such operators $f \mapsto S_{\zeta} \ast f$ converge, as $\zeta 
\rightarrow \infty$, in all $C^j$-topologies to the identity operator.\\
\indent
A normalized smoothing operator has Nash depth $\bar{d}$ whenever its
kernel function is orthogonal  to all homogeneous polynomials of degrees
$1,\cdots, \bar{d}$. One produces such kernels out of an arbitrary
normalized $S$ by taking linear combinations
$$S^{\text{deep}} = \sum_{m=0}^{\bar{d}} a_m S_{{\zeta}_m}.$$
\noindent
The normalization condition for $S^{\text{deep}}$ amounts to the 
identity $\sum_{m=0}^{\bar{d}} a_m =1$, and the depth condition is
encapsulated in the equation
$$\sum_{m=0}^{\bar{d}} a_m{\zeta}_m^{-j} =0,\;\;\; j = 1,\cdots,\bar{d}.$$
\noindent
All such operators satisfy the estimates of 
Definition~\ref{D: smoothing estimates}. Gromov (\cite{Grom}, section 2.3.4)
goes on to show those convolutions to be infinitely differentiable functions.\\
\indent
To adapt the above smoothing scheme to manifolds, we embed $U$ into
${\mathbb{R}}^{2k}$ isometrically, and extend local scalar fields and
(potential) kernels of the smoothing operators onto the ambient space
using partitions of unity. Owing to the locality condition of 
Definition~\ref{D: locality}, the resulting approximating fields would
depend entirely on the data originated in $U$.\\
\indent
Now we prove that such an essential property of the auxiliary scalar 
fields as living in the kernel of the Laplace-Beltrami operator is 
preserved after smoothing. 
\begin{lem}
For every solution of the boundary value problem~\eqref{E: boundary 
condition1},~\eqref{E: boundary condition2},~\eqref{E: Laplace} 
$e^{w_{vu}(x_1,...,x_{2k})}_U{\sin}({\kappa}_{vu}(x_1,...,x_{2k}))_U$
we have\\
\begin{equation}
|{\Delta}_g( e^{ S \ast w_{vu}(x_1,...,x_{2k})}_U\sin( S \ast
{\kappa}_{vu}(x_1,...,x_{2k}))_U)^{10}| \leqslant \varepsilon (\bar{d)}.
\end{equation}
\noindent
such that $\varepsilon (\bar{d)}$ is a strictly decreasing function.
\end{lem}
\begin{proof}
Blowing up of the metric coefficients significantly limits our choices
of the deep smoothing operators. We select them individually for each
particular coefficient, and label accordingly. Thus we define the two
selection criteria:
\begin{equation}
\sin (S^{vu}_{{\zeta}_m} \ast {\kappa}_{vu}) \bigr|_{U^N \cup U^O} =0,
\end{equation}
\noindent
and $m$ large enough to maintain
\begin{equation}
|\det \text{Hess}\cos ({\kappa}_{vu})| > 0 \Longrightarrow
|\det \text{Hess}\cos (S^{vu}_{{\zeta}_m} \ast {\kappa}_{vu})| > 0.
\end{equation}
\noindent
Our criteria are borne out by the fact that $U^N \cup U^O$ is a set
of measure zero, and Definition~\ref{D: smoothing estimates}. All
other features are subordinated. Thus stationary black hole solutions
of EVE take time-independent $S_{{\zeta}_m}$. Symmetries are allowed
so long as they do not interfere with the aforementioned selection
criteria.\\ 
\indent
Inside the ball of radius ${\zeta}^{-1}_m$ centered at 
$(x_1^{\text{blow}},x_2, \cdots, x_{2k})$, we employ the following bound:
\begin{multline}
||\sin(S^{vu}_{{\zeta}_m} \ast {\kappa}_{vu}) - \sin ({\kappa}_{vu})||_0
\leqslant \\C_o(m+1)^{- \bar{d} -1} \max \{ 
|{\kappa}_{vu}(x_1^{\text{blow}} + {\zeta}^{-1}_m,x_2, \cdots, x_{2k})|,
|{\kappa}_{vu}(x_1^{\text{blow}} - {\zeta}^{-1}_m,x_2, \cdots, x_{2k})|\},
\end{multline}
valid for all normalized $S^{vu}_{{\zeta}_m}$ of Nash depth $\bar{d}$ 
and all smooth proper hypersurfaces. In the case of higher codimension,
the estimate to be used is
\begin{equation}
||\sin(S^{vu}_{{\zeta}_m} \ast {\kappa}_{vu}) - \sin ({\kappa}_{vu})||_0
\leqslant \\C_o(m+1)^{- \bar{d} -1} ||{\kappa}_{vu}||_0(B({\zeta}_m^{-1})).
\end{equation}
\noindent
Similarly,
\begin{equation}
||e^{S^{vu}_{{\zeta}_m} \ast w_{vu}} - e^{w_{vu}}||_0
\leqslant C_o(m+1)^{- \bar{d} -1} ||e^{w_{vu}}||_0(B({\zeta}_m^{-1})).
\end{equation}
\indent
To estimate the derivatives we write
\begin{equation}
||\frac{\partial}{\partial x_i}(e^{S^{vu}_{{\zeta}_m} \ast w_{vu}} - 
e^{w_{vu}})||_0\leqslant C_o(m+1)^{- \bar{d} -1} 
||\frac{\partial}{\partial x_i}e^{w_{vu}}||_0(B({\zeta}_m^{-1})).
\end{equation}
\noindent
The rest of estimates is gotten in the same way.\\
\indent
To prove our claim we note that
\begin{multline}
|{\Delta}_g( e^{ S \ast w_{vu}}_U\sin( S \ast
{\kappa}_{vu})_U)^{10}| =\\
|{\Delta}_g((e^{ S \ast w_{vu}}_U\sin( S \ast
{\kappa}_{vu})_U)^{10} - (e^{w_{vu}}_U\sin({\kappa}_{vu})_U)^{10})|.
\end{multline}
\indent
From this point on, the proof becomes a tedious computation. To give
a taste, we tackle one summand:
$$|g^{ij}{\Gamma}^a_{ij}((\frac{\partial}{\partial x_a}
e^{10S^{vu}_{{\zeta}_m} \ast w_{vu}}){\sin(S^{vu}_{{\zeta}_m}} 
\ast {\kappa}_{vu})^{10} - (\frac{\partial}{\partial x_a}e^{10w_{vu}})
{\sin({\kappa}_{vu})}^{10})|=$$
$$|g^{ij}{\Gamma}^a_{ij}((\frac{\partial}{\partial x_a}
e^{10S^{vu}_{{\zeta}_m} \ast w_{vu}}){\sin(S^{vu}_{{\zeta}_m}}  
\ast {\kappa}_{vu})^{10} - (\frac{\partial}{\partial x_a}
e^{10S^{vu}_{{\zeta}_m} \ast w_{vu}}){\sin{\kappa}_{vu})}^{10} + $$
$$(\frac{\partial}{\partial x_a}e^{10S^{vu}_{{\zeta}_m} \ast w_{vu}})
\sin({\kappa}_{vu})^{10} -(\frac{\partial}{\partial x_a}e^{10w_{vu}})
{\sin({\kappa}_{vu})}^{10})| =$$
$$|g^{ij}{\Gamma}^a_{ij}((\frac{\partial}{\partial x_a}
e^{10S^{vu}_{{\zeta}_m} \ast w_{vu}} - \frac{\partial}{\partial x_a}
e^{10w_{vu}}){\sin(S^{vu}_{{\zeta}_m}} \ast {\kappa}_{vu})^{10} +$$
$$(\frac{\partial}{\partial x_a}e^{10w_{vu}})({\sin(S^{vu}_{{\zeta}_m}}
\ast {\kappa}_{vu})^{10}- \sin({\kappa}_{vu})^{10}))| \leqslant$$
$$|g^{ij}{\Gamma}^a_{ij}|(||(\frac{\partial}{\partial x_a}
e^{10S^{vu}_{{\zeta}_m} \ast w_{vu}} - \frac{\partial}{\partial x_a}
e^{10w_{vu}})||_0|({\sin(S^{vu}_{{\zeta}_m}}\ast 
{\kappa}_{vu})^{10}| +$$
$$|\frac{\partial}{\partial x_a}e^{10w_{vu}}||{\sin(S^{vu}_{{\zeta}_m}}
\ast {\kappa}_{vu})^{10}- \sin({\kappa}_{vu})^{10})|) \leqslant$$
$$|g^{ij}{\Gamma}^a_{ij}|(C_0 (m+1)^{- \bar{d} -1}
||\frac{\partial}{\partial x_a}e^{10w_{vu}}||_0 
|{\sin(S^{vu}_{{\zeta}_m}}\ast {\kappa}_{vu})^{10}|+ $$
$$|\frac{\partial}{\partial x_a}e^{10w_{vu}}|
|\sum_{b=0}^9{\sin(S^{vu}_{{\zeta}_m}}\ast {\kappa}_{vu})^{9-b}
\sin({\kappa}_{vu})^{b})|C_0(m+1)^{-\bar{d} -1}||{\kappa}_{uv}||_0).$$
\end{proof}
\noindent
Representatives of this set of infinitely differentiable functions can,
in principle, be used as coefficients of a preframe, but,
unlike the genuine solutions, they are not related
via the action of local isometries. Their problem of nonuniqueness 
is further exacerbated by the involvement of nonunique smoothing operators.
With an eye on Cartan's method, specifically, reduction of the coframe 
bundle down to a unique coframe, we have to effect further transformation.
Geometrically, we want to demonstrate that the preframes with
coefficients obtained this way cannot occur in a smooth configuration. 
There must exist a factorization into preternatural rotations
and regular functions.\\

\subsection{Spectral theory of the Hill operator}

\subsubsection{An overview}
We list all the relevant information on the Hill operator here.
Our sources (ranging from very basic to recent) are the book
by Magnus and Winkler~\cite{MW}, a seminal paper by McKean and
Trubowitz~\cite{McT}, and the monograph by Feldman, Kn\"{o}rrer and
Trubowitz~\cite{FKT}.\\
\indent
$Q$ denotes the Hill operator $-\frac{d^2}{dx^2} +q(x)$ with
a fixed $q(x)$, an infinitely differentiable function of
period 1. The function $y_1(x, \lambda)$, respectively
$y_2(x, \lambda)$, is the solution of $Qy = \lambda y$ with
$y_1(0, \lambda) =1,\;\;y'_1(0, \lambda) =0$, respectively
$y_2(0, \lambda) =0,\;\;y'_2(0, \lambda) =1$. These functions
allow the following integral representation~\cite{MW}:
\begin{eqnarray}\label{E: Hill solutions}
y_1(x, \lambda) = \cos(\sqrt{\lambda}\xi) +
\int^{\xi}_0 \frac{\sin (\sqrt{\lambda(\xi - \eta)})}{\sqrt{\lambda}}
q(\eta)y_1(\eta, \lambda)d\eta,\\
y_2(x, \lambda) = \frac{\sin(\sqrt{\lambda}\xi)}{\sqrt{\lambda}} +
\int^{\xi}_0 \frac{\sin (\sqrt{\lambda(\xi - \eta)})}{\sqrt{\lambda}}
q(\eta)y_2(\eta, \lambda)d\eta.
\end{eqnarray}
\indent
The spectrum of $Q$ acting on the class of twice differentiable
functions of period 2 is called periodic. It is a sequence  of
real single or double eigenvalues\footnote{The adjectives `simple' and
`double' refer to the dimension of the corresponding eigenspace.}
tending to infinity:
$$ {\lambda}_0 <   {\lambda}_1 \leqslant  {\lambda}_2  <{\lambda}_3
\leqslant  {\lambda}_4 <  {\lambda}_5  \leqslant {\lambda}_6 < \cdots 
\uparrow \infty.$$
The lowest eigenvalue ${\lambda}_0$ is simple, the eigenfunction
$f_0$ being root-free and of period 1. Then come pairs of eigenvalues
${\lambda}_{2i-1}  \leqslant {\lambda}_{2i},\;\;i = 1,2, \cdots$,
equality signifying that the eigenspace is of dimension 2. Both the
eigenfunctions $f_{2i-1}$ and $f_{2i}$ have $i$ roots in a period
$0 \leqslant x <1$ and are themselves of period 1 or 2 according to
the parity of $i$, i. e. being of period 1 if $i = 2,4,6 \cdots$,
and of period 2 if $i = 1,3,5 \cdots$. The eigenfunctions are 
normalized by
$$\int_0^1 f^2_i(x)dx =1, \;\;\; i = 0,1,2,\cdots .$$
The eigenvalues obey the estimate
\begin{equation}\label{E: lambda double}
{\lambda}_{2i-1}, {\lambda}_{2i} = i^2 {\pi}^2 + \int^1_0 q(x)dx
+ O(i^{-2})\;\; \textnormal{as} \; i \uparrow \infty.
\end{equation}
The periodic spectrum falls into two parts: the double spectrum of
pairs $ {\lambda}_{2i-1} = {\lambda}_{2i}$ and the simple spectrum
comprised of distinct eigenvalues. The interval $({\lambda}_{2i-2}  
{\lambda}_{2i-1})$ is an interval of stability; the nomenclature
is suggestive since every solution of $Qy = \lambda y$ is bounded
if ${\lambda}_{2i-2}< \lambda <  {\lambda}_{2i-1}$. The complementary
intervals of instability (also called lacunae) $(-\infty, {\lambda}_0]$,
$[{\lambda}_{2i-1}  {\lambda}_{2i}], \;\;i= 1,2,\cdots$, behave 
differently: no solution of $Qy =\lambda y$ is bounded for $\lambda
< {\lambda}_0$ or for $ {\lambda}_{2i-1}< \lambda < {\lambda}_{2i}$.
The periodic spectrum is all double except ${\lambda}_0$ if and only
if the potential is constant.\\
\indent
An alternative way to describe the spectrum makes use of the 
discriminant $\bigstar (\lambda)$ defined as
\begin{equation}\label{D: discriminant}
\bigstar (\lambda) \overset{\textnormal{def}}{=} 
y_1(1, \lambda) + y'_2(1, \lambda).
\end{equation}
\noindent
All eigenvalues lead to $\bigstar ({\lambda}_i) = \pm 2$ with
signature +1 or -1 according to whether $i \equiv 0, 3 \mod 4$ or
$i \equiv 1, 2 \mod 4$, therefore the periodic spectrum is just
the set of roots of ${\bigstar}^2(\lambda) -4 =0.$ The dicriminant
is an entire function of order $\frac{1}{2}$ and type 1, informally,
$\bigstar (\lambda) \backsim \cos \sqrt{\lambda}$, so 
$\bigstar (\lambda) + 2$, respectively $\bigstar (\lambda)-2$, may
be expressed as a constant multiple of the canonical product of
${\lambda}_0, {\lambda}_3, {\lambda}_4 \cdots$, respectively
${\lambda}_1, {\lambda}_2, {\lambda}_5, \cdots$. See~\cite{McT} and
references therein.\\
\indent
Apart from the periodic spectrum, there are the roots ${\mu}_i, 
i=1, 2,\cdots$ of $y_2(1,{\mu}_i)= 0$ forming the spectrum of
$Q$ acting  on the class of twice differentiable functions with
$f(0)=f(1)=0$. They interlace the periodic spectrum 
${\lambda}_{2i-1}< {\mu}_i <  {\lambda}_{2i}$, and fall into two
classes: the trivial roots at the double periodic eigenvalues,
and the remaining nontrivial roots in the nondegenerate intervals 
of instability. They are collectively named the tied spectrum.\\
\indent
Lastly, there is the reflecting spectrum ${\nu}_i,\;i=0,1,2,\cdots$
formed by the roots of $y_1(1, {\nu})=0$ encountered as $Q$ acts on
the class of functions $f \in C^2([0.1])$ obeying $f'(0) =f'(1) =0$.
They are similar to the tied spectrum eigenvalues with one notable
exception: there is an extra root ${\nu}_0$ in the lacuna
$(- \infty, {\lambda}_0]$.\\
\noindent
This classical result is known as the coexistence theorem for
the Hill operator.
\begin{backthm}[\cite{MW}, Theorem 7.11]\label{T: Magnus coexist}
Hill equations with trigonometric polynomial potentials 
cannot have finite trigonometric polynomial solutions.
\end{backthm}
\noindent
One way it can be formalized is as follows: 
\begin{equation*}
[\frac{{\partial}^2}{\partial x^2} + \sum_i^m a_i\cos(ix)]
\sum_i^N b_i\sin(ix) \ne \lambda \sum_i^N b_i\sin(ix)\;\;\forall  
\lambda, \;\; m, N < \infty.
\end{equation*}
\indent
To study the spectrum via modern functional analysis, McKean and
Trubowitz~\cite{McT} introduce the class of smooth functions of
period of length 1 having a fixed periodic spectrum of the Hill
operator: ${\lambda}_0 < {\lambda}_1 \leqslant {\lambda}_2 <
{\lambda}_3 \leqslant {\lambda}_4 < \cdots$. We designate this
isospectral class $\mathfrak{T} \subset C^{\infty}(\mathbb{R}/
\mathbb{Z})$. Geometrically, it is nicely nested inside
a Hilbert space where normalized eigenfunctions and their derivatives
constitute an orthonormal basis. While technically not being a manifold,
$\mathfrak{T}$ possesses a tangent bundle and a normal bundle viewed
as subspaces of the ambient Hilbert space. These findings are summarized in
\begin{backthm}[\cite{McT}, Theorem 1, Section 8]\label{T: McKean basis}
Let $T\mathfrak{T}$ be the span of $\ddx(f^o_{2i})^2, i= 1,...$ in $L^2_1$, 
and let $N\mathfrak{T}$ be the span of $(f^o_{2i})^2, i= 1,...$ 
supplemented by the span of $(f^{\times}_{2i-1})^2, i= 1,...$ and 
$(f^{\times}_{2i-1})f^{\times}_{2i}$ for such double eigenvalues
as may exist. Then \textnormal{(1)} $T\mathfrak{T} \perp N\mathfrak{T}$, 
\textnormal{(2)} $T\mathfrak{T} \oplus N\mathfrak{T} = L^2_1$, 
\textnormal{(3)} the unit function {\bf{1}} belongs to $N$, while the 
functions $$ F: \sqrt{2}[(f^o_{2i})^2 - 1],\; \sqrt{2}[(f^{\times}_{2i})^2 
- 1],\; -2^{\frac{3}{2}}f^{\times}_{2i-1}f^{\times}_{2i},\;
-\sqrt{2}(2 \pi)^{-1} f^o_{2i}\ddx f^o_{2i}$$
form an oblique base to the annihilator ${\bf{1}}^{\circ}$ of the
unit function, meaning that any $f \in {\bf{1}}^{\circ}$ can be
uniquely written as $f = \sum c_iF_i$, its norm
$$\sqrt{\int_0^1 |f(x)|^2dx}$$
being comparable to $\sqrt{\sum c_i^2}$.
\end{backthm}
\indent
$\mathfrak{T}$ is (generally) homeomorphic to the infinite-dimensional torus,
an arbitrary $q(x) \in \mathfrak{T}$ being uniquely determined by the sequence
${\mu}_i,\;i=0, 1, 2,\cdots$ and the norming constants. The former may be
regarded as coordinates on the torus. Each such sequence encodes a periodic $q$.\\
\indent
It is an important geometrical fact that $\mathfrak{T}$ has a Poisson bracket.
$\frac{d}{dx}$ maps $N\mathfrak{T}|_{f^o_{2i}}$ onto $T\mathfrak{T}$. Hence there
is the infinite-dimensional bilinear operation utilizing that property given by 
\begin{equation}
\{F,\;G\} = \int^1_0 \frac{\partial F}{\partial q}\frac{d}{dx}
\frac{\partial G}{\partial q}dx.
\end{equation}
\noindent
This system is completely integrable. The only dissimilarity with the 
classical Poisson bracket is the dimension of the normal space being larger
than the dimension of the tangent space by one. Thus there is a degeneracy
inherent in this operation.\\
\indent
The periodic spectrum, ergo $\mathfrak{T}$, is preserved by the flow with
Hamiltonian $H=\bigstar (\lambda)$, $\lambda$ fixed. The flow is defined by
solving $\frac{\partial q}{\partial t} = Xq = \frac{d}{dx}\frac{\partial H}
{\partial q}$. $X$ can be viewed as a tangent vector field to $\mathfrak{T}$.
The flows commute just as they do in finite-dimensional
classical mechanics whenever the fields preserve each other Hamiltonians
(\cite{McT}, Section 3, Theorem 1, Amplification 2),
since with the $X$ above we get $X\bigstar (\mu) =0$, and the flow of 
$H=\bigstar (\mu)$, say $Y$, returns the favor: $Y\bigstar (\lambda)=0$.
This allows natural representations of one-dimensional subgroups of Lie
groups with Hamiltonian action.\\
\indent 
$\mathfrak{T}$ can be considered the real part of a complex Jacobi variety of 
the transcendental hyperelliptic irrationality $\sqrt{{\bigstar}^2(\lambda)-4}$.
The differentials of the first kind on the attendant Riemann surface
${\mathfrak{T}}_{\mathbb{C}}$ of genus 
$g=\infty$ are of the form~\cite{McT, McT2}
\begin{equation}\label{E: trans differentials}
d\Phi = \frac{\phi }{\sqrt{{\bigstar}^2(\lambda)-4}}d\lambda.
\end{equation}
\noindent
$\phi$ belongs to the class $I^{\frac{3}{2}}$ of integral functions of
order $\frac{1}{2}$ and type at most 1 controlling its growth such that
\begin{equation}
\int^{\infty}_0 |\phi(\mu)|^2 {\mu}^{\frac{3}{2}}d \mu < \infty.
\end{equation}
\noindent
$I^{\frac{3}{2}}$ is naturally a separable Hilbert space. Its dual,
denoted $I^{\frac{3}{2}*}$, consists of absolutely summable sequences 
${\phi}^* =\{{\phi}^*_1, {\phi}^*_2, \cdots\} \in l^2_{-4}$, such that 
\begin{equation}
\sum^{\infty}_{i=1} |{\phi}^*_i|^2 {i}^{-4} < \infty.
\end{equation}
\noindent
The Jacobi map takes the form (modulo periods):
\begin{equation}
\sum_{v=0}^{\infty} 2\int^{{\mu}_v +i\sqrt{{\bigstar}^2({\mu}_v)-4}}
_{{\lambda}_{2v-1}}\frac{\phi(\mu)d \mu}{\sqrt{{\bigstar}^2({\mu}_v)-4}} 
= \sum_{j=1}^{\infty}\phi({\lambda}_{2j}){\phi}^*_j.
\end{equation}
\noindent
The path of integration is permitted to wind around the circle $n_v$ 
times provided $\sum n_v^2 v^{-2} < \infty$.

\subsubsection{Hyperelliptic curves}
The predominant source for this subsection is the monograph
by Trubowitz et al.~\cite{FKT}. The objects of study
are the hyperelliptic Riemann surfaces ${\mathfrak{T}}_{\mathbb{C}}$,
appearing in conjunction with purely simple spectra of the Hill operator,
colloquially known as Hill surfaces. We do not tackle a more complicated
case of Hill operator with mixed (single and double eigenvalue) spectrum,
and the limiting case of purely double spectrum does not entail
hyperelliptic Riemann surfaces at all.\\
\indent
From now on, ${\mathfrak{T}}_{\mathbb{C}}$ is an open Riemann surface
of infinite genus. We assume it possesses an infinite canonical homology
basis $A_1,B_1,A_2,B_2,\cdots$. With the Hodge decomposition, the 
first Hodge-Kodaira cohomology group $H^1_{HK}({\mathfrak{T}}_
{\mathbb{C}})$ consists of smooth, closed and coclosed differential
forms. It is a bona fide Hilbert space with the inner product 
\begin{equation}
< \eta, \beta > = \int_{{\mathfrak{T}}_{\mathbb{C}}} \eta 
\wedge \overline{\ast\beta}.
\end{equation} 
\noindent
The differential forms living in 
\begin{equation}
\textnormal{Hol}({\mathfrak{T}}_{\mathbb{C}})=
\{\beta \in H^1_{HK}({\mathfrak{T}}_{\mathbb{C}})|
\ast{\beta} = -i\beta \}
\end{equation}  
\noindent
are declared to be holomorphic. It follows that $\textnormal{Hol}
({\mathfrak{T}}_{\mathbb{C}})$ is a closed subspace of 
$H^1_{HK}({\mathfrak{T}}_{\mathbb{C}})$ with inner product
\begin{equation}
< \eta, \beta > = \int_{{\mathfrak{T}}_{\mathbb{C}}} \eta 
\wedge \overline{\ast\beta}= i\int_{{\mathfrak{T}}_{\mathbb{C}}} \eta 
\wedge \overline{\beta}.
\end{equation} 
\noindent
These forms are the backbone of the holomorphic structure on 
noncompact complex varieties $\mathbb{S}$ of dimension one, not 
just Riemann surfaces. In particular, $\text{Hol}(\mathbb{S})$ nail 
down the holomorphic structure on pairs of complex planes
identified along a discrete set of points $[(\mathbb{C}\backslash 
\{{\lambda}_i\})\times  (\mathbb{C}\backslash \{{\lambda}_i\})]/_
{\sim}$. Such varieties will serve as limiting cases for 
Hill potentials with purely double spectrum.\\
\indent
Now we impose an extra condition: on ${\mathfrak{T}}_{\mathbb{C}}$
there exists an exhaustion function with finite charge - a proper
nonnegative Morse function $h(z)$ that satisfies
\begin{equation}
\int_{{\mathfrak{T}}_{\mathbb{C}}}|d\ast{dh}| < \infty.
\end{equation}
\indent
For each $\beta \in \textnormal{Hol} ({\mathfrak{T}}_{\mathbb{C}})$,
let $D_r(z), z \in ({\mathfrak{T}}_{\mathbb{C}})$ be the disk of 
radius $r$ centered at $z$. We have a representation in terms of
a local coordinate $\zeta$:
\begin{equation}
\beta|_{D_r(z)} = f(\zeta (z)) d\zeta.
\end{equation}
\noindent
One might think of this as an analogue of the Poincare lemma for
holomorphic (possibly multivalued) functions.\\
\indent
Now we introduce a linear functional taking the Riemann
surface into the Hilbert space dual to ${\textnormal{Hol}}
({\mathfrak{T}}_{\mathbb{C}})$,
\begin{equation}
{\delta}_{z, \zeta} (\cdot): D_r(z) \longrightarrow
{\textnormal{Hol}}^*({\mathfrak{T}}_{\mathbb{C}})\;\;
\textnormal{via}
\end{equation}
\begin{equation}
{\delta}_{z, \zeta} (\beta)\overset{\textnormal{def}}{=}
\oint_{\partial D_r(z)} f(\zeta)d\zeta \;=\;
\frac{1}{2\pi} \int_0^{2\pi} f(\zeta(z) + re^{i\alpha})d\alpha.
\end{equation}
\noindent
${\delta}_{z, \zeta} (\cdot)$ is a weakly analytic map that is 
actually holomorphic. As it stands, the choice of local
coordinates determines the local representative $f(\zeta)$.
However, there is a way to get rid of that dependence given by
the canonical map $\aleph$. It utilizes the Cauchy-Riemann
equivalence relation. Namely, ${\delta}_{z, \zeta} \sim
{\delta}_{z, {\zeta}'}$ if and only if the coordinates
${\zeta},{\zeta}'$ are both holomorphic with respect to $z$:
\begin{equation}
\aleph = {\delta}_{z, \zeta}/ \sim \;: \;{\mathfrak{T}}_{\mathbb{C}}
\longrightarrow {\textnormal{Hol}}^*({\mathfrak{T}}_{\mathbb{C}})/
 _{\textnormal{Cauchy-Riemann}} = \mathbb{P}
({\textnormal{Hol}}^*({\mathfrak{T}}_{\mathbb{C}})).
\end{equation}
\begin{defn}
A Riemann surface is hyperelliptic if there is a finite subset
$I \in {\mathbb{P}}^1(\mathbb{C})$, a discrete subset $S \in 
{\mathbb{P}}^1(\mathbb{C})\backslash I$, and a proper holomorphic
map $\tau :\;\;{\mathfrak{T}}_{\mathbb{C}} \longrightarrow 
{\mathbb{P}}^1(\mathbb{C})\backslash I$ of degree 2 that ramifies
over $S$.
\end{defn}
\indent
The map $\tau$ is called the hyperelliptic projection for 
${\mathfrak{T}}_{\mathbb{C}}$. One can utilize it to construct 
an exhaustion function with finite charge.\\
\indent
The canonical map $\aleph$ factors through the following
commutative diagram:
\begin{equation}\label{E: aleph}
\begin{tikzcd}
{\mathfrak{T}}_{\mathbb{C}}\arrow{r}{\aleph}\arrow{d}{\tau} &
\mathbb{P} ({\text{Hol}}^* ({\mathfrak{T}}_{\mathbb{C}}))\\
{\mathbb{P}}^1(\mathbb{C})\backslash I \arrow{ur} &
\end{tikzcd}
\end{equation}
\indent
On Riemann surfaces that admit an exhaustion function with finite
charge, there exists a unique canonical basis of the Hilbert space
$\textnormal{Hol}({\mathfrak{T}}_{\mathbb{C}})$ such that
\begin{equation}
\oint_{A_i} {\beta}_j ={\delta}_{ij}.
\end{equation}
\begin{defn}
The Riemann period matrix expressed  in terms of the 
canonical homology cycles and the canonical basis of holomorphic
forms has the $ij$-entry 
$$ {\mathfrak{R}}^{ij} = \oint_{B_i} {\beta}_j.$$
\end{defn}
\noindent
$[{\mathfrak{R}}^{ij}]$ is symmetric, and its imaginary part
$\Im [{\mathfrak{R}}^{ij}]$ is positive-definite. There is
no such luxury as a unique canonical basis on $\text{Hol}
(\mathbb{S})$ for $\mathbb{S} =[(\mathbb{C}\backslash 
\{{\lambda}_i\})\times  (\mathbb{C}\backslash \{{\lambda}_i\})]/_
{\backsim}$, but since the cycles $B_i$ become real line segments,
two line integrals traversed in opposite 
directions with square integrable functions identical on the
two copies of the complex plane result in   
$[{\mathfrak{R}}^{ij}](\mathbb{S}) = [0]$.\\

\subsubsection{Theta functions}
\paragraph{}
Now we tackle the question of existence and naturality for 
theta functions on hyperelliptic Riemann surfaces of infinite
genus treated in the second article by McKean and Trubowitz
~\cite{McT2}. Here we introduce a closed linear subspace
$K \subset I^{\frac{3}{2}}$   containing all integral functions
asymptotically dropping off fast enough to satisfy
\begin{equation}
\sum_{m =1}^{\infty}\frac{|\phi ({\mu}_m)|^2}{({\lambda}_{2m} -
{\lambda}_{2m-1})^2} < \infty,\;\; \forall {\mu}_m \in 
[{\lambda}_{2m} - {\lambda}_{2m-1}].
\end{equation}
\noindent
Such a differential is completely determined by its real periods:
\begin{equation}
{\mathsf{A}}_m (\phi) = \int_{{\lambda}_{2m-1}}^{{\lambda}_{2m}}
\frac{\phi(\lambda)}{\sqrt{{\bigstar}^2(\lambda)-4}} d\lambda,
\end{equation}
\noindent
so that ${\mathsf{A}}_m (\phi) = 0,\;m \geqslant 1$ implies
$\phi =0$ by interpolation. ${\mathsf{B}}_m (\phi)$ are 
dependent on the holomorphic structure at infinity.  \\
\indent
${\mathsf{A}}_m (\cdot)$ are naturally elements of the dual
Hilbert space $I^{\frac{3}{2}*}$ and they furnish an orthogonal
basis. Any individual element of the $l^2$ sequence 
${\phi}^* \in I^{\frac{3}{2}*}$ can be decomposed into 
${\phi}^*_v = \sum a_m^v {\mathsf{A}}_m (\cdot)$. To enter the 
realm of $K^*$, ${\phi}^*$ ought to satisfy $\sum ({\phi}^*_m)^2 
m^2 ({\lambda}_{2m}-{\lambda}_{2m-1})^2< \infty$.\\
\indent
Our next building block is the Hilbert space $H[\phi]$ of differentials 
(not just integral functions) of the transcendental irrationality
~\eqref{E: trans differentials}. The Hilbert space norm is 
\begin{equation}
H[\phi] = i \int_{{\mathfrak{T}}_{\mathbb{C}}} d\Phi \wedge
\overline{d \Phi} < \infty.
\end{equation}
\noindent
Sadly, $H[\phi] \nsubseteq I^{\frac{3}{2}}$. That is why we have 
to clutter the proofs with an extra space. Still, we would denote
the elements of its dual space $H^*[\phi]$ by the same symbol,
${\phi}^* $. It is an absolutely summable sequence of complex
numbers, each pertaining to a real part cycle of some
hyperelliptic Riemann surface of infinite genus. $H[\phi]$ is
endowed with an orhonormal basis ${\mathtt{1}}_j \in K,\;(j 
\geqslant 1)$ such that ${\mathsf{A}}_m ({\mathtt{1}}_j) =
{\delta}^m_j$, and 
$$ \int_{{\mathfrak{T}}_{\mathbb{C}}} d\Phi \wedge
\overline{{\mathtt{1}}_j } = -2{\mathsf{B}}_j (\phi) .$$
\indent
Now, by analogy with the classical case, the theta function,
$\theta$, for the Hill surface  ${\mathfrak{T}}_{\mathbb{C}}$,
is defined for ${\phi}^* = {\phi}^*_{\text{real}} + 
i{\phi}^*_{\text{im}} \in K^* + iH^*[\phi]$ by the formula
\begin{equation}
\theta ({\phi}^*) = \sum e^{2\pi i{\phi}^*(\phi)}
e^{-\frac{\pi}{2}{\phi}^*_{\text{im}}(\phi)},
\end{equation}
\noindent
$\phi = \sum m_j {\mathtt{1}}_j$, for finite sums only to 
ensure convergence. With this expression, the analogy with
the compact Riemann surfaces is complete, as we can state
$$\theta ({\phi}^*+ m_j {\mathsf{A}}_j({\mathtt{1}}_j))=
\theta ({\phi}^*),$$
$$\theta ({\phi}^*+ {\mathsf{B}}_j({\mathtt{1}}_j))=
e^{-2\pi i({\phi}^*({\mathtt{1}}_j) + \frac{1}{2}
{\mathsf{B}}_j({\mathtt{1}}_j))}\theta ({\phi}^*).$$
\indent
An explicit connection between the holomorphic structure on 
hyperelliptic Riemann surfaces of finite genus (stemming from
the Hill operators with a finite number of simple eigenvalues), 
and the elements of the isospectral class per se had been found 
by Its and Matveev~\cite{IM} and later generalized to encompass 
hyperelliptic Riemann surfaces of infinite
genus by McKean and Trubowitz~\cite{McT2}. 
\begin{equation}\label{E: Its-Matveev}
q(\xi) = -2 \frac{{\partial}^2}{\partial {\xi}^2}\log \theta
({\phi}^* + \xi v_1),\;\;\;0 \leqslant\xi<1,
\end{equation}
\noindent
and $v_1$ is expressed in terms of periods (\cite{McT2}, Section 9):
\begin{equation}
v_1 (\phi) = \sum_{j \geqslant 1} 2j\int^{{\lambda}_{2j}}_
{{\lambda}_{2j-1}} d\Phi.
\end{equation}
\noindent
One remarkable property of this equation is that every $q(\xi)$
inside the isospectral class corresponds to its own unique
differential. Uncertainty inherent in the winding numbers of the 
integration contours has been absorbed into the theta function.

\subsubsection{Torelli theorem}
\paragraph{}
The most far-reaching result in hyperelliptic function theory that
we depend on in a crucial proof is the Torelli theorem for 
hyperelliptic Riemann surfaces of infinite genus. Without 
equivocation, Torelli theorem says that the Riemann period matrix
is the one and only invariant of (certain class of) hyperelliptic
Riemann surfaces of infinite genus. This result is applicable
to a broader class of surfaces of infinite genus, so we restate it 
as proven in (\cite{FKT}, Chapter 2, Section 11). We sketch the 
geometric hypotheses (GH1) - (GH6), although exact formulations
take several pages to state fully. The decomposition
into a union of fragments - a compact one with
a finite number of boundary components ${\mathfrak{T}}^{\text{com}}
_{\mathbb{C}}$, a finite number of regular components each 
attached to one boundary component of ${\mathfrak{T}}^{\text{com}}
_{\mathbb{C}}$ denoted ${\mathfrak{T}}^{\text{reg}}_{\mathbb{C}}$,
an infinite number of closed `handles' ${\mathfrak{T}}^{\text{han}}
_{\mathbb{C}}$ is used below.\\

\begin{enumerate}
\item
(GH1) Regular fragments. Informally, the closure of a regular 
fragment is biholomorphic to a complex plane minus an open 
neighborhood of a discrete set.
\item
(GH2) Handles. Some restrictions are placed on  the deformations
of cylinders that constitute handles.
\item
(GH3) Gluing ${\mathfrak{T}}^{\text{reg}}_{\mathbb{C}}$ and
${\mathfrak{T}}^{\text{han}}_{\mathbb{C}}$. It deals with the 
possible overlaps of two handles. Some regularity conditions are
imposed.
\item
(GH4) Gluing in the compact fragment. Regularity of the gluing
biholomorphic maps so as not to disturb the canonical homology
cycles forming the basis of $H_1({\mathfrak{T}}^{\text{com}}
_{\mathbb{C}}, \;\mathbb{Z})$ to induce an inclusion into
$H_1({\mathfrak{T}}_{\mathbb{C}}, \;\mathbb{Z})$.
\item
(GH5) Estimates on the Gluing maps. The handles are to be
separated, and their density at infinity is to remain bounded.
\item
(GH6) The discrete set of points used to guide the attachment
of handles has to be distributed in a particular way.
\end{enumerate}
\noindent
The first four hypotheses are topological in nature. The
estimates in (GH5) control the holomorphic structure of 
${\mathfrak{T}}_{\mathbb{C}}$.
\begin{backthm}\label{T: Torelli}\textnormal{(Torelli)}
Let ${\mathfrak{T}} = {\mathfrak{T}}^{\text{com}}_{\mathbb{C}} \cup
{\mathfrak{T}}^{\text{reg}}_{\mathbb{C}} \cup {\mathfrak{T}}
^{\text{han}}_{\mathbb{C}}$ and\\ ${\mathfrak{T}'_{\mathbb{C}}} = 
{\mathfrak{T}'}^{\text{com}}_{\mathbb{C}} \cup
{\mathfrak{T}'}^{\text{reg}}_{\mathbb{C}} \cup {\mathfrak{T}'}^
{\text{han}}_{\mathbb{C}}$ be Riemann surfaces that fulfill 
the hypotheses (GH1)-(GH6). Denote 
their canonical homology bases by $A_1, B_1, A_2, B_2, \cdots$ 
and\\ $A'_1, B'_1, A'_2, B'_2, \cdots$. Let $[{\mathfrak{R}}^{ij}]$,
respectively $ [{\mathfrak{R}}^{'ij}]$ be the associated period
matrices. If $[{\mathfrak{R}}^{ij}] = [{\mathfrak{R}}^{'ij}]$ for
all $i,j \in \mathbb{Z}$, then there is a biholomorphic map
$BH: {\mathfrak{T}}_{\mathbb{C}} \longrightarrow {\mathfrak{T}'}
_{\mathbb{C}}$ and 
$\epsilon \in \{\pm 1\}$ such that for all $j \in \mathbb{Z}$
$$BH_* (A_j) = \epsilon A'_j, \;\; BH_* (B_j) = \epsilon B'_j.$$
\end{backthm}
\indent 
Trubowitz et al. (\cite{FKT}, Section 12) verify that Hill surfaces 
satisfy the hypotheses (GH1)-(GH6). In general, not all
hyperelliptic surfaces do. Another well-behaved kind of Riemann
surfaces of infinite genus are complexified Fermi curves that
serve as spectral curves of the two-dimensional Schr\H{o}dinger
operator with doubly periodic potentials discussed at length
in (\cite{FKT}, Section 16).

\subsubsection{Regularized standard potentials}
Fix an open set $U \subset M$. Next, consider
$\partial U$ - a smooth closed hypersurface, and via Implicit function 
theorem set $x_1^{-}(x_2, ...,x_{2k})$ and $x_1^{+}(x_2, ...,x_{2k})$
as $x_1$- coordinates of the hypersurface $\partial U$. On the space
of real-valued measurable functions on $U$, define the inner product
\begin{equation*}
< f, h>_{U} \;\overset{\textnormal{def}}{=} \; 
\frac{1}{\text{vol}(U)}\int_Ufh\; dx_1.....dx_{2k}.
\end{equation*}
The associated norm would be
\begin{equation*}
||f||_{U} \overset{\textnormal{def}}{=}\sqrt{
\frac{1}{\text{vol}(U)}\int_Uf^2 dx_1.....dx_{2k}}.
\end{equation*}
This is just $L^2(U)$. From that, we extract all measurable functions
periodic in $x_1$ with periods $\frac{2\pi}{x_1^{+}(x_2, ...,x_{2k}) -
x_1^{-}(x_2, ...,x_{2k})}$. We denote this closed linear subspace 
$L^2([x_1^{-}, x_1^{+}] \times (U/x_1))$. It allows a partial inner
product:
\begin{equation*}
< f, h>_{x_1} \;\overset{\textnormal{def}}{=} 
\;\frac{1}{x_1^{+}(x_2, ...,x_{2k}) -
x_1^{-}(x_2, ...,x_{2k})}{\int_{{x_1^{-}}}^{x_1^{+}}}fh \;dx_1,
\end{equation*}
\noindent
and a partial norm\\
\begin{equation*}
||f||_{x_1} \;\overset{\textnormal{def}}{=} 
\;\sqrt{\frac{1}{x_1^{+}(x_2, ...,x_{2k}) -
x_1^{-}(x_2, ...,x_{2k})}{\int_{{x_1^{-}}}^{x_1^{+}}}f^2 \;dx_1},
\end{equation*}
We interpret $L^2([x_1^{-}, x_1^{+}] \times (U/x_1))$ as a family
of Hilbert spaces of periodic functions parameterized by $(x_2, ...,x_{2k})$.
Within that family, there is a subfamily of smooth periodic functions of $x_1$
smoothly parameterized by $(x_2, ...,x_{2k})$. We 
call that subfamily $C^{\infty}([x_1^{-}, x_1^{+}] \times (U/x_1))$. It
is a closed linear subspace: \\
$C^{\infty}([x_1^{-}, x_1^{+}] \times (U/x_1))\subsetneq C^{\infty}(U)$.\\
\indent
For a fixed metric coefficient, $g^{vw}$, we run through all ${\kappa}_{vw}$
satisfying~\eqref{E: Laplace}, and all the allowable kernels of the smoothing
operator to produce all the attendant periodic potentials.\\ 
\indent
The equation below is set for
$U^N_{vw} \cap U = \{x_1=x_1^{\textnormal{blow}}\}$, and
$U^O_{vw} \cap U = \emptyset$. If there are multiple hypersurfaces, 
some zeroes, or both, then ${\lambda}_0 \mapsto {\lambda}_{2s}$, and to
determine $s$ we count the number of roots in the product
$\sin g_{ij}^- \sin g_{ij}^+$ as $x_1$ traverses the length of $U$ through
$U^N_{ij} \cap U$ and on to $U^O_{ij} \cap U$. This procedure is in agreement
with numbering of the periodic spectra in~(\cite{McT}, Section 1).\\
\begin{equation}\label{E: potential}
[-\frac{{\partial}^2}{\partial x_1^2} + \cos (\frac{x_1}{T_{vw}})]{\sin} 
({\kappa}_{vw}) = {\lambda}_0 {\sin} ({\kappa}_{vw});\\
\end{equation}
\noindent
As it stands, there are two unknown functions: ${\lambda}_0(x_2,...,x_{2k})$,
and $T_{vw}(x_1,...,x_{2k})$. However,
\eqref{E: potential} is informed by the observation made in (\cite{McT},
Amplification 1 of Section 7), according to which
${\nu}_s = {\lambda}_{2s}$ if and only if the potential is an even
periodic function. Thus $\cos (\frac{x_1}{T_{vw}})$ is a correct choice.
We know that potentials of the form $\cos (\frac{x_1}{T_{vw}})$ have purely
simple periodic spectrum for almost all $T$. Hence if ${\kappa} = 
\frac{x_1}{x_1^{+}(x_2, ...,x_{2k}) - x_1^{-}(x_2, ...,x_{2k})}$ , 
our potential must be constant by Background Theorem~\ref{T: Magnus
coexist}. Indeed, a sine polynomial cannot be an eigenfunction
of the standard cosine potential. Thus in that special case, the function 
$T_{vw}=cx_1$. At the other extreme, as $\kappa \rightarrow \frac{\pi}{4}$
with the onset of regularity, $\cos (\frac{x_1}{T_{vw}}) \rightarrow 1$,
and once again we arrive at $T_{vw}=cx_1$ with $c=\frac{1}{2\pi}$. Otherwise, we obtain 
a bona fide trigonometric potential, and $T_{vw}(x_2, ...,x_{2k})$ becomes a  
nonvanishing function. The second equation is given by~\eqref{E: Hill solutions}
modified for $C^{\infty}([x_1^{-}, x_1^{+}] \times (U/x_1))$:\\
\begin{align}\label{E: lowest eigenvalue}
&{\sin}({\kappa}_{vw}(\xi, x_2,...,x_{2k}))
= \frac{\sqrt{2}\cos(\sqrt{{\lambda}_0}\xi)}{2}\\\notag
& + \frac{1}{x_1^+ -x_1^-}\int^{\xi}_{x_1^-} 
\frac{\sin (\sqrt{\lambda}_0(\xi - \eta))}
{\sqrt{{\lambda}_0}}\cos(\frac{\eta}{T_{vw}})
{\sin}({\kappa}_{vw}(\eta, x_2,...,x_{2k})) d\eta. \notag
\end{align}
\noindent
Now~\eqref{E: potential} in conjunction with~\eqref{E: lowest eigenvalue}
determines a potential and the lowest eigenvalue uniquely, though not
explicitly. With the potential comes the entire periodic spectrum:\\
\begin{equation}\label{E: my eigenvalues}
[-\frac{{\partial}^2}{\partial x_1^2} + \cos (\frac{x_1}{T_{vw}})]f_s =
{\lambda}_s f_s;\;\; s \in \mathbb{Z}, s \geqslant 0,\\
\end{equation}
\noindent
such that $f_s(x_1^-, x_2,...,x_{2k}) = f_s(x_1^+, x_2,...,x_{2k})$. Now we
select the potential whose periodic spectrum is located as close to that
of some constant function as possible, i.e. as close to a spectrum
allocated for regular configurations as possible. Those spectra are purely double, 
hence we seek a pair $({\lambda}_0(x_2,...,x_{2k}),\;T_{vw}(x_1,...,x_{2k}))$ that 
solve~\eqref{E: potential},~\eqref{E: lowest eigenvalue} and produce a spectrum
~\eqref{E: my eigenvalues} attaining the following limit: 
\begin{equation}\label{E: optimized potential}
\lim_{N \rightarrow \infty} \liminf_{(S_{\zeta},\;\kappa)} 
(||{\lambda}_{0}^{\text{c}} - 
{\lambda}_{0}^{}||^2_U  +\sum_{s=1}^{N}\frac{1}{2^s} (||{\lambda}_{2s}^{\text{c}} - 
{\lambda}_{2s-1}^{}||^2_U + ||{\lambda}_{2s}^{\text{c}} - {\lambda}_{2s}^{}||^2_U)).
\end{equation}
\noindent
Above, ${\lambda}_{2s}^{\text{c}}$'s stand for the periodic eigenvalues
of the constant function $q(x_1) \equiv 1$. Using~\eqref{E: lambda double},
without loss of generality, we may set
\begin{equation}
{\lambda}_{2s}^{\text{c}} = \frac{s^2 {\pi}^2}{x_1^{+}(x_2, ...,x_{2k}) -
x_1^{-}(x_2, ...,x_{2k})} +1.
\end{equation}
\noindent
Our potential possesses a purely simple spectrum that is as close to a 
purely double spectrum in the $L^2(U)$-norm as the metric data would allow. 
Our objective is to jettison extraneous ${\kappa}(x_1,...,x_{2k})$'s. 
For this reason, we refer to the Hill potentials satisfying~\eqref{E: 
optimized potential} as \textit{optimized potentials}.\\  
\indent
However, the optimized potentials are still not unique, since by Borg's
theorem (\cite{McT}, section 2), only the periodic spectrum, reflecting
spectrum, and the normalizations of eigenfuctions together determine
the potential uniquely. Consequently, we are compelled to consider all
potentials with this spectrum.\\
\indent
At this point we have gathered enough information to introduce our main object:
the isospectral class 'manifold' (almost)
\begin{equation}\label{E: isospectral class}
{\mathfrak{T}}(U) \overset{\textnormal{def}}{=} 
\{q \in C^{\infty}([x_1^{-}, x_1^{+}] \times (U/x_1))|\; 
{\lambda}^q_{2s} = {\lambda}_{2s},\;{\lambda}^q_{2s-1}={\lambda}_{2s-1}\}.
\end{equation}
\noindent
Here ${\lambda}^q_{2s}$, ${\lambda}^q_{2s-1}$ stand for the periodic
eigenvalues of the Hill operator with potential $q$, and
${\lambda}_{2s}$, ${\lambda}_{2s-1}$ are given by~\eqref{E: my eigenvalues}
attaining~\eqref{E: optimized potential}. It is naturally a fibration over
the foliation of codimension $2k-1$ by the integral curves of 
$\frac{\partial}{\partial x_1}$. The $L^2$ norm $||\cdot||_{x_1}$ agrees
with the foliation downstairs. Hence there is a linear subspace
$T{\mathfrak{T}}(U)$. It happens to be closed due to Background 
Theorem~\ref{T: McKean basis}, even though the dimension jumps from 0 at
${\lambda}^q_{2s-1}={\lambda}^q_{2s}$ all the way to $\infty$ at
the leaves with ${\lambda}^q_{2s-1}<{\lambda}^q_{2s}$. Thus it is possible
to introduce bounded linear operators having both the domain and the range
within $T{\mathfrak{T}}(U)$. They just have to slim down to the zero
operator over the leaves with double spectra. All the same, the typical
${\mathfrak{T}}(U)$ does not stray away from uniform infinite-dimensional 
tori. A sufficient condition for uniformity is that there exists at least
one hypersurface $U_{ij}^N$ such that the whole of $U$ can be split into
the backward and forward $x_1$-developments of the locus of 
nondifferentiability via the flow of $\frac{\partial}{\partial x_1}$:
$U = (U_{ij}^N)^{\swarrow} \cup U_{ij}^N \cup (U_{ij}^N)^{\nearrow}$.  
These observations pave the way to extend the action of the 
isometry group to this isospectral class.

\subsubsection{Isometric invariance}

\begin{thm}\label{T: spectrum}
The symmetry group of the nondifferentiable (pseudo-)metric $g$ preserves
the periodic spectrum of the Hill operator with optimized potentials
of the form $\cos (\frac{x_1}{T}).$\footnote{The last qualifying hypothesis is
unnecessary, as the gist of Theorem~\ref{T: spectrum} holds for all optimized potentials, 
but that is all we need in the sequel, and the proof is streamlined as a consequence.}
\end{thm}
\begin{proof}
\indent
On  $C^{\infty}(U)$ we define rotation operators $K_{ab}$
stemming from the (smooth) Killing vector fields on $U$. Their respective
flows $\exp(\varepsilon K_{ab})$ preserve the blow-up parameter just as local
Killing vector fields preserve the locus of nondifferentiability. Thus in 
terms of brackets, $[{\frac{\partial}{\partial x_1}}, K_{ab}] =0$. They
act on 
$C^{\infty}([x_1^{-}, x_1^{+}] \times (U/x_1))\subsetneq C^{\infty}(U)$,
and the subspace 
$\{K_{ab}f|f \in C^{\infty}([x_1^{-}, x_1^{+}] \times (U/x_1));\; a \ne b; \;1 < a,b 
\leqslant 2k \}$ is closed in $C^{\infty}([x_1^{-}, x_1^{+}] \times (U/x_1))$ in 
the standard $C^{\infty}-$topology. Via Background Theorem~\ref{T: McKean basis} 
adapted to $L^2([x_1^{-}, x_1^{+}] \times (U/x_1))$, we have for some 
$s \in \mathbb{Z}$ (possibly after a permutation of basis eigenfunctions)
\begin{equation}
[1 + \varepsilon K_{ab} + O({\varepsilon}^2)] f_{2s}^2= h f_{2s}^2 + 
z\frac{\partial}{\partial x_1}f_{2s}^2 + O({\varepsilon}^2);
\end{equation}
\begin{equation}
[1 + \varepsilon K_{ab} + O({\varepsilon}^2)] 
\frac{\partial}{\partial x_1}f_{2s}^2= \hat{h}\frac{\partial}{\partial x_1}f_{2s}^2 + 
\hat{z}f_{2s}^2 + O({\varepsilon}^2).
\end{equation}
\noindent
Here $h, z, \hat{h}, \hat{z}$ are smooth functions of $(x_2, ..., x_{2k})$. If 
$z=\hat{z} \equiv 0$ for all pairs $a \ne b$, we are done as 
$\exp(\varepsilon K_{ab}) (T{\mathfrak{T}}_{\cos (\frac{x_1}{T_{ij}})}) \subset 
T{\mathfrak{T}}_{\cos (\frac{x_1}{T_{ij}})}, \;\; 
\exp(\varepsilon K_{ab}) (N{\mathfrak{T}}_{\cos (\frac{x_1}{T_{ij}}}) \subset 
N{\mathfrak{T}}_{\cos (\frac{x_1}{T_{ij})})}.$ Otherwise, $\varepsilon K_{ab}
\dda f_{2s}^2= \hat{z}f_{2s}^2$ on an 
open subset of $U$. Hence $z \ne 0, \hat{z} \ne 0$ since $K_{ab}$ is 
a rotation. Therefore we obtain\\
\begin{equation}
\hat{z}f_{2s}^2 = \varepsilon K_{ab}\dda f_{2s}^2
= \dda \varepsilon K_{ab}f_{2s}^2 =\dda z \dda f_{2s}^2.
\end{equation}
\noindent
The resulting ODE is\\
\begin{equation}\label{E: harmonic oscillator}
\frac{{\partial}^2}{\partial x_1^2}f_{2s}^2= \frac{\hat{z}}{z}f_{2s}^2.
\end{equation}
\noindent
Now if $\frac{\hat{z}}{z}>0$, the solutions of~\eqref{E: harmonic 
oscillator} would be unbounded, so our only choice is to force
$\frac{\hat{z}}{z}<0$. And this is a regular harmonic oscillator.
Further constrains whittle the coefficients down to\\
\begin{equation}
\sqrt{|\frac{\hat{z}}{z}|} = \frac{2\pi m}{x_1^+ -x_1^-},\;\; m \in 
\mathbb{Z}.
\end{equation} 
\noindent
Thus the only possible solutions of~\eqref{E: harmonic oscillator} are 
linear combinations of
\begin{equation}
f^2_{2s} = \cos (\frac{2\pi m}{x_1^+ -x_1^-}x_1),\;\textnormal{and}
\;\;f^2_{2s} = \sin (\frac{2\pi m}{x_1^+ -x_1^-}x_1).
\end{equation}
\noindent
But on the open subset $\{0 < x_1 < \frac{x_1^+ -x_1^-}{4m}\}$
their square roots\\
\begin{equation}
 \sqrt[4]{1 -{\sin}^2(\frac{2\pi m}{x_1^+ -x_1^-}x_1)},\;\;\;
 \sqrt[4]{1-{\cos}^2 (\frac{2\pi m}{x_1^+ -x_1^-}x_1)},
\end{equation}
\noindent 
are not the eigenfunctions of the original Hill operator. We have reached 
a contradiction assuming that the flow of $K_{ab}$ shifts the periodic
spectrum.
\end{proof}
\noindent
To recapitulate the conclusion of Theorem~\ref{T: spectrum}: given
a potential of the form $\cos(\frac{x_1}{T})$ obtained via the 
Laplace-Beltrami $\longrightarrow$ Nash-Gromov $\longrightarrow$
Hill equation procedure, and further selected via~\eqref{E: optimized
potential}, its simple periodic spectrum is comprised of 
analytic functions of elementary symmetric polynomials in the
$(x_2,...,x_{2k})$ variables.\\
\indent
Having established a general pattern for the action of rotation
operators, we proceed to delineate their modus operandi:
\begin{thm}\label{T: representation}
There exists a faithful representation of $\exp tK_{ab}$ on the 
infinite-dimensional isospectral class $\mathfrak{T}(U)$ by the Hamiltonian 
flows constructed by McKean and Trubowitz.
\end{thm}
\noindent
\begin{proof}
By Theorem~\ref{T: spectrum}, $K_{ab}q \subset T{\mathfrak{T}}_q(U)$. 
According to Background Theorem~\ref{T: McKean basis},
its tangent space is spanned by $\dda f^2_{2s}$, $s \in \mathbb{Z},\;s \geqslant 1$.
Therefore, we formally equate\\
\begin{equation}\label{E: T-vectors decomposition}
K_{ab}q(x_1,...,x_{2k}) = \sum_s u_{2s}(x_2,...,x_{2k}) \dda f^2_{2s}.
\end{equation}
\noindent
$u_{2s}(x_2,...,x_{2k})\ne 0$ as $s \rightarrow \infty$ generically, and the question of 
convergence ought to be addressed. Even though $\exp tK_{ab}q \in C^{\infty}(U)$, that
fact alone does not guarantee the existence of a decomposition into orthonormal tangent 
vectors. To compound our predicament, $K_{ab}$ is not a self-adjoint operator. To show
convergence, we embed $U \hookrightarrow {\mathbb{R}}^{2k}$ isometrically, and
consider some appropriate Sobolev space of compactly supported functions. Clearly,
the image of $C^{\infty}(U)$ belongs inside. That allows for the Fourier transforms
$F(q)$ to be employed. Then by Parseval's formula, as well as by the basic properties
of differential operators we have (here $\Upsilon(K_{ab})$ denotes a unitary
representation of isometries in the frequency domain realized by Hilbert-Schmidt
operators):\\
\begin{alignat}{2}
||q||_{\textnormal{Sobolev}} &= ||F(q)||_{\textnormal{Sobolev}} \;\;\text{(by Parseval's
formula)}\\ 
=||\Upsilon (K_{ab}) F(q)||_{\textnormal{Sobolev}} &= 
||F(K_{ab}q)||_{\textnormal{Sobolev}}\;\;\text{(by unitarity)}\notag\\
&=||K_{ab}q||_{\textnormal{Sobolev}}\;\;\text{(by Parseval's formula)}.\notag 
\end{alignat}
\noindent
At this point we invoke the Sobolev embedding theorem~(\cite{Heb}, Chapter 2.
Section 2.3) to claim that $||K_{ab}q||_U \leqslant c||K_{ab}q)||_{\textnormal{Sobolev}}$.
The norm on the left is the Hilbert space norm on $U$. That bound implies $tK_{ab}$,
possibly after normalization, is a generator of a one-parameter contraction semigroup 
on $C^{\infty}([x_1^{-}, x_1^{+}] \times (U/x_1))\subsetneq C^{\infty}(U)$, provided
the normalization constant $c_K$ is so chosen as to satisfy 
$\max_{\overline{U/x_1}}||c_K K_{ab}q||_{x_1}
\leqslant ||K_{ab}q||_U$. Then $u^2_{2s} \leqslant c_u 2^{-s}$ for every $s > s^K$
in~\eqref{E: T-vectors decomposition} uniformly in $U$. \\
\indent
Now by (\cite{McT}, Amplification 1 of Section 14), we have
\begin{equation}
Xq = \sum_{s=1} (-\frac{\partial}{\partial \lambda}\bigstar)({\lambda}_{2s})
q_{2s} \dda f^2_{2s},
\end{equation}
\noindent
for all $q \in C^{\infty}([x_1^{-}, x_1^{+}] \times (U/x_1))$. By transitivity
of the flows on $\mathfrak{T}(U)$, every such decomposition is
a Hamiltonian flow, and to produce
the desired representation we just identify $q_{2s}(x_2,...,x_{2k}) =u_{2s}
(x_2,...,x_{2k})(-\frac{\partial}{\partial \lambda}\bigstar)^{-1}({\lambda}_{2s})$.
\end{proof}
\indent
The upshot of the last result is that given a one-parameter subgroup of the isometry
group, denoted by $\exp tK_{ab}$, acting linearly on the coordinates (except $x_1$), 
we have the following diagram (its commutativity is a direct consequence of the 
spectrum being fixed by the limiting process~\eqref{E: optimized potential}):\\
\begin{equation}\label{E: linear action}
\begin{tikzcd}
y_{1,2}(q,{\lambda}_s)  \arrow{r}{\cong} & 
y_{1,2}(\exp tK_{ab}q, {\lambda}_s) \\
q(x_1,...,x_{2k}) \arrow{r} \arrow{u}{\text{Hill operator}}
 & \exp tK_{ab}q(x_1,...,x_{2k}) \arrow{u}{\text{Hill operator}}\\
\end{tikzcd}
\end{equation}
The action on the potentials is (in spite of notation) by Hamiltonian flows.
This paves the way for Cartan's method to go through, specifically, it facilitates
reduction to a unique coframe. We were compelled to introduce $\mathfrak{T}(U)$
as a way to accommodate the regularized metric coefficients, their multitude
stemming from the Laplace-Beltrami equation and Nash-Gromov smoothing. That 
may necessitate a prolongation.\\

\subsubsection{Invariant Riemann period matrix}
Our last step is to repackage the information pertaining to the
Laplace-Beltrami $\longrightarrow$ Nash-Gromov $\longrightarrow$
Hill equation procedure into a more appealing device. As the
periodic spectrum is comprised of analytic functions of 
elementary symmetric polynomials of $(x_2,...,x_{2k})$, one would expect
the same to hold true for the Riemann period matrix. However, the difficulty 
here is to establish preservation of holomorphic structures at each point 
of $U/x_1$. They stem from hyperelliptic Riemann surfaces of infinite 
genus over the (tubular neighborhood of) locus of nondifferentiability, 
possibly comingling with  adjoined complex planes identified along purely
double eigenvalues, corresponding to the constant potential. In general,
the fibration of Riemann surfaces over $U/x_1$, denoted ${\mathfrak{T}}_
{\mathbb{C}}(U)$, would contain those adjoined complex planes:
\begin{equation}
[(\mathbb{C}\backslash \{{\lambda}_i\})\times 
(\mathbb{C}\backslash \{{\lambda}_i\})]/_{\backsim} \subset
{\mathfrak{T}}_{\mathbb{C}}(U).
\end{equation} 
\noindent
The period maps are still well-defined for those planes, with all the
periods shrinking to zero. However the Hilbert spaces of holomorphic forms
get depleted due to the reduction in the number of zeroes available to
holomorphic forms. Nevertheless, we state
\begin{lem}\label{T: period matrix invariance}
The Riemann period matrices of ${\mathfrak{T}}_{\mathbb{C}}(U)$ corresponding 
to the periodic spectra of the Hill potentials~\eqref{E: potential},
~\eqref{E: lowest eigenvalue},
~\eqref{E: optimized potential} are local metric invariants.
\end{lem}
\begin{proof}
The action of $\exp t K_{ab}$ on individual infinite-dimensional tori
via the Hamiltonian flows established by Theorem~\ref{T: representation}
trivially extends to the case where the torus collapses, $({\lambda}_{2i}
- {\lambda}_{2i-1}) \longrightarrow 0$, and, as a result,
${\dim}_{\mathbb{R}} T\mathfrak{T}=0$, and the action becomes trivial as
there is literally no direction to flow in, and this `inaction' persists
on $[(\mathbb{C}\backslash \{{\lambda}_i\})\times 
(\mathbb{C}\backslash \{{\lambda}_i\})]/_{\backsim}$.\\
\indent
On an individual proper infinite-dimensional torus, the action extends
to its hyperelliptic Riemann surface via the Its-Matveev formula
~\eqref{E: Its-Matveev}:
\begin{equation}
\exp t K_{ab}q(\xi) = -2 \frac{{\partial}^2}{\partial {\xi}^2}\log \theta
(\exp t K_{ab}{\phi}^* +  \xi (\exp t K_{ab} v_1)),\;\;\;0 \leqslant\xi<1,
\end{equation}
\noindent
where 
\begin{equation}
\exp t K_{ab} v_1 = \sum_i^{\infty} \int^{{\lambda}_{2i}}_{{\lambda}_{2i-1}}
\frac{\exp t K_{ab} \phi}{\sqrt{{\bigstar}^2(\lambda)-4}}d\lambda,
\end{equation}
\noindent
and $\exp t K_{ab} {\phi}^*,\;\; {\phi}^* \in H^*[\phi]$ 
is defined via the duality pairing:
\begin{equation}
\sum^{\infty}\phi({\lambda}_{2i})\exp t K_{ab}m_i{\mathsf{A}}_i\;\; 
\overset{\textnormal{def}}{=}
\;\;\sum^{\infty}m_i{\mathsf{A}}_i(\exp t K_{ab}\phi)({\lambda}_{2i}).
\end{equation}
\noindent
Since the Hilbert space of differentials is complete, $\exp t K_{ab}$,
acting via Hilbert-Schmidt integral operators, just reshuffles them, 
and the underlying holomorphic structure remains intact.\\
\indent
The only property left to establish is that the evolution of holomorphic
structures from hyperelliptic Riemann surfaces $(g=\infty)$ to the conjoined
complex planes $[(\mathbb{C}\backslash \{{\lambda}_i\})\times 
(\mathbb{C}\backslash \{{\lambda}_i\})]/_{\backsim}$ is smooth despite the 
genus jumping from infinity to zero. To overcome the genus problem, we resort 
to working with universal 2-fold covers in place of Riemann surfaces, and 
generalize the map $\aleph$ of~\eqref{E: aleph} to encompass all the Riemann 
surfaces as well as the conjoined complex planes. 
\begin{equation}
\begin{tikzcd}
\widetilde{{\mathfrak{T}}}_{\mathbb{C}}(U)\arrow{r}{\aleph (U)}\arrow{d} &
\mathbb{P} ({\text{Hol}}^* (\widetilde{{\mathfrak{T}}}_{\mathbb{C}}(U)))\\
{\mathbb{P}}^1({\mathbb{C}}^{2k-1})\backslash I(U) \arrow{ur} &
\end{tikzcd}
\end{equation}
\noindent
All elements in $[{\phi}^*]\in \mathbb{P}({\text{Hol}}^*
(\widetilde{{\mathfrak{T}}}))$ share the same zeroes. 
\noindent
The evolution of the covering spaces embedded into 
${\mathbb{P}}^1({\mathbb{C}}^{2k-1})\backslash I(U)$ is smooth, 
and by the universal property, ${\aleph (U)}$ 
factors through. Finally, the quotient map 
$$ {\text{Hol}}^* (\widetilde{{\mathfrak{T}}}_{\mathbb{C}}(U))
\longrightarrow {\text{Hol}}^* ({\mathfrak{T}}_{\mathbb{C}}(U))$$
retains that smoothness. 
\end{proof}
\indent
To summarize, we proved the following equalities:
\begin{alignat}{2}
[[{\mathfrak{R}}^{ij}]_{\beta \gamma}](x_2,...,x_{2k})
=&\exp tK_{ab}[[{\mathfrak{R}}^{ij}]_{\beta \gamma}](x_2,...,x_{2k})\\\notag
=&[[{\mathfrak{R}}^{ij}]_{\beta \gamma}]( \exp tK_{ab}(x_2,...,x_{2k}))\notag.
\end{alignat}

\pagebreak
\subsection{The sheaf $\Xi (U, ds^2)$}
\paragraph{}
Given the following data: an open set $U \subset M$, endowed with
closed loci of nondifferentiability $\{U^N_{ij}\}$, and rank deficiency
$\{U^O_{ij}\}$, a line element
$ds^2$ corresponding to a fixed metric\footnote{We prefer to 
specify the line element as the metric is essentially an equivalence 
class of expressions defined up to a diffeomorphism, whereas the line 
element is less elusive.} satisfying our requirements listed in Section 3.1,
and the attendant Riemann period matrices 
$[[{\mathfrak{R}}^{ij}]_{\beta \gamma}]$, we define $\Xi (U, ds^2) 
\subset \Gamma (TU \oplus T^*U)$ to be the sheaf of smooth $2k$-tuples of 
sections of the standard Courant algebroid that generate that particular 
set of Riemann period matrices on $U$. For instance, for $A_{ij}dx_j - 
B_{ij}{\ddx}_{j+k}$, such that $A_{ij}|_{U^N_{ij}} = 0$, we set 
(with appropriate boundary values)
$$\frac{A_{ij}}{\sqrt{A^2_{ij} + B^2_{ij}}} = \sin {\kappa}_{ij},
\;\;e^{2w_{ij}}= {A^2_{ij} + B^2_{ij}},$$
and use $\sin {\kappa}_{ij}$ to extract a Hill potential.\\
\indent
For every such $k$-tuple, its image
under $({\omega}^{\#}_U - {\pi}^{\#}_U)$ would also be admissible.
Together they comprise a basis of $(TU \oplus T^*U)$. The regular case
is trivially included as the zero Riemann period matrix could be a part 
of the data set and $\sin {\kappa}_{ij} \equiv \frac{\sqrt{2}}{2}$. \\
\indent
All these local elements share the same `maternal' $\mathfrak{T}(U)$. 
On the overlaps, the transition functions give rise to conformal 
transformations $F_{ij}:\;\; \mathfrak{T}(U_i \cap U_j) \longrightarrow
\mathfrak{T}(U_j \cap U_i)$. 
One particular feature of our sheaf deserves mentioning. Namely,
it is not microflexible in the sense of Gromov (\cite{Grom},
Section 2.2). 

\subsection{The algebroid $\mathfrak{Aut}(\Xi (U, ds^2))$}
\paragraph{}
Our sheaf is associated with the fixed line element $ds^2$, described by 
the set of Riemann period matrices $[[{\mathfrak{R}}^{ij}]_{\beta \gamma}]
(x_2,...,x_{2k})$. This set is an invariant of the direct product of
infinite-dimensional tori $\bigotimes_{(i,j)(ds^2)}{\mathfrak{T}}_{ij}(U)$
parameterized by the base variables. 
Hence the symmetries there descend onto $\Xi (U, ds^2)$. Precisely,\\
$$\mathfrak{Aut}(\Xi (U, ds^2))= 
\bigotimes_{(i,j)(ds^2)}{\mathbf{j}}_{ij}(t_{ab}K_{ab})$$
\noindent
can be delineated as follows: for each family of eigenfunctions obtained
via the commutative diagram~\eqref{E: linear action}, we compute its 
$t_{ab}$-dependent Hamiltonian vector field 
$H_{\exp t_{ab}K_{ab}(e^{w_{ij}}\cos {\kappa}_{ij})}$ 
encoding  auxilliary scalar fields in congruence with
the original action by Hamiltonian flows $\cos (\exp t_{ab}K_{ab}
{\kappa}_{ij})$ of McKean and Trubowitz. This congruence is based on
the Laplace-Beltrami equation:
\begin{equation} 
{{\Delta}_g}(e^{\exp t_{ab}K_{ab}w_{uv}}_U
\sin({\exp t_{ab}K_{ab}\kappa}_{uv})_U)^{10}=0,\;\;\forall t_{ab}.
\end{equation} 
\noindent
Then the vector field ${\mathbf{j}}_{ij}(t_{ab}K_{ab})$ solves
\begin{equation}
[{\mathbf{j}}_{ij}(t_{ab}K_{ab}), H_{e^{w_{ij}}\cos {\kappa}_{ij}}] = 
\frac{\partial} {\partial t_{ab}} H_{\exp t_{ab}K_{ab}(e^{w_{ij}}
\cos {\kappa}_{ij})}.
\end{equation} 
\noindent
Thus $\bigotimes_{(i,j)(ds^2)}{\mathbf{j}}_{ij}(t_{ab}K_{ab})$ is a local
subalgebra of $\Gamma ({\mathbb{R}}^{\dim \{K_{ab}\}} \times TM)$.
Its action on ${\Psi}_i = \sum_{j=1} (e^{w_{ij}}
\cos{{\kappa}}_{ij}dx_j + e^{w_{ij}}
\sin{{\kappa}}_{ij}dx_{2k+j}),\;\;{\Psi}_i \in \Xi (U, ds^2)$ is
\begin{align}
&\bigotimes_{(i,j)(ds^2)}{\mathbf{j}}_{ij}(t_{ab}K_{ab}){\Psi}_i \\\notag 
=&\sum_{j=1}{\mathbf{j}}_{ij}(t_{ab}K_{ab}) (e^{w_{ij}}
\cos{{\kappa}}_{ij}dx_j + e^{w_{ij}}\sin{{\kappa}}_{ij}dx_{2k+j})\\\notag
=& \sum_{j=1}\exp t_{ab}K_{ab} (e^{w_{ij}}\cos{\kappa}_{ij})dx_j + 
\exp t_{ab}K_{ab}(e^{w_{ij}}\sin{\kappa}_{ij})dx_{2k+j}.\notag
\end{align}
\indent
We have traded nondifferentiability/rank deficiency of the metric
for families of solutions. That is the long and short of our
regularization procedure. Their totality has now to be preserved in order 
for us to be able to reverse-engineer the original metric.\\

\section{Equivalence}
\paragraph{}
From this section on, we relabel independent variables on $T_UU \subset
TT^*M$ to simplify notation. Set $ {\dot{x}}_j =x_{j+2k}, j \leqslant 2k$.
The canonical symplectic form is written as $\Omega = \sum_j dx_j \wedge dx_{j+2k}$.

\subsection{Lifting of preframes to $T^*TM$}
\paragraph{}
Now is the moment to take full advantage of the symplectic connection
introduced in subsection 2.1.3. Given a preframe and its complement
under  $({\omega}^{\#}_U - {\pi}^{\#}_U)$, we produce a smooth frame
on $T_UU \subset TT^*M$ via $\ccc$. To produce an equivalent coframe,
we invoke the symplectomorphism of Abraham and Marsden (\cite{AM}, 
chapter 2):
$$\boldsymbol{\alpha}: TT^*M \longrightarrow T^*TM.$$ 
\noindent
From now on, by a slight misnomer, we will tacitly 
assume $T_UU \subset T^*TM$.\\
\indent
Having set up the coframes, we state our objective: to formulate 
the necessary and sufficient conditions for the existence of local
diffeomorphisms (elements of the transitive pseudogroup of 
diffeomorphisms  $\text{Diff} \;T_UU \subset \text{Diff}\; T^*M$) 
that preserve the Riemann period matrices derived from the given 
(nondifferentiable) coefficients of the metric tensor on $M$.\\
\indent
The equivalence problem we are about to pose has to encompass both
the regular metric case~\eqref{E: regular preframe}, and the 
nondifferentiable case~\eqref{E: singular preframe}. Thus, the 
structure group has to be a closed subgroup of 
$SU(2k) = Sp(2k) \cap SO(4k)$ in keeping with the lift provided 
by the symplectic connection. The freedom afforded by the choice 
between the preframe and its complement (and, presumably, other
$2k$-tuples in between) has to be reflected in that subgroup.
However, further reduction is necessary
as all the off-diagonal metric coefficients are in place to manifest
various rank deficiencies, and, consequently, rotations among
the base variables $(x_i, x_j,\;\;i,j \leqslant 2k)$ are out.
What is left are the rotations of symplectic conjugate pairs
$(x_i, x_{i+2k},\;\;i\leqslant 2k)$. Thus, necessarily, the 
Jacobians of our diffeomorphisms ought to belong to
\begin{equation}\label{E: structure group} 
\Big(\bigoplus_{i=1}^{2k} SO_i(2, \mathbb{R})\Big) \times T_UU.
\end{equation}
\indent
Heuristically speaking, this is the smallest subgroup
of $\text{Diff}(T_UU) \subset \text{Diff}(T^*M)$ containing
the full set of local preternatural rotations.

\subsection{Prolongation}
\paragraph{} 
Having delineated the structure group, we tackle the inherent
nonuniqueness of the regularized preframes. Evidently, 
\eqref{E: structure group} is not spacious enough to furnish a 
faithful representation of the scalar action of 
$\mathfrak{Aut}(\Phi (U, ds^2))$. To carve out such a 
representation, we identify the fiber variables on
$TG \cong \mathfrak{g}$ with the fiber variables on $T_UU$. 
Consequently, our construction can only be realized for subgroups
of the orthogonal group $G \subset SO(2k)$, such that 
${\dim}_{\mathbb{R}} \;\mathfrak{g} \leqslant 2k -1$. One fiber
dimension (locally parameterized by $x_{2k+1}$) has already been
utilized as symplectic conjugate of the singularity parameter.
The most interesting case of $M = {\mathbb{R}}^{3+1}$ endowed
with the reduced rotational group, including the Kerr spacetime
and a large class of gravitational collapse scenarios, is,
fortuitously, covered.\\
\indent
We identify \begin{align}
\exp t_{ab}K_{ab}(e^{w_{ij}}&\cos {\kappa}_{ij}) \mapsto  
\exp x_{2k+1+\sigma (a,b)}K_{ab}(e^{w_{ij}}\cos{\kappa}_{ij})\\\notag
\overset{\textnormal{def}}{=} &e^{{\widehat{w}}_{ij}(x_1,,,,,x_{2k}, 
x_{2k+2},...,x_{4k})} 
\cos\widehat{{\kappa}}_{ij}(x_1,,,,,x_{2k}, x_{2k+2},...,x_{4k}), 
\end{align} 
\noindent
where the basis of the Lie algebra, consisting of anti-Hermitian 
matrices, is labeled via $1 \leqslant a < b$, and $\sigma$ is a 
bijection from the set of pairs of indices $(a,b)$ into natural
numbers, $\sigma (a,b) \leqslant 2k-1$. With the aid of
Lemma~\ref{T: period matrix invariance}, we ensure that the prolonged
Hill eigenfunctions $\cos \widehat{{\kappa}}_{ij}(x_1,,,,,x_{2k},
x_{2k+1},...,x_{4k})$ generate the same Riemann period matrices.
This ansatz is nothing more than an application of the third
fundamental theorem of Lee. Compactness of $G$ ensures locality,
i. e. $-c \leqslant x_{2k+1+\sigma (a,b)} \leqslant c$ covers
all possible eigenfunctions of the Hill operator with optimized
potentials.\\
\indent
For each $(2k-1)$-tuple $(x_{2k+2},...,x_{4k}) \in T_UU$, we let
$ \widehat{w}_{ij}(x_1,,,,,x_{2k}, x_{2k+2},...,x_{4k})$ be 
the corresponding weight, stemming from the solutions 
of~\eqref{E: Laplace}, subsequently smoothed. 

\subsection{Free scalar fields}
\paragraph{}
With the advent of prolonged auxiliary fields, we now need to make 
certain they are amenable to the method of equivalence's requirements 
listed in the hypothesis of Background 
Theorem~\ref{T: e-structure equivalence}. 
More specifically, their partial derivatives (those relevant for 
geometric features, curvature above all) ought to remain independent 
throughout $T_UU$ in order for the desired $e$-structures to exist.\\
\indent
The fields with independent partial derivatives will be found via
$h$-principle for one partial differential relation, known as the
\textit{freedom} relation (see~\cite{Grom}, Section 1.1.4).
Here is a more or less detailed outline.\\
\indent
Consider a smooth fibration $p: X \longrightarrow V$, and let
$J^rX$ be the bundle of $r$-jets (of germs) of smooth sections
$f: V \longrightarrow X$. For instance, $J^1X$ consists of linear
maps $L: T_v(V) \longrightarrow T_x(X)$ for all $x \in X$ and
$v =p(x) \in V$ such that $D_p \circ L = \textnormal{Id}: T_v(V) 
\longrightarrow T_v(V)$. Here $D_p: T(X) \longrightarrow T(V)$ is the
induced map of tangent bundles.\\
\indent
Given a $C^2$-map $f: V \longrightarrow {\mathbb{R}}^c$ defined locally
in coordinates $x= (x_1,...,x_m)$ in the vicinity of $x \in V$, we 
denote by $T^2_f(V,x) \subset T_y ({\mathbb{R}}^c) = {\mathbb{R}}^c,\;\;
y=f(x)$  the subspace spanned by the vectors
$$ \frac{\partial f}{\partial x_i}(x),\;\;\; \frac{{\partial}^2 f}
{\partial x_i \partial x_j}(x), \;\;\; 1 \leqslant i,j \leqslant m.$$
\noindent
This subspace is called the (second) \textit{osculating} space of $f$.
The first osculating space is just $D_f(T_x(V))$. The dimension of
$T^2_f(V,x)$ can vary between 0 and $\frac{1}{2}m(m+3)$. $f$ is 
called \textit{free} if $\dim T^2_f(V,x) = \frac{1}{2}m(m+3),\;\;\;
\forall x \in V$.
\indent
The freedom relation $\wp \subset J^2X$ for $X=V\times{\mathbb{R}}^c
\longrightarrow V$ could either be an empty set if $c <\frac{1}{2}m(m+3)$,
or an open dense subset in $J^2X$ invariant under the 
natural action of diffeomorphisms of $V$ and affine transformations of
${\mathbb{R}}^c$. Furthermore, the  action  of $\text{Diff}(V) \times
\text{Aff}({\mathbb{R}}^c)$ is transitive on $\wp$. In fact, $\wp$ is
the only proper open subset enjoying these properties. \\
\indent
The simplest example of a free map $f: {\mathbb{R}}^m \longrightarrow
{\mathbb{R}}^c$ is given by the $c$ monomials $x_i, x_ix_j$ on 
${\mathbb{R}}^m$, where $1 \leqslant i \leqslant j \leqslant m$.
For closed manifolds, free maps are rare. The only known natural
ones involve embeddings/immersions of spheres and real projective 
spaces into ${\mathbb{R}}^c$.\\
\indent
The homotopy principle (in the context of free maps) asseverates 
the existence of a smooth homotopy between an arbitrary twice
differentiable one and a free one provided some additional conditions
are met. For open manifolds, serendipiously, $f: V \longrightarrow 
{\mathbb{R}}^c$ can be homotoped to a free map at the critical
dimension $c =\frac{1}{2}m(m+3)$ (see~\cite{Grom}, Section 2.2.2 for
the proof). These results have been generalized to higher order 
derivatives, as well as to covariant derivatives, but for our
purposes, this will suffice.\\
\indent
We regard ${\widehat{w}}_{vu}(x_1,..,x_{2k}, x_{2k+2},...,x_{4k})$,
$\widehat{{\kappa}}_{vu}(x_1,...,x_{2k},x_{2k+2},...,x_{4k})$ as maps 
${\widehat{w}}_{vu},\;\widehat{{\kappa}}_{vu}: T_UU \longrightarrow
{\mathbb{R}}^{c_{vu}}$, $c_{vu}= \frac{1}{2}(4k-1- \#(l_0))(4k +2 - 
\#(l_0))$, where $\#(l_0)$ denotes an enumeration of base variables
informed by the number of identically vanishing Christoffel symbols
and Riemann curvature tensor coefficients in~\eqref{E: Christoffel},  
and their components being the partial derivatives varying over 
$T_UU$. To be punctilious, their range is a subbundle of the 
second osculating plane bundle, but since $T_UU$ is contractible,
we might as well take the aforementioned simplistic view.  
Our objective is to effect a smooth homotopy of (partial 
derivatives of) ${\widehat{w}}^t_{vu}$, $\widehat{{\kappa}}^t_{vu},\;\;
t \in [0,1]$ satisfying
\begin{equation}
{\widehat{w}}^0_{vu} ={\widehat{w}}_{vu},\;\;\widehat{{\kappa}}^0_{vu}
=\widehat{{\kappa}}_{vu}, \;\;{\widehat{w}}^1_{vu}, 
\widehat{{\kappa}}^1_{vu} \in {\wp}_{vu}. 
\end{equation}
\indent
Owing to 
$\wp$ being dense in $J^2T_UU$, we can find free ${\widehat{w}}^1_{vu},
\;\widehat{{\kappa}}^1_{vu}$ arbitrarily close to the original ones.
In the sequel we use the same notation for the free weights and kappas. 
With such functions, the (generic) auxiliary scalar fields 
$e^{{\widehat{w}}_{vu}}\cos\widehat{{\kappa}}_{vu}$ referenced in 
the title of this section are free.
\subsection{Main theorem}
\paragraph{}
All the seemingly disconnected pieces from the previous sections
come together. So far we have accumulated the following data related
to the original (pseudo-)\\
metric $g$: smooth weights $\{e^{\widehat{w}}_{vu}(x_1,..,x_{2k}, 
x_{2k+2},...,x_{4k})\}$, the prolonged Hill operator 
eigenfunctions $\{\pm\sin(\widehat{{\kappa}}_{vu}(x_1,...,x_{2k}))\}$,  
$\{\pm\cos (\widehat{{\kappa}}_{vu}(x_1,...,x_{2k}))\}$ corresponding 
to the loci of nondifferentiability and rank deficiency $U^N \cup U^O$,
their attendant Riemann period matrices 
$[[{\mathfrak{R}}^{vu}]_{\beta \gamma}]$.
\noindent
Generically, $\det [e^{{w_{vu}(x_1,..,x_{2k})}}] > 0$, as rank deficiencies
are now carried by the Hill operator eigenfunctions.\\
\indent
Now our regularized $4k$-coframe $\Psi$ is comprised
of the image of our preframe under $\boldsymbol{\alpha} \circ \ccc$ joined
with its complement. Below we make use of the symmetries involved 
($\widehat{{\kappa}}_{ij}= \widehat{{\kappa}}_{ji}$, 
${\widehat{w}}_{ij}= {\widehat{w}}_{ji}$, 
$\widehat{{\kappa}}_{i+2kj}=\widehat{{\kappa}}_{ij}$, 
${\widehat{w}}_{i+2kj}={\widehat{w}}_{ij}$).
\begin{equation}\label{E:  preframe}
\begin{cases}
{\Psi}_1 = \sum_{j=1}^{k} (e^{{\widehat{w}}_{1j}}\cos\widehat{{\kappa}}_{1j}dx_j + 
e^{{\widehat{w}}_{1j}}\sin\widehat{{\kappa}}_{1j}dx_{2k+j});\\
.........................................................................\\
{\Psi}_k = \sum_{j=1}^{k} (e^{{\widehat{w}}_{kj}}\cos\widehat{{\kappa}}_{kj}dx_j + 
e^{{\widehat{w}}_{kj}}\sin\widehat{{\kappa}}_{kj}dx_{2k+j});\\
{\Psi}_{k+1} = \sum_{j=1}^{k} (e^{{\widehat{w}}_{k+1j}}\sin\widehat{{\kappa}}_{k+1j}dx_j -
e^{{\widehat{w}}_{k+1j}}\cos\widehat{{\kappa}}_{k+1j}dx_{2k +j});\\
.........................................................................\\
{\Psi}_{2k} = \sum_{j=1}^{k} (e^{{\widehat{w}}_{2kj}}\sin\widehat{{\kappa}}_{2kj}dx_j -
e^{{\widehat{w}}_{2kj}}\cos\widehat{{\kappa}}_{2kj}dx_{2k +j}).
\end{cases}
\end{equation} 
\noindent
To write down the image of the complementary preframe, simply replace
$\sin\widehat{{\kappa}}_{ij} \mapsto \cos\widehat{{\kappa}}_{ij}$,
$\cos\widehat{{\kappa}}_{ij} \mapsto -\sin\widehat{{\kappa}}_{ij}$
in the coframe components advanced by $2k$.

\begin{thm}\label{T: main}
Given sets $U$ and $V$ on symplectic  (pseudo-)Riemannian manifolds 
$(M,g, \omega)$, and $(N,g, \omega)$, there exists an isometry 
${\AA}_{UV}: T_UU \to T_VV$, ${\AA}_{UV}(x_1)= y_1$, if and only if 
${\digamma}_2(\Psi)_U$, ${\digamma}_2(\psi)_V$ satisfy the hypothesis 
of Background Theorem~\ref{T: e-structure equivalence}. 
\end{thm}
\begin{corollary}\label{T: corollary}
${\AA}^*([[{\mathfrak{R}}^{vu}]_{\beta \gamma}]_U)  =
[[{\mathfrak{R}}^{vu}]_{\beta \gamma}]_V$ as lexicographically 
ordered sets.
\end{corollary}
\noindent
Should the metric turn out to be differentiable, we set
$[[{\mathfrak{R}}^{vu}]_{\beta \gamma}] = [[0]_{\beta \gamma}]$,
and the erstwhile condition becomes vacuous.
\subsection{Proof of the Main Theorem}
\paragraph{}
To invoke Background Theorem~\ref{T: e-structure equivalence}
we have to show its hypothesis being satisfied. Our coframes
were chosen on general grounds, and therein lies some uncertainty as 
to whether our choices do not run afoul of the Cartan's group
reduction and normalization methods.\\
\indent
The first order of business is to delineate the intrinsic 
torsion coefficients. Those depend on the group structure
as well as on the integrability properties of the coframe.
Thus we find the requisite 2-forms for $i \leqslant k$.
The other case can be derived mutatis mutandis.
{\allowdisplaybreaks
\begin{align}
d{\Psi}_i =&\;d\Big(\sum_je^{{\widehat{w}}_{ij}}\Big(\cos 
{\widehat{\kappa}}_{ij}dx_j +
\sin {\widehat{\kappa}}_{ij}dx_{j+2k}\Big)\Big)\\\notag
=&\sum_l \sum_j\frac{\partial {\widehat{w}}_{ij}}{\partial x_l}
e^{{\widehat{w}}_{ij}}\Big(\cos{\widehat{\kappa}}_{ij}dx_{l} \wedge dx_j +
\sin {\widehat{\kappa}}_{ij}dx_{l} \wedge dx_{j+2k}\Big)\\\notag
+&\sum_l \sum_j\frac{\partial {\widehat{w}}_{ij}}{\partial x_{l+2k}}
e^{{\widehat{w}}_{ij}}\Big(\cos{\widehat{\kappa}}_{ij}dx_{l+2k} \wedge dx_j +
\sin {\widehat{\kappa}}_{ij}dx_{l+2k} \wedge dx_{j+2k}\Big)\\\notag
+&\sum_l \sum_j\frac{\partial {\widehat{\kappa}}_{ij}}{\partial x_l}
e^{{\widehat{w}}_{ij}}\Big(-\sin{\widehat{\kappa}}_{ij}dx_l \wedge dx_j +
\cos {\widehat{\kappa}}_{ij}dx_l \wedge dx_{j+2k}\Big)\\\notag
+&\sum_l \sum_j\frac{\partial {\widehat{\kappa}}_{ij}}{\partial x_{l+2k}}
e^{{\widehat{w}}_{ij}}\Big(-\sin{\widehat{\kappa}}_{ij}dx_{l+2k} \wedge dx_j +
\cos {\widehat{\kappa}}_{ij}dx_{l+2k} \wedge dx_{j+2k}\Big)\\\notag
=&\sum_l \sum_je^{{\widehat{w}}_{ij}}\Big(
\frac{\partial {\widehat{w}}_{ij}}{\partial x_{l}}\cos{\widehat{\kappa}}_{ij}
-\frac{\partial {\widehat{\kappa}}_{ij}}{\partial x_{l}} 
\sin{\widehat{\kappa}}_{ij}\Big)\\\notag
\times &\sum b^{u \wedge v}_{l+2k \wedge j}{\Psi}_u \wedge {\Psi}_v+
b^{u+2k \wedge v}_{l+2k \wedge j}{\Psi}_{u+2k} \wedge {\Psi}_v
+ b^{u+2k \wedge v+2k}_{l+2k \wedge j}{\Psi}_{u+2k} \wedge {\Psi}_{v+2k}\\\notag
+&\sum_l \sum_je^{{\widehat{w}}_{ij}}\Big(
\frac{\partial {\widehat{w}}_{ij}}{\partial x_{l+2k}}\cos{\widehat{\kappa}}_{ij}
-\frac{\partial {\widehat{\kappa}}_{ij}}{\partial x_{l+2k}} 
\sin{\widehat{\kappa}}_{ij}\Big)\\\notag
\times &\sum b^{u \wedge v}_{l \wedge j}{\Psi}_u \wedge {\Psi}_v+
b^{u+2k \wedge v}_{l \wedge j}{\Psi}_{u+2k} \wedge {\Psi}_v
+ b^{u+2k \wedge v+2k}_{l \wedge j}{\Psi}_{u+2k} \wedge {\Psi}_{v+2k}\\\notag
+&\sum_l \sum_je^{{\widehat{w}}_{ij}}\Big(
\frac{\partial {\widehat{w}}_{ij}}{\partial x_{l}}\sin{\widehat{\kappa}}_{ij}
+\frac{\partial {\widehat{\kappa}}_{ij}}{\partial x_{l}} 
\cos{\widehat{\kappa}}_{ij}\Big)\\\notag
\times &\sum b^{u \wedge v}_{l \wedge j+2k}{\Psi}_u \wedge {\Psi}_v+
b^{u+2k \wedge v}_{l \wedge j+2k}{\Psi}_{u+2k} \wedge {\Psi}_v
+ b^{u+2k \wedge v+2k}_{l \wedge j+2k}{\Psi}_{u+2k} \wedge {\Psi}_{v+2k}\\\notag
+&\sum_l \sum_je^{{\widehat{w}}_{ij}}\Big(
\frac{\partial {\widehat{w}}_{ij}}{\partial x_{l+2k}}\sin{\widehat{\kappa}}_{ij}
+\frac{\partial {\widehat{\kappa}}_{ij}}{\partial x_{l+2k}} 
\cos{\widehat{\kappa}}_{ij}\Big)\\\notag
\times &\sum b^{u \wedge v}_{l+2k \wedge j+2k}{\Psi}_u \wedge {\Psi}_v+
b^{u+2k \wedge v}_{l+2k \wedge j+2k}{\Psi}_{u+2k} \wedge {\Psi}_v
+ b^{u+2k \wedge v+2k}_{l+2k \wedge j+2k}{\Psi}_{u+2k} \wedge {\Psi}_{v+2k}.\notag
\end{align}}
\noindent
The coefficients above are entries of the lexicographically ordered
compound matrix induced by ${\Lambda}^2T^*TM$. They are defined by the equation
\begin{equation}
[b^u_v][\vec{\Psi}]=[\vec{dx}],\quad
b^{u \wedge v}_{i \wedge j}\overset{\textnormal{def}}{=}b^u_ib^v_j-b^u_jb^v_i. 
\end{equation}
\noindent
Thus $[b^{u \wedge v}_{i \wedge j}] \in GL(\binom{4k}{2}, \mathbb{R})
\times T_UU$ since the preframes have maximal rank in $U$.\\
\indent
Similarly, {\allowdisplaybreaks
\begin{align}
&d{\Psi}_{i+2k} =\;d\Big(\sum_je^{{\widehat{w}}_{ij}}\Big(\cos 
{\widehat{\kappa}}_{ij}dx_{j+2k} -
\sin {\widehat{\kappa}}_{ij}dx_{j}\Big)\Big)\\\notag
=&\sum_l \sum_j\frac{\partial {\widehat{w}}_{ij}}{\partial x_l}
e^{{\widehat{w}}_{ij}}\Big(\cos{\widehat{\kappa}}_{ij}dx_{l} \wedge 
dx_{j+2k}-\sin {\widehat{\kappa}}_{ij}dx_{l} \wedge dx_{j}\Big)\\\notag
+&\sum_l \sum_j\frac{\partial {\widehat{w}}_{ij}}{\partial x_{l+2k}}
e^{{\widehat{w}}_{ij}}\Big(\cos{\widehat{\kappa}}_{ij}dx_{l+2k} \wedge 
dx_{j+2k} -\sin {\widehat{\kappa}}_{ij}dx_{l+2k} \wedge dx_{j}\Big)\\\notag
-&\sum_l \sum_j\frac{\partial {\widehat{\kappa}}_{ij}}{\partial x_l}
e^{{\widehat{w}}_{ij}}\Big(\sin{\widehat{\kappa}}_{ij}dx_l \wedge 
dx_{j+2k} +\cos {\widehat{\kappa}}_{ij}dx_l \wedge dx_{j}\Big)\\\notag
-&\sum_l \sum_j\frac{\partial {\widehat{\kappa}}_{ij}}{\partial x_{l+2k}}
e^{{\widehat{w}}_{ij}}\Big(\sin{\widehat{\kappa}}_{ij}dx_{l+2k} \wedge 
dx_{j+2k}+\cos {\widehat{\kappa}}_{ij}dx_{l+2k} \wedge dx_{j}\Big)\\\notag
=&\sum_l \sum_je^{{\widehat{w}}_{ij}}\Big(
\frac{\partial {\widehat{w}}_{ij}}{\partial x_{l}}\cos{\widehat{\kappa}}_{ij}
-\frac{\partial {\widehat{\kappa}}_{ij}}{\partial x_{l}} 
\sin{\widehat{\kappa}}_{ij}\Big)\\\notag
\times &\sum b^{u \wedge v}_{l \wedge j+2k}{\Psi}_u \wedge {\Psi}_v+
b^{u+2k \wedge v}_{l \wedge j+2k}{\Psi}_{u+2k} \wedge {\Psi}_v
+ b^{u+2k \wedge v+2k}_{l \wedge j+2k}{\Psi}_{u+2k} \wedge {\Psi}_{v+2k}\\\notag
+&\sum_l \sum_je^{{\widehat{w}}_{ij}}\Big(
\frac{\partial {\widehat{w}}_{ij}}{\partial x_{l+2k}}\cos{\widehat{\kappa}}_{ij}
-\frac{\partial {\widehat{\kappa}}_{ij}}{\partial x_{l+2k}} 
\sin{\widehat{\kappa}}_{ij}\Big)\\\notag
\times &\sum b^{u \wedge v}_{l+2k \wedge j+2k}{\Psi}_u \wedge {\Psi}_v+
b^{u+2k \wedge v}_{l+2k \wedge j+2k}{\Psi}_{u+2k} \wedge {\Psi}_v
+ b^{u+2k \wedge v+2k}_{l+2k \wedge j+2k}{\Psi}_{u+2k} \wedge {\Psi}_{v+2k}\\\notag
-&\sum_l \sum_je^{{\widehat{w}}_{ij}}\Big(
\frac{\partial {\widehat{w}}_{ij}}{\partial x_{l}}\sin{\widehat{\kappa}}_{ij}
+\frac{\partial {\widehat{\kappa}}_{ij}}{\partial x_{l}} 
\cos{\widehat{\kappa}}_{ij}\Big)\\\notag
\times &\sum b^{u \wedge v}_{l \wedge j}{\Psi}_u \wedge {\Psi}_v+
b^{u+2k \wedge v}_{l \wedge j}{\Psi}_{u+2k} \wedge {\Psi}_v
+ b^{u+2k \wedge v+2k}_{l \wedge j}{\Psi}_{u+2k} \wedge {\Psi}_{v+2k}\\\notag
-&\sum_l \sum_je^{{\widehat{w}}_{ij}}\Big(
\frac{\partial {\widehat{w}}_{ij}}{\partial x_{l+2k}}\sin{\widehat{\kappa}}_{ij}
+\frac{\partial {\widehat{\kappa}}_{ij}}{\partial x_{l+2k}} 
\cos{\widehat{\kappa}}_{ij}\Big)\\\notag
\times &\sum b^{u \wedge v}_{l+2k \wedge j}{\Psi}_u \wedge {\Psi}_v+
b^{u+2k \wedge v}_{l+2l \wedge j}{\Psi}_{u+2k} \wedge {\Psi}_v
+ b^{u+2k \wedge v+2k}_{l+2k \wedge j}{\Psi}_{u+2k} \wedge {\Psi}_{v+2k}\notag.
\end{align}}
\indent
Decomposing the structure group~\eqref{E: structure group} into one-parameter
subgroups acting on $2\times 2$ blocks combining the conjugate constituent 
1-forms, we begin by computing the Maurer-Cartan matrix of the $i$-th subgroup:
\begin{equation}
\Bigg( d\left[\begin{array}{clcr}
\cos t & \sin t \\
-\sin t & \cos t 
\end{array} \right]\Bigg)
\left[\begin{array}{clcr}
\cos t & -\sin t \\
\sin t & \cos t 
\end{array} \right]=
\left[\begin{array}{clcr}
0 & {\beta}_i \\
-{\beta}_i & 0 
\end{array} \right].
\end{equation}
\noindent
Writing down the structure equation~\eqref{E: Cartan2},
we arrive at
\begin{equation}
d \left[\begin{array}{l}
{\Psi}_i\\
{\Psi}_{i+2k}\\
\end{array}\right] =
 \left[\begin{array}{clcr}
0 & {\beta}_i \\
-{\beta}_i & 0 
\end{array} \right]
\wedge
\left[\begin{array}{l}
{\Psi}_i\\
{\Psi}_{i+2k}\\
\end{array}\right] +
\left[\begin{array}{l}
\sum(\cdot) \wedge (\cdot)\\
\sum(\cdot) \wedge (\cdot)\\
\end{array}\right]
\end{equation}
\noindent
upon modifying the Maurer-Cartan matrix via {\allowdisplaybreaks
\begin{align}\label{E: absorbed 2-forms}
&{\beta}_i \mapsto {\beta}_i \\\notag
-&\sum_l \sum_je^{{\widehat{w}}_{ij}}\Big(
\frac{\partial {\widehat{w}}_{ij}}{\partial x_{l}}\cos{\widehat{\kappa}}_{ij}
-\frac{\partial {\widehat{\kappa}}_{ij}}{\partial x_{l}} 
\sin{\widehat{\kappa}}_{ij}\Big)\\\notag
\times &\sum (b^{u \wedge i+2k}_{l \wedge j+2k}+
b^{u+2k \wedge i+2k}_{l \wedge j+2k}){\Psi}_{i+2k}\\\notag
-&\sum_l \sum_je^{{\widehat{w}}_{ij}}\Big(
\frac{\partial {\widehat{w}}_{ij}}{\partial x_{l+2k}}\cos{\widehat{\kappa}}_{ij}
-\frac{\partial {\widehat{\kappa}}_{ij}}{\partial x_{l+2k}} 
\sin{\widehat{\kappa}}_{ij}\Big)\\\notag
\times &\sum (b^{u \wedge i+2k}_{l+2k \wedge j+2k}+
b^{u+2k \wedge i+2k}_{l+2k \wedge j+2k}){\Psi}_{i+2k}\\\notag
-&\sum_l \sum_je^{{\widehat{w}}_{ij}}\Big(
\frac{\partial {\widehat{w}}_{ij}}{\partial x_{l}}\sin{\widehat{\kappa}}_{ij}
+\frac{\partial {\widehat{\kappa}}_{ij}}{\partial x_{l}} 
\cos{\widehat{\kappa}}_{ij}\Big)\\\notag
\times &\sum (b^{u \wedge i+2k}_{l \wedge j+2k}+
b^{u+2k \wedge i+2k}_{l \wedge j+2k}){\Psi}_{i+2k}\\\notag
-&\sum_l \sum_je^{{\widehat{w}}_{ij}}\Big(
\frac{\partial {\widehat{w}}_{ij}}{\partial x_{l+2k}}\sin{\widehat{\kappa}}_{ij}
+\frac{\partial {\widehat{\kappa}}_{ij}}{\partial x_{l+2k}} 
\cos{\widehat{\kappa}}_{ij}\Big)\\\notag
\times &\sum (b^{u \wedge i+2k}_{l+2k \wedge j+2k}+
b^{u+2k \wedge i+2k}_{l+2k \wedge j+2k}){\Psi}_{i+2k}\\\notag 
+&\sum_l \sum_je^{{\widehat{w}}_{ij}}\Big(
\frac{\partial {\widehat{w}}_{ij}}{\partial x_{l}}\cos{\widehat{\kappa}}_{ij}
-\frac{\partial {\widehat{\kappa}}_{ij}}{\partial x_{l}} 
\sin{\widehat{\kappa}}_{ij}\Big)\\\notag
\times &\sum (b^{u \wedge i+2k}_{l \wedge j+2k}+
b^{u+2k \wedge i+2k}_{l \wedge j+2k}){\Psi}_i\\\notag
+&\sum_l \sum_je^{{\widehat{w}}_{ij}}\Big(
\frac{\partial {\widehat{w}}_{ij}}{\partial x_{l+2k}}\cos{\widehat{\kappa}}_{ij}
-\frac{\partial {\widehat{\kappa}}_{ij}}{\partial x_{l+2k}} 
\sin{\widehat{\kappa}}_{ij}\Big)\\\notag
\times &\sum (b^{u \wedge i+2k}_{l+2k \wedge j+2k}+
b^{u+2k \wedge i+2k}_{l+2k \wedge j+2k}){\Psi}_i\\\notag 
-&\sum_l \sum_je^{{\widehat{w}}_{ij}}\Big(
\frac{\partial {\widehat{w}}_{ij}}{\partial x_{l}}\sin{\widehat{\kappa}}_{ij}
+\frac{\partial {\widehat{\kappa}}_{ij}}{\partial x_{l}} 
\cos{\widehat{\kappa}}_{ij}\Big)\\\notag
\times &\sum (b^{u \wedge i+2k}_{l \wedge j+2k}+
b^{u+2k \wedge i+2k}_{l \wedge j+2k}){\Psi}_i\\\notag
-&\sum_l \sum_je^{{\widehat{w}}_{ij}}\Big(
\frac{\partial {\widehat{w}}_{ij}}{\partial x_{l+2k}}\sin{\widehat{\kappa}}_{ij}
+\frac{\partial {\widehat{\kappa}}_{ij}}{\partial x_{l+2k}} 
\cos{\widehat{\kappa}}_{ij}\Big)\\\notag
\times &\sum (b^{u \wedge i+2k}_{l+2k \wedge j+2k}+
b^{u+2k \wedge i+2k}_{l+2k \wedge j+2k}){\Psi}_i\notag.
\end{align}}
\noindent
Thanks to a) our matrix group being sparse, and b) our normal
form being so simple, this modification is an example of Lie
algebra-compatible absorption. Sparsity is manifested by the fact
that out of $2k(4k-1)$ linearly independent monomials spanning
${\Lambda}^2T^*TM$, only $2k$ are absorbed - precisely those
involving symplectic conjugate pairs of variables.\\
\indent
Hence the intrinsic torsion coefficients occur for the monomial 2-forms
not listed in~\eqref{E: absorbed 2-forms}.
In terms of combinations of indices, for ${\gamma}^i_{ms}$ the general 
expression is {\allowdisplaybreaks 
\begin{align}\label{E: expression for gamma}
{\gamma}^i_{ms} 
=&\sum_l \sum_je^{{\widehat{w}}_{ij}}\Big(
\frac{\partial {\widehat{w}}_{ij}}{\partial x_{l}}\cos{\widehat{\kappa}}_{ij}
-\frac{\partial {\widehat{\kappa}}_{ij}}{\partial x_{l}} 
\sin{\widehat{\kappa}}_{ij}\Big)b^{m \wedge s}_{l \wedge j}\\\notag
+&\sum_l \sum_je^{{\widehat{w}}_{ij}}\Big(
\frac{\partial {\widehat{w}}_{ij}}{\partial x_{l+2k}}\cos{\widehat{\kappa}}_{ij}
-\frac{\partial {\widehat{\kappa}}_{ij}}{\partial x_{l+2k}} 
\sin{\widehat{\kappa}}_{ij}\Big)b^{m \wedge s}_{l+2k \wedge j}\\\notag
+&\sum_l \sum_je^{{\widehat{w}}_{ij}}\Big(
\frac{\partial {\widehat{w}}_{ij}}{\partial x_{l}}\sin{\widehat{\kappa}}_{ij}
+\frac{\partial {\widehat{\kappa}}_{ij}}{\partial x_{l}} 
\cos{\widehat{\kappa}}_{ij}\Big)b^{m \wedge s}_{l \wedge j+2k}\\\notag
+&\sum_l \sum_je^{{\widehat{w}}_{ij}}\Big(
\frac{\partial {\widehat{w}}_{ij}}{\partial x_{l+2k}}\sin{\widehat{\kappa}}_{ij}
+\frac{\partial {\widehat{\kappa}}_{ij}}{\partial x_{l+2k}} 
\cos{\widehat{\kappa}}_{ij}\Big)b^{m \wedge s}_{l+2k \wedge j+2k}.\notag
\end{align}}
One other coefficient (to prime one's intuition), 
${\gamma}^{i+2k}_{m+2ks},\;\;s \ne i$ is given by {\allowdisplaybreaks
\begin{align}
{\gamma}^{i+2k}_{m+2ks} 
=&\sum_l \sum_je^{{\widehat{w}}_{ij}}\Big(
\frac{\partial {\widehat{w}}_{ij}}{\partial x_{l}}\cos{\widehat{\kappa}}_{ij}
-\frac{\partial {\widehat{\kappa}}_{ij}}{\partial x_{l}} 
\sin{\widehat{\kappa}}_{ij}\Big)b^{m+2k \wedge s}_{l \wedge j}\\\notag
+&\sum_l \sum_je^{{\widehat{w}}_{ij}}\Big(
\frac{\partial {\widehat{w}}_{ij}}{\partial x_{l+2k}}\cos{\widehat{\kappa}}_{ij}
-\frac{\partial {\widehat{\kappa}}_{ij}}{\partial x_{l+2k}} 
\sin{\widehat{\kappa}}_{ij}\Big)b^{m+2k \wedge s}_{l+2k \wedge j}\\\notag
+&\sum_l \sum_je^{{\widehat{w}}_{ij}}\Big(
\frac{\partial {\widehat{w}}_{ij}}{\partial x_{l}}\sin{\widehat{\kappa}}_{ij}
+\frac{\partial {\widehat{\kappa}}_{ij}}{\partial x_{l}} 
\cos{\widehat{\kappa}}_{ij}\Big)b^{m+2k \wedge s}_{l \wedge j+2k}\\\notag
+&\sum_l \sum_je^{{\widehat{w}}_{ij}}\Big(
\frac{\partial {\widehat{w}}_{ij}}{\partial x_{l+2k}}\sin{\widehat{\kappa}}_{ij}
+\frac{\partial {\widehat{\kappa}}_{ij}}{\partial x_{l+2k}} 
\cos{\widehat{\kappa}}_{ij}\Big)b^{m+2k \wedge s}_{l+2k \wedge j+2k}.\notag
\end{align}}
\noindent
We note (without performing actual computations) that the 
parenthetical expressions cannot all vanish at the loci
of nondifferentiability/rank deficiency. In part, this
is due to the fact that all $\sin{\widehat{\kappa}}_{iu},
\;\sin{\widehat{\kappa}}_{mv}$ vanish at $U^N \cup U^O$ and
nowhere else. And the summation
over $l, j$ in conjunction with the fact that our compound
matrix is invertible ensures 
$|{\gamma}^i_{ms}|,\; |{\gamma}^{i+2k}_{m+2ks}|>0$ on $T_UU$.\\
\indent 
We split the rest of the proof into smaller fragments, each 
stated as a subsidiary lemma.
\begin{lem}
The action of the structure group on the intrinsic torsion of 
a $4k$-coframe comprised of~\eqref{E: preframe} and its 
complement is limited to translations involving 
symplectic conjugates.
\end{lem}
\begin{proof}
Take an arbitrary intrinsic torsion coefficient ${\gamma}^i_{ms}$.
All the sources of ${\Psi}_m \wedge {\Psi}_s$ are listed here:
$$d{\gamma}^i_{ms} = \sum{\gamma}^i_{m+2ks}{\beta}_m - 
\sum {\gamma}^i_{ms+2k}{\beta}_s +\sum{\gamma}^
{i+2k}_{ms}{\beta}_i\;\;\mod {\Psi}_m \wedge {\Psi}_s.$$
\indent 
Inspecting this expression, and comparing it with the general
infinitesimal group action
formula~\eqref{E: torsion group action}, we conclude that
since $o^l$ is the unique anti-Hermitian matrix generating
$SO_i(2,\mathbb{R})$, we have
$a^m_{ml} = a^s_{sl} =0$. Hence only translations 
are present. Now in view of~\eqref{E: expression for gamma}, 
we can exclude some combinations of indices. Thus for 
${\gamma}^i_{ms},\;s, m \ne i+2k$, for those differentials 
are absorbed into ${\beta}_i$.
\end{proof}
\begin{lem}
The equivalence problem involving $4k$ coframes~\eqref{E: preframe}
reduces to a regular $e$-structure.
\end{lem}
\begin{proof}
Regularity in the context of $e$-structures presupposes that a) the
action of the structure group on the intrinsic torsion forms a
single orbit, and is amenable to normalization, and b) the rank 
of the set
$${\digamma}_{s}(\Psi)
\{{\gamma}^i_{jm},{\gamma}^i_{jm|l_1},\cdots ,{\gamma}^v_{jm|l_1, 
\cdots,|l_{s-1}}; 1 \leqslant i, j, m, v, l_1,\cdots, l_{s-1} \leqslant 4k\}$$
\noindent
is constant on some open neighborhood of a fixed point $(x_1,\cdots,x_{4k})$.\\
\indent
Concerning the orbit unity, we can begin with $d{\gamma}^i_{jm} \equiv 0 
\mod {\Psi}_j \wedge {\Psi}_m$, which entails ${\beta}_s =0,\;\; \forall s 
\leqslant 2k$, and then gradually `turn on' the ${\beta}_s$. This way, 
judiciously choosing values of the $2k$ group parameters, we can continuously 
deform the set of intrinsic torsion coefficients to obtain any other 
configuration within the $2k$-dimensional group parametric space.\\
\indent
To demonstrate the rank being constant, we expand on the remark
below~\eqref{E: expression for gamma}. Namely, the intrinsic coefficients
are nonvanishing due to the structure of our standard normal form as well
as the facts that the freedom relation $\wp$ is open and dense in the
jet bundle, and, in addition, satisfies the $h$-principle.\\
\indent
The `next layer' consists of first derivatives of the intrinsic torsion
coefficients ${\gamma}^i_{jm|l_1}$. They involve $\frac{\partial 
b^{m \wedge s}_{v \wedge j}}{\partial x_{l}}$, which, in turn, are
products of $\frac{\partial {\widehat{\kappa}}_{ij}}{\partial x_{l}}$,
$\frac{{\partial}^2 {\widehat{\kappa}}_{ij}}{\partial x_{l}
\partial x_{v}}$,$\frac{\partial {\widehat{w}}_{ij}}{\partial x_{l}}$,
$\frac{{\partial}^2 \widehat{w}_{ij}}{\partial x_{l}
\partial x_{v}}$, and $\sin{\widehat{\kappa}}_{iu},
\;\sin{\widehat{\kappa}}_{mv}$ all summed over $l, v, j$. It is not
difficult to see that ${\gamma}^i_{jm|l_1}$ do not contribute anything
to the rank of ${\digamma}_{2}(\Psi)$ as opposed to a scenario wherein
some ${\gamma}^i_{jm}(0,0,\cdots,0)=0$ but the rank remains constant
due to the contribution of ${\gamma}^i_{jm|l_1}(0,0,\cdots,0)>0$. 
\end{proof}
\indent
Now we can scale all the ${\gamma}^i_{jm}$ that do not vanish
identically to be equal to one, and 
${\beta}_s \equiv 0 \mod {\Psi}$. $\bigoplus_{i=1}^{2k} 
SO_i(2, \mathbb{R})$ has been reduced to $\{e\}$. We fix $4k$ unique 
smooth functions $f_i,\;\;f_{i+2k}$ so that ${\beta}_i = f_i{\Psi}_i + 
f_{i+2k}{\Psi}_{i+2k}$, and we have an $e$-structure. All the 
requirements listed in the hypothesis of Background
Theorem~\ref{T: e-structure equivalence} are satisfied. 
That completes the proof.\\
\indent
To prove Corollary~\ref{T: corollary} we remark that in view of
Lemma~\ref{T: period matrix invariance}, the Riemann period matrices
are independent of the fiber variables. Now we claim that for 
an arbitrary ${\AA}: T_UU \longrightarrow T_VV$, the induced map of 
cotangent bundles is on target: ${\AA}^* (\Psi) \in \Xi (V, ds^2)$.
With the Main Theorem above, we know that ${\AA}(U)=V,\; {\AA}(U^N)
=V^N,\; {\AA}(U^O)=V^O$. Consequently, the periods, periodic spectra 
of the Hill operators for every metric tensor coefficient,
and all the norming constants are identical on $U$
and $V$. Now we think of ${\mathfrak{T}}_{\mathbb{C}}(U)$ as a 
fibration over a foliation of codimension $2k-1$, whose fibers are
Riemann surfaces (including the limiting case of 
$[(\mathbb{C}\backslash \{{\lambda}_i\})\times 
(\mathbb{C}\backslash \{{\lambda}_i\})]/_{\backsim}$, and fix one
leaf determined by the coordinates $(x_2, x_3, \cdots, x_{2k})$. 
However, the limiting fiber occurs only if the underlying
leaf does not intersect the loci of nondifferentiability/rank
deficiency. Formally,
$$(U^N \cup U^O) \cap \{(x_2, x_3, \cdots, x_{2k})|{\mathfrak{T}}_
{\mathbb{C}}(x_2, x_3, \cdots, x_{2k}) = [(\mathbb{C}\backslash 
\{{\lambda}_i\})\times (\mathbb{C}\backslash \{{\lambda}_i\})]
/_{\backsim} \}= \emptyset.$$
\noindent
Additionally, between transversality of the foliation (to the 
loci of nondifferentiability/rank deficiency) and the smoothness,
the only possible scenario would materialize if either
$\partial (U^N \cup U^O) \ne \emptyset$, or $\dim (U^N \cup U^O)
< 2k-1$. At any rate, the presence of constant potentials is 
completely determined by the Laplace-Beltrami equation.
These conditions imply that zero Riemann period matrices are
mapped onto zero period matrices by an arbitrary isometry,
and the induced biholomorphic map on the fiber is the identity.
One such example is an axisymmetric black hole with the incomplete
event horizon.\\
\indent
We have disposed of the limiting surfaces, so from this point on, the
fiber is a bona fide Riemann surface of infinite genus ${\mathfrak{T}}_
{\mathbb{C}}$. Now we truncate it so as to have only the largest
$m$ homology cycles $A_1, B_1,\cdots, A_m,B_m$ as the basis of its
relative homology group: ${\mathfrak{T}}_{\mathbb{C}}^m$ has
$H_1({\mathfrak{T}}_{\mathbb{C}}^m \backslash \partial
{\mathfrak{T}}_{\mathbb{C}}^m, \mathbb{Z})$ being generated by
the cycles $(A_1, B_1,\cdots, A_m,B_m)$. This is a compact Riemann
surface.\\
\indent
If ${\AA}^*({\mathfrak{T}}_{\mathbb{C}} (x_2, x_3, \cdots, x_{2k}))
\ncong {\mathfrak{T}}_{\mathbb{C}} (y_2, y_3, \cdots, y_{2k})$,
by Background Theorem~\ref{T: Torelli} (Torelli theorem) we have
${\AA}^*([{\mathfrak{R}}^{ij}]_{\beta \gamma}(x_2, x_3, \cdots, x_{2k}))
\ne [{\mathfrak{R}}^{ij}]_{\beta \gamma}(y_2, y_3, \cdots, y_{2k})$, 
and since all the periods are identical, there must exist a
holomorphic form  $d\Phi \in \text{Hol}({\mathfrak{T}}_{\mathbb{C}}
(x_2, x_3, \cdots, x_{2k}))$ or  $d\Phi \in \text{Hol}
({\mathfrak{T}}_{\mathbb{C}}(y_2, y_3, \cdots, y_{2k}))$, responsible
for the discrepancy. By Its-Matveev formula~\eqref{E: Its-Matveev},
$d \Phi$ translates into a particular Hill potential. On the real
torus ${\mathfrak{T}}^m (y_2, y_3, \cdots, y_{2k})$ it corresponds to
an $m$-tuple of tied spectrum eigenvalues $({\mu}_1, \cdots  {\mu}_m)$.
But that $m$-tuple is already accounted for by virtue of the fact that
all the lacunae are identical on $U$ and $V$. Hence we are compelled
to look for the offending ${\mu}_i$ someplace else, more specifically
at the cycles with $i>m$. Now given that $\exists \;d{\Phi}_m \in 
\text{Hol}({\mathfrak{T}}_{\mathbb{C}}^m)$ with
$$\oint_{A_i}d\Phi = \oint_{A_i}d{\Phi}_m,\;\;\; i\leqslant m,$$ 
we can apply the Gram-Schmidt orthogonalization algorithm to produce
$d{\Phi}^{\bot}$ such that $$\oint_{A_i}d{\Phi}^{\bot} =0, \;\;i 
\leqslant m.$$ Now we let $m \longrightarrow \infty$. Then by 
(\cite{FKT}, Chapter 1,Theorem 1.17), $d{\Phi}^{\bot}=0$ and we have 
reached a contradiction assuming $${\AA}^*({\mathfrak{T}}_{\mathbb{C}} 
(x_2, x_3, \cdots,x_{2k})) \ncong {\mathfrak{T}}_{\mathbb{C}} 
(y_2, y_3, \cdots, y_{2k}).$$

\subsection{Weak equivalence}
\paragraph{}
Having set the stage by describing in great detail the problem of equivalence
of nondifferentiable metrics, we now proceed to take a closer look at an
approximation problem of nondifferentiable metrics. Is there a way to tell
if two such metrics expressed in terms of the same local coordinate system
are so `close' to each other that their first derivatives match? \\
\indent
Courant algebroids to the rescue! Recall~\eqref{E: bilinear forms}:
\begin{equation}
(X_1 + {\xi}_1, X_2 + {\xi}_2)_{+} =
\frac{1}{2}(\langle {\xi}_1, X_2 \rangle + \langle {\xi}_2, X_1 \rangle).
\end{equation} 
\noindent
This is a natural nondegenerate bilinear form associated with Courant
algebroids. It allows one to circumscribe the extent of nonintegrability
of a preframe. 
\begin{defn}
Given a preframe $\vec{\Psi}$, we define its annihilator, 
$\vec{{\Psi}^{\bot}}$ by the equation
$$(\vec{\Psi},\; \vec{{\Psi}^{\bot}})_+ =
([\vec{\Psi},\;\vec{\Psi}],\;\vec{{\Psi}^{\bot}})_+ =0.$$
\end{defn}
\noindent
Thus an annihilator is related to the original preframe.
\begin{defn}
Two preframes,  ${\vec{\Psi}}^1, \; {\vec{\Psi}}^2$ are weakly
equivalent if there exists a nontrivial annihilator $ \vec{{\Psi}^{\bot}}$
such that for some constants ${\beta}_1, {\beta}_2$ we obtain
\begin{align*}
&({\vec{\Psi}}^1,\; \cos {\beta}_1\vec{{\Psi}^{\bot}} + \sin {\beta}_1
({\omega}^{\#} - {\pi}^{\#})\vec{{\Psi}^{\bot}})_+\\ 
=&({\vec{\Psi}}^2,\;\cos {\beta}_2\vec{{\Psi}^{\bot}} + \sin {\beta}_2
({\omega}^{\#} - {\pi}^{\#})\vec{{\Psi}^{\bot}} )_+ \\
=&([{\vec{\Psi}}^1,{\vec{\Psi}}^1],\; \cos {\beta}_1\vec{{\Psi}^{\bot}} 
+ \sin {\beta}_1({\omega}^{\#} - {\pi}^{\#})\vec{{\Psi}^{\bot}})_+\\ 
=&([{\vec{\Psi}}^2,{\vec{\Psi}}^2],\; \cos {\beta}_2\vec{{\Psi}^{\bot}} 
+ \sin {\beta}_2({\omega}^{\#} - {\pi}^{\#})\vec{{\Psi}^{\bot}})_+= 0.
\end{align*}
\end{defn} 
\indent
Knowing that the structure group can be reduced explains our definition.
The constant preternatural rotations should suffice.
The number of independent elements of the Courant algebroid that 
may possibly comprise an annihilator ranges from zero to $2k$. 
Maximally nonintegrable preframes are extremely picky and hard to
annihilate, whereas sections of Dirac subbundles allow $2k$-dimensional 
annihilators. That number is invariant under the action of the bundle 
map ${\omega}^{\#}- {\pi}^{\#}$. One instructive case is that of preframes 
that happen to be $2k$-tuples of linearly independent sections of a typical 
Dirac subbundle. Any two of those are weakly equivalent. It makes sense since 
they all induce identically vanishing Riemann period matrices on open subsets.\\
\indent
Obviously, equivalent preframes are weakly equivalent. That readily
follows from~\eqref{E: structure group}. It is unclear
whether real analytic weakly equivalent preframes are equivalent.\\

\section{Black Hole Solutions of EVE}
\paragraph*{}
Throughout this section, we work with two basic objects:
The Einstein vacuum equations, under the acronym EVE\footnote{Strictly
speaking, Sections 5.2 and 5.4 presuppose an observer with a mass 
crossing the event horizon, and gravitational collapse, both of which 
take the full-fledged Einstein equations},
$$ \text{Ric}(g)=0,$$
\noindent
and a special sort of open subsets of the space-time.
Specifically, for a black hole solution of EVE, we make
\begin{defn}
An open set $U \subset {\mathbb{R}}^{1+3}$ is called
\textit{representative} if for a singular preframe
encoding the black hole solution, there is an element
such that $U^N_{ij} \cap U \ne \emptyset $ (an inner curvature blow-up), and
$U^O_{ij} \cap U \ne \emptyset$ (the Cauchy horizon).
\end{defn}
\noindent
The significance of our choice is obvious: we intend to make use
of the new metric invariants in general relativity. Admittedly, the
condition of `general covariance' imposed by Einstein involves
the entire group of diffeomorphisms, so metric invariants,
the Riemann period matrices among them, are not preserved. But the
incidence relations associated with two metric configurations do
persist.\\

\subsection{The Kerr Preframe}
\paragraph*{}
The Kerr solution~\cite{Kerr}, discovered by Roy Kerr in 1963, is,
perhaps, the most interesting of all exact solutions of EVE. It describes
a rotating axisymmetric black hole, and 
provides a natural testing ground (space?) for preframe theory.
Hence we work out the details, including the blowup parameter, and
a representative set, explicitly. As a primary source, we utilize
the book~\cite{WVS}, predominantly Chapter 1, written by Matt Visser.\\
\indent
The second version (and the one we are interested in) of the Kerr line 
element presented in~\cite{Kerr} was defined in terms of `Cartesian'
coordinates $(t,x, y, z)$:
\begin{alignat}{2}\label{E: Kerr}
ds^2 =\;\;&-dt^2 + dx^2 + dy^2 + dz^2 \\
&+ \frac{2mr^3}{r^4 + a^2z^2}\left[dt + \frac{r(xdx + ydy)}{a^2 +r^2} +
\frac{a(ydx-xdy)}{a^2 +r^2} + \frac{z}{r}dz \right]^2, \notag
\end{alignat}
\noindent
subject to $r(x,y,z)$, which is now a dependent function, and not a
radial coordinate, being implicitly determined by:
\begin{equation}
x^2 +y^2 +z^2 = r^2 + a^2\left[ 1- \frac{z^2}{r^2} \right].
\end{equation}
\noindent
The angular momentum $J$ is incorporated into the line element~(\ref{E: Kerr})
via $a=\frac{J}{m}$, $m$ being the mass.\\
\indent
The full $0<a<m$ metric is now manifestly of the Kerr-Schild form:
\begin{equation}
g_{ij} = {\eta}_{ij} + \frac{2mr^3}{r^4 + a^2z^2}{\pounds}_i{\pounds}_j,
\end{equation}
\begin{equation}
{\pounds}_i=\left(1,\;\frac{rx+ay}{r^2+a^2},\;\frac{ry-ax}{r^2+a^2},\;
\frac{z}{r}\right).
\end{equation}
\noindent
Here ${\pounds}_i$ is a null vector with respect to both $g_{ij}$ and
${\eta}_{ij}$.\\
\indent
To delineate a standard representative set $U_{\text{Kerr}}$, we need to restrict
our base manifold to a more manageable (and physically relevant) subset $r \geqslant 0$.
Next, we choose $z$ as the blowup parameter. Our choice is informed by the fact that
the Kerr spacetime is naturally axisymmetric, and its symmetry group is $SO(2)$.
The crucial (but not all) loci are as follows 
$$U^0_{\text{Kerr}} =\{(t,x,y,z) \in {\mathbb{R}}^{1+3}| x^2 +y^2 = a^2,\; z=0\},$$ 
$$U^{1,\;2}_{\text{Kerr}} =\{(t,x,y,z) \in {\mathbb{R}}^{1+3}| x^2 +y^2 + 
\frac{2m}{{r}_{\pm}}z^2 = 2m{r}_{\pm}\}.$$
All the loci have to be transversal to the $z$-direction. To include at least 
some parts of these three sets, we choose our slice of heaven to be
\begin{equation}\label{E: rep set}
U = \{(t,x,y,z) \in {\mathbb{R}}^{1+3}|0<t<1,\;\;
0<x,\; y < \frac{\sqrt{2}a}{2} + \epsilon,\;\;\;z^2 < (r_+ + \epsilon)^2 \}.  
\end{equation}
\indent
The remaining vanishing loci are
$$U^3_{\text{Kerr}} = \{(t,x,y,z) \in {\mathbb{R}}^{1+3}|rx +ay =0\},$$
$$U^4_{\text{Kerr}} = \{(t,x,y,z) \in {\mathbb{R}}^{1+3}|ry -ax =0\},$$
$$U^5_{\text{Kerr}}= \{(x,y,z) \in {\mathbb{R}}^{1+3}| x^2 + y^2 > a^2,\;\;z=0 \}.$$
\noindent
They are comprised of ellipses and circles, one each for a fixed value of $r$.\\
\indent
In terms of individual elements, we have
\begin{equation}
U^O_{00} = U^{1,\;2}_{\text{Kerr}}; \;\;\;\; 
U^N_{00}=U^0_{\text{Kerr}}; \;\;\;\;U^N_{ij} = U^0_{\text{Kerr}};                
\end{equation}
\begin{equation}
U^O_{01} = U^3_{\text{Kerr}};\;\;\;U^O_{02} = U^4_{\text{Kerr}};\;\;\;
U^O_{03} =U^5_{\text{Kerr}}.
\end{equation}
\indent
The rest are just unions:
\begin{equation}
U^O_{13} =  U^3_{\text{Kerr}} \cup U^5_{\text{Kerr}};\;\;\;
U^O_{23} =  U^4_{\text{Kerr}} \cup U^5_{\text{Kerr}};\;\;\;U^O_{12} = 
U^3_{\text{Kerr}} \cup U^4_{\text{Kerr}}.
\end{equation}
\indent
Individual standard normal form coefficients are as follows:
\begin{equation}
e^{2w_{00}}{\cos}^2{\kappa}_{00}\bigr|_{U\cap (U^N \cup U^O_{00})} = 
\biggr|\frac{\sin(2mr^3 -r^4 -a^2z^2)}{2mr^3 -r^4 -a^2z^2}
\frac{\sin(r^4 +a^2z^2)}{r^4 +a^2z^2} \biggr|.
\end{equation}
\noindent
Hence, prior to smoothing, we find that $\sin({\kappa}_{00}(x,y,z))$ is
an eigenfunction of the Hill potential $\cos(\frac{z}{T_{00}(x,y)})$, 
with the designated eigenvalue indexed ${\lambda}^{00}_{12}(x,y)$.
\begin{equation}
e^{2w_{01}}{\cos}^2{\kappa}_{01}\bigr|_{U\cap (U^N \cup U^O_{01})} = 
\biggr|\frac{\sin(2mr^3(rx +ay))}{2mr^3(rx + ay)}
\frac{\sin((r^4 +a^2z^2)(r^2 + a^2))}{(r^4 +a^2z^2)(r^2 + a^2)}\biggr|.
\end{equation}
\noindent
The function $\sin({\kappa}_{01}(x,y,z))$ has
exactly two roots in $U$. Thus, we obtain ${\lambda}^{01}_{4}(x,y)$. 
\begin{equation}
e^{2w_{02}}{\cos}^2{\kappa}_{02}\bigr|_{U\cap (U^N \cup U^O_{02})} = 
\biggr|\frac{\sin(2mr^3(ry -ax))}{2mr^3(ry - ax)}
\frac{\sin((r^4 +a^2z^2)(r^2 + a^2))}{(r^4 +a^2z^2)(r^2 + a^2)}\biggr|.
\end{equation}
\noindent
$\sin({\kappa}_{02}(x,y,z))$ has two roots, hence ${\lambda}^{02}_{4}(x,y)$.
\begin{equation}
e^{2w_{03}}{\cos}^2{\kappa}_{03}\bigr|_{U\cap (U^N \cup U^O_{03})} = 
\biggr|\frac{\sin(2mr^2z)}{2mr^2z}\frac{\sin(r^4 +a^2z^2)}{(r^4 +a^2z^2)}\biggr|.
\end{equation}
\noindent
$\sin({\kappa}_{03}(x,y,z))$ features a coalescence of roots, $\{z=0\}$ being
their common locus. That corresponds to ${\lambda}^{03}_{2}(x,y)$. 
Its behavior is singularly simple:
$$\sin({\kappa}_{03}(x,y,z))\bigr|_{|z|> \epsilon} =\frac{\sqrt{2}}{2},
\;\;\;\;\sin({\kappa}_{03}(x,y,z))\bigr|_{z=0} =0.$$
\begin{align}
e^{2w_{11}}{\cos}^2{\kappa}_{11}\bigr|_{U\cap (U^N \cup U^O_{11})} &= 
\biggr|\frac{\sin(2mr^3(rx +ay)^2+(r^2 +a^2)^2)}
{2mr^3(rx + ay)^2+ (r^2 + a^2)^2}\\
&\times \frac{\sin((r^4 +a^2z^2)(r^2 + a^2)^2)}
{(r^4 +a^2z^2)(r^2 + a^2)^2}\biggr|.\notag
\end{align}
$\sin({\kappa}_{11}(x,y,z))$ corresponds to ${\lambda}^{11}_{2}(x,y)$
by virtue of $U^O_{11} =\emptyset$.
\begin{align}
e^{2w_{12}}{\cos}^2{\kappa}_{12}\bigr|_{U\cap (U^N \cup U^O_{12})} &= 
\biggr|\frac{\sin(2mr^3(rx +ay)(ry -ax)}{2mr^3(rx + ay)(ry -ax)}\\
&\times \frac{\sin((r^4 +a^2z^2)(r^2 + a^2)^2)}
{(r^4 +a^2z^2)(r^2 + a^2)^2}\biggr|.\notag
\end{align}
\noindent
This coefficient takes ${\lambda}^{12}_{6}(x,y)$, 
as $U^3_{\text{Kerr}} \cap U^4_{\text{Kerr}}\cap U = \emptyset$.\\
\begin{align}
e^{2w_{13}}{\cos}^2{\kappa}_{13}\bigr|_{U\cap (U^N \cup U^O_{13})} &= 
\biggr|\frac{\sin(2mr^2(rx +ay)z)}{2mr^2(rx + ay)z}\\
&\times \frac{\sin((r^4 +a^2z^2)(r^2 + a^2)^2)}
{(r^4 +a^2z^2)(r^2 + a^2)^2}\biggr|.\notag
\end{align}
\noindent
$\sin({\kappa}_{13}(x,y,z))$ corresponds to ${\lambda}^{13}_{6}(x,y)$.
\begin{align}
e^{2w_{22}}{\cos}^2{\kappa}_{22}\bigr|_{U\cap (U^N \cup U^O_{22})} &= 
\biggr|\frac{\sin(2mr^3(ry -ax)^2+(r^2 +a^2)^2)}{2mr^3(ry -ax)^2+(r^2 +a^2)^2}\\
&\times \frac{\sin((r^4 +a^2z^2)(r^2 + a^2)^2)}
{(r^4 +a^2z^2)(r^2 + a^2)^2}\biggr|.\notag
\end{align}
\noindent
$\sin({\kappa}_{22}(x,y,z))$ corresponds to ${\lambda}^{22}_{2}(x,y)$ in
view of the fact that $U^O_{22} =\emptyset$.
\begin{align}
e^{2w_{23}}{\cos}^2{\kappa}_{23}\bigr|_{U\cap (U^N \cup U^O_{23})} &= 
\biggr|\frac{\sin(2mr^2z(ry -ax))}{2mr^2z(ry -ax)}\\
&\times \frac{\sin(r^4 +a^2z^2)(r^2 + a^2)^2)}
{(r^4 +a^2z^2)(r^2 + a^2)^2}\biggr|.\notag
\end{align}
\noindent
$\sin({\kappa}_{23}(x,y,z))$ takes ${\lambda}^{23}_{6}(x,y)$.
\begin{equation}
e^{2w_{33}}{\cos}^2{\kappa}_{33}\bigr|_{U\cap (U^N \cup U^O_{33})}= 
\biggr|\frac{\sin(2mrz^2 +r^4 +a^2z^2)}{2mrz^2 +r^4+a^2z^2}
\frac{\sin(r^4 +a^2z^2)}{(r^4 +a^2z^2)}\biggr|.
\end{equation}
\noindent
Lastly, $\sin({\kappa}_{33}(x,y,z))$ takes ${\lambda}^{33}_{2}(x,y)$
as we realize that $U^O_{33} =\emptyset$.\\
\indent
Those are partially regularized functions. As yet, they are still
to undergo further treatment. As ${\square}_{\text{Kerr}}$ and the
smoothing operators are axisymmetric, the fully regularized 
functions would satisfy
$${\mathcal{L}}_{f_1K_1 + f_2K_2}e^{w_{ij}(x,y,z)}\cos({\kappa}_{ij}(x,y,z))=0,
\;\;\;{\mathcal{L}}_{f_1K_1 + f_2K_2}{\lambda}^{ij}_{2s}(x,y)=0,$$
where $K_1$, $K_2$ are the original Killing vectors. \\
\indent
We fix the symplectic form: $\omega = dt \wedge dz + dx \wedge dy$.
Then the preframe becomes
\begin{align}
{\varPsi}_i = &e^{w_{i0}(x,y,z)}\cos ({\kappa}_{i0}(x,y,z))dt +
e^{w_{i0}(x,y,z)}\sin ({\kappa}_{i0}(x,y,z))\ddz+\\\notag
&e^{w_{i1}(x,y,z)}\cos ({\kappa}_{i1}(x,y,z))dx +
e^{w_{i1}(x,y,z)}\sin ({\kappa}_{i1}(x,y,z))\ddy+\\\notag
&e^{w_{i2}(x,y,z)}\cos ({\kappa}_{i2}(x,y,z))dy -
e^{w_{i2}(x,y,z)}\sin ({\kappa}_{i2}(x,y,z))\ddx +\\\notag
&e^{w_{i3}(x,y,z)}\cos ({\kappa}_{i3}(x,y,z))dz -
e^{w_{i3}(x,y,z)}\cos ({\kappa}_{i3}(x,y,z))\frac{\partial}{\partial t};
\notag
\end{align}
\begin{equation}
\vec{\varPsi}=
\left[\begin{array}{cc}
{\varPsi}_0\\
{\varPsi}_1\\
{\varPsi}_2\\
{\varPsi}_3
\end{array}\right]
\end{equation}
\noindent
And the complementary preframe is gotten via the canonical bundle map:
$$ \vec{{\varPsi}^*} = ({\omega}^{\#}-{\pi}^{\#}) \vec{\varPsi}.$$

\subsection{Strong Cosmic Censorship Conjecture}
\paragraph{}
With hyperbolic PDE on Lorenzian manifolds, there comes the 
attendant causal structure (a Cauchy hypersurface $\mathcal{S}$ 
with the past Cauchy development $\mathcal{S}^{\swarrow}$ and the 
future Cauchy development $\mathcal{S}^{\nearrow}$). The existence 
of such a hypersurface is tantamount to $\mathcal{S}^{\swarrow} 
\cup \mathcal{S} \cup \mathcal{S}^{\nearrow} = M$, the entire 
manifold. If, however, $\mathcal{S}^{\swarrow} \cup \mathcal{S} \cup 
\mathcal{S}^{\nearrow} \neq M$, there exists a Cauchy horizon - a
light-like boundary of the domain of validity of the underlying
Cauchy problem for the PDE in question. One side of the horizon 
contains closed space-like geodesics, whereas the opposite side
may contain closed time-like geodesics. Beyond the horizon,
Cauchy data no longer uniquely determines the solution.
Thus the presence of such a horizon is manifestly a violation of 
classical determinism for all observers passing through the horizon. 
Some black holes, particularly the Kerr solutions with $0<a<m,$ 
are known to feature a distinct Cauchy horizon. Hence, to save 
classical determinism, Roger Penrose~\cite{Pen2} put forth the cosmic 
censorship conjectures. They are classified into two categories: weak 
and strong ones. We focus on the strong cosmic censorship conjecture.
There are several different formulations, the one we discuss below
states that the maximal Cauchy development of generic compact or 
asymptotically flat initial data is locally inextendible as a 
$C^0$-Lorentzian manifold. The original formulation of the strong 
cosmic censorship conjecture was modelled on the prototype of the
Schwarzschild solution $(a=0),$ where no Cauchy horizon is present,
and the spacetime can be construed to `terminate' at a spacelike
curvature nondifferentiability locus, the metric being 
inextendible past the nondifferentiability surface, not even
continuously. As Dafermos and Luk~\cite{DL} put it: "The singular
behavior of Schwarzschild, though fatal for reckless observers
entering the black hole, can be thought of as epistemologically
preferable for general relativity as a theory, since this ensures
that the future, however bleak, is indeed determined". In that   
paper, dating back to 2017, the authors disproved the strong
cosmic censorship conjecture formulation above for the Cauchy 
horizon of a charged, rotating black hole. Furthermore, according
to~\cite{DL}, the $C^0$-extensions are patently nonunique.\\
\indent
However, a modicum of uniqueness can be recovered for the black
hole solutions with Cauchy horizons. It should be emphasized,
this only applies to the structure of spacetime, not 
relativistic dynamics on the black hole background. We state
\begin{thm}\label{T: uniqueness of black holes}
Within an assertedly representative set $U$, the outer part of the
tubular neighborhood of $U^O_{ij} \cap U$ determines
the Riemann period matrices $[[{\mathfrak{R}}^{ij}]_{\beta \gamma}]_U$
of the regularized singular preframe.
\end{thm}
\begin{proof}
Here we already have the complete Cauchy horizon. It is topologically 
a sphere. The undetermined part is the locus
of nondifferentiability that hosts the curvature blow-up. However,
all periods would be fixed once we settle on a representative set.
Now using all conceivable extensions past the Cauchy horizon, we 
arrive at multiple regularized auxiliary fields for every such
extension following our Laplace-Beltrami $\longrightarrow$ 
Nash-Gromov $\longrightarrow$ Hill procedure. Then we turn off 
the beaten path and minimize over the extensions one field at a time 
mimicking~\eqref{E: optimized potential}:
\begin{equation}
\lim_{N \rightarrow \infty} \liminf_{(\text{extensions},\;S_{\zeta},\;
{\kappa})} (||{\lambda}_{0}^{\text{c}} - {\lambda}_{0}^{}||^2_U  
+\sum_{s=1}^{N}\frac{1}{2^s} (||{\lambda}_{2s}^{\text{c}} - 
{\lambda}_{2s-1}^{}||^2_U + ||{\lambda}_{2s}^{\text{c}} - 
{\lambda}_{2s}^{}||^2_U)).
\end{equation}
\noindent
This piecemeal optimization scheme may yield a configuration 
that is not equal to any of the existing extensions, but we
still find optimized extensions by melding those suboptimal
ones using partitions of unity. They are equivalent in view of
Theorem~\ref{T: spectrum}. 
\end{proof}
\indent
Our eclectic optimal extension takes advantage of the fact 
that there no longer are relations between individual metric 
coefficients. $C^0$-configurations are not solutions of EVE.\\
\indent
The similarity between Theorem~\ref{T: uniqueness of black holes}
and the holographic principle championed by some physicists is
purely coincidental. We knew from the outset that the locus of 
curvature blow-up must be present inside.\\
\indent
Theorem~\ref{T: uniqueness of black holes} does not make or
reinforce an argument for the validity of the strong cosmic
censorship conjecture. More precisely, in conjunction with the Penrose
singularity theorem~\cite{Pen}, as the prevalence of black holes had
been established, our result is a way to safeguard uniqueness. Thus,
owing to Penrose, we look for a black hole solution of EVE. Then the 
periodicity of the Hill potentials and the structure of the Cauchy
horizon determines the location and structure of $U^N_{ij} \cap U$.\\ 
\indent
A much more difficult question still remains unaddressed. With closed
timelike geodesics stubbornly persisting beyond the Cauchy horizon, no
semblance of causality is possible. And the results of~\cite{DL} only
bring this unsatisfactory state of affairs into sharper focus.
Our proposal to reestablish causality parts ways with general relativity.
Rather than trying to finesse some ever more sophisticated formulation
of the cosmic censorship, we look at the Riemann period matrices
corresponding to representative sets with masses/observers moving inside
the Cauchy horizon. For each instance (of an affine geodesic parameter), 
due to the gravitational pull of the mass/observer, we get a particular 
set of functions of the Riemann period matrix. 
Thus each instance produces a different configuration.
If we compare this situation with that of a mass/observer living in a
smooth metric configuration, we see that the latter (in view of 
identically zero Riemann period matrix) does not induce this kind of 
inequivalence (the curvatures do change, but the finer invariants 
do not). Therefore this fragmentation of physical reality
resolves the problem of causality in the presence of loci of 
nondifferentiability. Within our paradigm, physical movement always
takes place in an internal parameter-independent Riemann period matrix 
environment, as evidenced by dynamics on smooth Lorenzian manifolds. 
Such dynamics in general do not exist beyond the Cauchy horizon. 
Intrepid observers are the proverbial elephants rambling in the 
china shop of the black hole interior. \\
\indent
Mathematically, it makes sense: a mass moving on a regular background
simply reshuffles sections of a fixed Lagrangian subbundle of 
$T^*{\mathbb{R}}^{1+3}$, whereas the same mass' influence on a lift
of some regularized preframe impinges on local integrability properties
of the underlying subbundle.\\

\subsection{Dynamics in the Interior}
\paragraph*{}
As we postulate, the presence of nonzero Riemann period matrix 
radically alters dynamics, to the point of excluding movement of
masses along generic timelike geodesics. However, there still remains 
one avenue open: moving masses without disturbing the invariants. Owing
to Torelli theorem (Background Theorem~\ref{T: Torelli}), that amounts 
to conformal transformations of individual Riemann surfaces within 
${\mathfrak{T}}_{\mathbb{C}}(U)$ as the affine parameter transversal to the 
$x_1$-direction unfolds. Now on the moduli space of the Riemann 
surfaces encoding the spectrum of relevant Hill operators, we can 
define a monodromy group of automorphisms. There would be a 
representation of the set of timelike geodesics on the aforementioned 
group. The dimension of the kernel of this representation map is what 
we call \textit{the effective dimension}. Our conjecture is, the 
effective dimension is strictly less than three, and all closed 
geodesics are in the image of that representation. Also, the dimension 
may depend on the ratio of masses of the observer and the black hole
being observed. The reason we cannot dismiss the possibility
of nondisturbing mass movement is, absorption of gravitational waves by a
black hole must induce a conformal transformation of 
${\mathfrak{T}}_{\mathbb{C}}(U)$.

\subsection{Black Hole Evaporation}
\paragraph*{}
We inadvertently venture into the quantum theory territory.
The presence of additional differential invariants sheds some
light on the process of evaporation (via emission of thermal
radiation) inaugurated by Stephen Hawking in~\cite{Hawk, HaHa}. 
Instead of sticking with the full-fledged Riemann surface invariants, 
we deconstruct the discriminant~\eqref{D: discriminant} of a 
pertinent Hill operator  $\bigstar (\lambda, t, x,y)$ on a representative 
open set $U = \{(t,x,y,z) \in {\mathbb{R}}^{1+3}\}$. For uniformity, 
we regard $z$ as the singularity parameter. Specifically, 
$\bigstar (\lambda, t, x,y)$ is an integral function of $\lambda$, growing
as $\sqrt{\lambda}$ as $\lambda \longrightarrow \infty$ (\cite{McT}, Section 1).
As such, it allows an infinite product representation:
\begin{equation}
\bigstar (\lambda, t, x,y) = v(\lambda, t,x,y)\prod_{i=0}^{\infty}
(\lambda -{\lambda}_{2i})(\lambda -{\lambda}_{2i-1}).
\end{equation}
\noindent
This  product converges absolutely, so by selecting a principal
branch of $ \sqrt{v(\lambda, t,x,y)}$ over some fixed ramified cover,
we define
\begin{equation}
\bigstar^+ (\lambda, t, x,y) \overset{\textnormal{def}}{=} 
\sqrt{v(\lambda, t,x,y)}\prod_{i=0}^{\infty}(\lambda -{\lambda}_{2i}),
\end{equation}
\begin{equation}
\bigstar^- (\lambda, t, x,y) \overset{\textnormal{def}}{=} 
\sqrt{v(\lambda, t,x,y)}\prod_{i=0}^{\infty}(\lambda -{\lambda}_{2i-1}).
\end{equation}
If ${\lambda}_{2i-1} = {\lambda}_{2i}$, we call this a vacuum configuration
(in conjunction with a Hamiltonian preframe on $U$), or a regular
configuration (in conjunction with  a regular but not a Hamiltonian 
preframe). Otherwise, this would designate a piece of some black hole 
if and only if ${\lambda}_{2i-1} < {\lambda}_{2i}$. The only possibility 
left, ${\lambda}_{2i-1} > {\lambda}_{2i}$, has no geometrical content,
although physicists would be tempted to call such monstrosities
`negative energy vacuum fluctuations'.\\
\indent
To effect a semi-rigorous derivation of black hole evaporation,
Hawking~\cite{Hawk} used the following reasoning: ``... negative
energy flux will cause the area of the event horizon to decrease
and so the black hole will not, in fact, be in a stationary state.
However, as long as the mass of the black hole is large compared
to the Planck mass, the rate of evolution of the black hole will
be very slow compared to the characteristic time for light to
cross the Schwarzschild radius. Thus it is a reasonable 
approximation to describe the black hole by a sequence of stationary
solutions and to calculate the rate of emission in each solution.''
Therefore we let the discriminant to be time-independent, and
the changes be induced by some abstract time-dependent operators. 
Mathematically, they are Hermitian integral operators with compact kernels 
acting on ${\bigstar}^+ (\lambda, 0,x,y)$ and ${\bigstar}^- (\lambda, 0,x,y)$: 
\begin{equation}
(\mathfrak{A}(t_m){\bigstar}^{\pm})({\lambda}, 0,x,y) = \int_{\mathfrak{T}(U)}
H(t_m,x-\chi,y-\upsilon){\bigstar}^{\pm} (l,0,\chi,\upsilon)dl d\chi d\upsilon,
\end{equation}
\begin{equation}
(\mathfrak{E}(t_m){\bigstar}^{\pm})({\lambda}, 0,x,y) = \int_{\mathfrak{T}(U)}
\bar{H}(t_m,x-\chi,y-\upsilon){\bigstar}^{\pm} (l,0,\chi,\upsilon)dl d\chi d\upsilon.
\end{equation}
\noindent
Such operators must be invertible by virtue of the fact their natural
extensions to the sections of the cotangent bundle induce
isomorphisms of the respective Hodge-Kodaira cohomology spaces:
$$H^1_{\textnormal{HK}}(\mathfrak{A}(0)\mathfrak{T}(U)) \cong
H^1_{\textnormal{HK}}(\mathfrak{A}(t_m)\mathfrak{T}(U)),\;\;
H^1_{\textnormal{HK}}(\mathfrak{E}(0)\mathfrak{T}(U)) \cong
H^1_{\textnormal{HK}}(\mathfrak{E}(t_m)\mathfrak{T}(U)).$$
\noindent 
Their main feature is stretching/shrinking of the intervals of 
instability of the periodic spectrum that vary depending on the absorbed 
mass, charge, spin, and emitted energy (written abusing notation):
\begin{equation}
\mathfrak{A}{\lambda}_{2i} - \mathfrak{A}{\lambda}_{2i-1} >
{\lambda}_{2i} - {\lambda}_{2i-1},
\end{equation}
\begin{equation}
\mathfrak{E}{\lambda}_{2i} - \mathfrak{E}{\lambda}_{2i-1} <
{\lambda}_{2i} - {\lambda}_{2i-1}.
\end{equation}
\indent
This reformulation opens a veritable Pandora box, as there is no
reason to believe that all $\mathfrak{A}(t_m){\bigstar}^{\pm}$ have
an underlying stationary black hole solution of EVE. That is why it
was not an empty declaration to classify this as a quantum theory.\\
\indent
The gist of the information loss paradox can be recast in terms
of the absorbtion/emission operators. Namely, the chronological
products $$\cdots \circ\mathfrak{E}(t_m)\circ\mathfrak{E}(t_{m-1})
\circ \cdots \circ \mathfrak{E}(t_1)$$ may have commuting operators, 
in which case the history path of a black hole would be nonunique,
and, as a consequence, noninvertible, contradicting one of the
central tenets of quantum theory. Hence, from that standpoint,
their brackets under composition ought not to vanish:
$$ [\mathfrak{A}(t_1), \mathfrak{A}(t_2)] =
\mathfrak{A}(t_1)\circ\mathfrak{A}(t_2)-
\mathfrak{A}(t_2)\circ\mathfrak{A}(t_1) \neq 0,$$
$$ [\mathfrak{E}(t_1), \mathfrak{E}(t_2)] =
\mathfrak{E}(t_1)\circ\mathfrak{E}(t_2)-
\mathfrak{E}(t_2)\circ\mathfrak{E}(t_1) \neq 0.$$
That is a necessary condition. So the operator algebra of 
$\{\mathfrak{A}(t),\;\mathfrak{E}(t)\}$ is properly defined over
the field of quaternions, so that the chronological
products are unique (hence invertible via substitution
$ \mathfrak{E}(t_m)\mapsto \mathfrak{A}(t_{n-m})$).  
Allowing mixed chronological products would lead
to negative energy vacuum fluctuations. The time degenerates 
into a discretely-valued parameter. Precisely, we set the 
time interval to be bounded below by the Planck time constant:
\begin{equation}
t_{m+1} -t_m \geqslant t_P =\sqrt{\frac{\hbar G}{c^5}}.
\end{equation}
\indent
To cram more information into the chronological products, one 
may want to define this operator algebra over the field of 
octonions thus making it nonassociative.

\subsection{No-Go Theorem}
\paragraph*{}
In a deterministic Universe one should be able to decide 
whether a gravitational collapse will result in a known
stationary black hole well in advance. Any conceivable litmus 
test has to do with the initial mass distribution, angular 
momentum, charge. We present one such test based on the 
differential invariants of Section 4. Since the moment in 
time of the collapsing matter settling down to a stationary 
EVE solution $t_{\textnormal{still}}$ depends crucially on 
the radius of the initial mass distribution, we can modify
our representative set~\eqref{E: rep set} to last long
enough to accommodate a large class of gravitational collapse
scenarios just by keeping the matter within a fixed radius:
\begin{equation}\label{E: collapse}
U = \{ 0 <t <t_{\textnormal{still}} + \epsilon,\;\; 0<x,\; y < 
\frac{\sqrt{2}a}{2} + \epsilon,\;\;\;z^2 < (r_+ + \epsilon)^2 \}.  
\end{equation}
\noindent
Furthermore, we confine those bounded mass distributions to
have an $SO(2)$ symmetry configuration, so that the Kerr 
solution is a distinct collapse endpoint. Specifically, we state
\begin{thm}
For every eigenvalue of the (parametric) Hill operator 
${\lambda}^{\textnormal{Kerr}}_{2s}(x,y)$, there exists a 
function $F^{\textnormal{Kerr}}_{2s}(t,x,y)$ defined on the 
Kerr representative set~\eqref{E: collapse} such that its values 
antedating $t_{\textnormal{still}}$ are sufficient to predict 
whether an asymptotically flat axisymmetric gravitational 
collapse configuration will settle down to the Kerr spacetime.
\end{thm}
\begin{proof}
Our proof utilizes the universal oscillator equation. As matter
collapses, the second time derivative of the attendant periodic
spectrum discriminant (formally introduced in~\eqref{D: discriminant})
undergoes some oscillations. A simple mechanical analogue would be
a pendulum with some time-dependent restoring force, which grows in
magnitude from zero to a maximum at the time $t_{\textnormal{max}} <
t_{\textnormal{still}}$, then gradually returns to zero at 
$t_{\textnormal{still}}$. The fact that this oscillator is anharmonic
is encapsulated by the presence of the (unknown) damping function
$F^{\textnormal{Kerr}}_{2s}(t,x,y)$:
\begin{equation}
[\frac{{\partial}^2}{\partial t^2} + 2F^{\textnormal{Kerr}}_{2s}(t,x,y)
\frac{\partial}{\partial t} + 1]\frac{{\partial}^2 {\bigstar}
(t,x,y, {\lambda}^{\textnormal{Kerr}}_{2s}(x,y))}{\partial t^2} =0.
\end{equation}
\noindent
We are primarily interested in the behavior of solutions (known via
our Laplace-Beltrami $\longrightarrow$ Nash-Gromov $\longrightarrow$ Hill 
procedure) on the interval $[0, t_{\textnormal{max}}]$, as the essential
features of the (pending) Kerr black hole would be revealed early.\\
\indent
Using the properties of our ODE, we decompose solutions into the 
steady-state factor (de facto an amplitude), and the oscillatory part:
\begin{equation}
\frac{{\partial}^2 {\bigstar}
(t,x,y, {\lambda}^{\textnormal{Kerr}}_{2s}(x,y))}{\partial t^2} =
A_{\frac{{\partial}^2 {\bigstar}}{\partial t^2}} \cos(at + \theta).
\end{equation}
\noindent
The steady-state factor is expressly a function of the damping. That 
allows us to apply the inverse function theorem in a neighborhood of
$t=0$, where the damping function dominates. As a result, we obtain
$F^{\textnormal{Kerr}}_{2s}(A_{\frac{{\partial}^2 {\bigstar}}
{\partial t^2}})$. It needs to be reiterated, this procedure does not
depend on $t_{\textnormal{max}}$, a fortiori  $t_{\textnormal{still}}$.\\
\indent
Now we approximate by polynomials to compute first few polynomial 
coefficients of the expansion of $F^{\textnormal{Kerr}}_{2s}$
in terms of the steady-state factor. Since those coefficients are linked
to the eigenvalue ${\lambda}^{\textnormal{Kerr}}_{2s}(x,y)$, it follows
that an abstract discriminant ${\bigstar}(t,x,y, {\lambda}_{2s}(t,x,y))$
that does not produce the same polynomial coefficients will never settle
down to the Kerr spacetime.
\end{proof}
The alternative scenarios include, among others, precession and nutation, 
as well as incomplete event horizon formation in spite of the fact that
all the initial data are axisymmetric. The only possible incomplete
event horizons would have to be those with axisymmetric `bold spots'.

\subsection{Beyond Monoparametric Solutions}
\paragraph*{}
Our solution to the equivalence problem of nondifferentiable
metrics is strictly local. On the manifold of general relativity,
some ground is gained with the assumption of asymptotic flatness.
The preframes are somewhat extended over a larger set in view of 
the representative set being constant in time. Still,  
our method only allows for a single black hole. To deal with 
several black holes, comoving or coalescing, a single global
blow-up parameter is insufficient. Instead, we would need to
utilize two (or more) local parameters. Hence new invariants based 
on relative positions may emerge. The Hill operator is not equipped 
to handle those. A natural generalization would seem to be the
two-dimensional Schr\H{o}dinger equation
$$[\frac{{\partial}^2}{\partial x_1^2} + \frac{{\partial}^2}
{\partial x_2^2} + Q(x_1,x_2)]f = \lambda f.$$ 
\noindent
For real $Q(x_1,x_2) \in L^2 ({\mathbb{R}}^2/{\mathbb{Z}^2})$, and
a fixed $\vec{v} \in {\mathbb{Z}}^2$, such that $f(x_1 + v_1, x_2) =
f(x_1, x_2 +v_2) = f(x_1,x_2)$, this poses a self-adjoint boundary
value problem with discrete spectrum (customarily denoted by
$E_i(\vec{v})$), interpreted as the energy levels in solid state physics.
Solutions are represented by a curve on the two-dimensional torus,
called the real Fermi curve. It is complexified by modifying the 
original operator to add a purely imaginary term:
$$[\frac{{\partial}^2}{\partial x_1^2} + \frac{{\partial}^2}
{\partial x_2^2} + i\vec{v} \cdot \nabla +{\vec{v}}^2 +Q(x_1,x_2)].$$
\noindent
For this operator, the Fermi curve is a one-dimensional complex analytic
variety in  ${\mathbb{C}}^2/{\mathbb{Z}^2}$. Its doubly periodic spectrum 
possesses the necessary complexity~\cite{K}, (\cite{FKT}, Section 16),
yet the Torelli theorem remains valid (\cite{FKT}, Section 18).\\
\indent 
However, that entails a radical redefinition of the time dimension. 
Due to the presence of two distinct event horizons, at least one of 
those moving, the local times in the vicinity of them are not related by
a Lorentz transformation. The difference is absolute and depends on
the space coordinates.

\end{document}